\documentclass[a4paper,11pt]{amsart}
\title{Surface homeomorphisms with big rotation set}
\date{\today}
\author{Pierre-Antoine Guih\'eneuf}
\address{Sorbonne Universit\'e, Universit\'e Paris Cit\'e, CNRS, IMJ-PRG, F-75005 Paris, France --- 
IRL Jean-Christophe Yoccoz CNRS / IMPA, Estr. Dona Castorina, 110
Jardim Bot\^anico, Rio de Janeiro, Brasil}
\email{pierre-antoine.guiheneuf@imj-prg.fr}
\thanks{The author thanks the Jean-Christophe Yoccoz international laboratory CNRS/IMPA for the semester in Brazil during which the ideas of this work were born.}

\usepackage[utf8]{inputenc}
\usepackage[english]{babel}
\usepackage[T1]{fontenc}  
\usepackage{amsfonts}
\usepackage{amsmath}
\usepackage{amssymb}
\usepackage{amsthm}
\usepackage{mathtools}

\usepackage{enumitem}
\setlist{noitemsep}
\usepackage{tikz}
\usetikzlibrary{arrows}
\PassOptionsToPackage{hyphens}{url}
\usepackage[pdftex,colorlinks=true,linkcolor=blue,citecolor=blue,urlcolor=blue]{hyperref}

\newtheorem{lemma}{Lemma}[section]
\newtheorem{theorem}[lemma]{Theorem}
\newtheorem{theo}{Theorem}

\newtheorem{propo}[theo]{Proposition}
\newtheorem{corol}[theo]{Corollary}
\newtheorem{prop}[lemma]{Proposition}
\newtheorem{question}[lemma]{Question}
\newtheorem{conj}[lemma]{Conjecture}
\newtheorem{coro}[lemma]{Corollary}

\newtheorem{fact}[lemma]{Fact}
\theoremstyle{definition}
\newtheorem{definition}[lemma]{Definition}
\theoremstyle{remark}
\newtheorem{rem}[lemma]{Remark}

\newcommand{\F}{\mathcal{F}}
\newcommand{\Hy}{\mathbf{H}}
\newcommand{\N}{\mathbf{N}}

\newcommand{\R}{\mathbf{R}}

\newcommand{\T}{\mathbf{T}}

\newcommand{\G}{\mathcal{G}}

\newcommand{\Q}{\mathbf{Q}}
\newcommand{\Z}{\mathbf{Z}}

\newcommand{\Homeo}{\operatorname{Homeo}}
\newcommand{\supp}{\operatorname{supp}}

\newcommand{\rot}{\operatorname{rot}}
\newcommand{\rote}{\operatorname{rot}_{\mathrm{erg}}}
\newcommand{\rotm}{\operatorname{rot}_{\mathrm{mes}}}
\newcommand{\card}{\operatorname{Card}}
\newcommand{\conv}{\operatorname{conv}}
\newcommand{\diam}{\operatorname{diam}}
\newcommand{\inte}{\operatorname{int}}

\newcommand{\Id}{\operatorname{Id}}
\newcommand{\dd}{\,\mathrm{d}}
\newcommand{\cl}{\mathcal{N}}
\newcommand{\wt}{\widetilde}

\newcommand{\pr}{\operatorname{pr}}
\newcommand{\Tr}{\mathcal{T}}
\newcommand{\spn}{\operatorname{span}}
\newcommand{\extr}{\operatorname{extr}}
\newcommand{\M}{\mathcal{M}}
\newcommand{\Me}{\mathcal{M}^{\mathrm{erg}}}
\newcommand{\Merg}{\mathcal{M}^{\mathrm{erg}}_{\vartheta>0}}

\begin{document}

\maketitle

\begin{abstract}
This article consists in applications of \cite{G25Cvx1} in the case of homemomorphisms of higher genus surfaces whose homological rotation set is big enough --- a class of dynamics that is open. 

We first prove a structure theorem for the rotation set of such homeomorphisms: it is a finite union of convex sets, we get an optimal bound for the number of such pieces. This bound can be improved in the case of transitive (in this case the rotation set is convex) and non-wandering dynamics (and for such homeomorphisms we get the existence of a family of invariant essential open sets).

We also get boundedness of deviations for homeomorphisms with big rotation set and some consequences of it, including a answer to Boyland's conjecture in our framework.
\end{abstract}

\tableofcontents

\section{Introduction}

The goal of this work is to describe in deep the rotational behaviour of a class of homeomorphisms of closed surfaces $S$ of genus $g\ge 2$ we call \emph{with big rotation set}. This class is defined as the set of $f\in\Homeo_0(S)$ such that $\mathrm{int}(\conv(\rot(f))) \neq\emptyset$ (or, equivalently, whose rotation set spans the whole homology of the surface), where $\rot(f)$ is the homological rotation set of $f$ (defined thereafter). By Proposition~\ref{PropContinuityRot}, this class is open in $\Homeo_0(S)$. 
It should be seen as a case study in the general goal of understanding the rotational dynamics of any element of $\Homeo_0(S)$. Some subsets of $\Homeo_0(S)$ have already been studied in detail (see Figure~\ref{FigSpaceHomeo}): 
\begin{itemize}
\item The set of $f\in\Homeo_0(S)$ such that $\mathrm{int}(\rote(f)) \neq\emptyset$ --- that is contained in the set of homeomorphisms with big rotation set and is also open in $\Homeo_0(S)$ (Proposition~\ref{PropContinuityRot}) --- was studied in \cite{guiheneuf2023hyperbolic}. Such homeomorphisms can be seen as ``Anosov from the rotational viewpoint'', in particular they are pseudo-Anosov relative to an $f$-invariant finite set (see also \cite{zbMATH05634807}).
\item The set of area-preserving $f\in\Homeo_0(S)$ with 0 entropy is studied in detail in the article in preparation \cite{GlCPT}, where a classification \emph{\`a la} Franks-Handel \cite{frankshandel,lct1} holds. Note that by \cite{lellouch}, the homological rotation set of a zero entropy homeomorphism is contained in a totally isotropic subspace of $H_1(S,\R)$ (for the intersection form $\wedge$ defined thereafter), hence in a subspace of dimension $\ge g$.
\item The set of \emph{fitted Axiom A} diffeomorphisms of $S$ \cite{MR4578317}; this study allows the authors to prove that on an open and dense subset of $\Homeo_0(S)$, the rotation set is the union of at most $2^{5g-3}$ convex sets containing 0.
\end{itemize}
Note that the dynamics of points that are typical for some measure was studied in \cite{alepablo}.

\begin{figure}[ht]
\begin{center}

\tikzset{every picture/.style={line width=0.75pt}} 

\begin{tikzpicture}[x=0.75pt,y=0.75pt,yscale=-1,xscale=1]

\draw  [color={rgb, 255:red, 0; green, 0; blue, 0 }  ,draw opacity=1 ][fill={rgb, 255:red, 213; green, 223; blue, 0 }  ,fill opacity=0.1 ][dash pattern={on 0.84pt off 2.51pt}] (253.9,70.16) .. controls (281.9,50.16) and (319.99,44.95) .. (350.94,64.22) .. controls (381.9,83.49) and (397.23,109.27) .. (398.28,148.22) .. controls (354.19,164.83) and (270.56,131.94) .. (219.94,147.22) .. controls (219.23,116.83) and (225.9,90.16) .. (253.9,70.16) -- cycle ;
\draw  [draw opacity=0][fill={rgb, 255:red, 0; green, 85; blue, 186 }  ,fill opacity=0.08 ] (257.54,224.17) .. controls (291.04,242.42) and (329.29,237.92) .. (363.79,221.17) .. controls (388.04,207.67) and (396.79,181.42) .. (398.28,148.22) .. controls (354.19,164.83) and (270.56,131.94) .. (219.94,147.22) .. controls (223.79,191.17) and (241.29,214.17) .. (257.54,224.17) -- cycle ;
\draw  [color={rgb, 255:red, 0; green, 0; blue, 0 }  ,draw opacity=1 ][fill={rgb, 255:red, 198; green, 154; blue, 0 }  ,fill opacity=0.2 ][dash pattern={on 0.84pt off 2.51pt}] (305.53,51.9) .. controls (335.06,49.9) and (380.77,74.47) .. (391.91,113.32) .. controls (351.06,105.9) and (242.49,133.61) .. (227.34,103.32) .. controls (241.63,68.18) and (276.01,53.9) .. (305.53,51.9) -- cycle ;
\draw  [color={rgb, 255:red, 0; green, 0; blue, 0 }  ,draw opacity=1 ][fill={rgb, 255:red, 198; green, 83; blue, 0 }  ,fill opacity=0.25 ][dash pattern={on 0.84pt off 2.51pt}] (305.53,51.9) .. controls (323.06,51.04) and (355.34,60.18) .. (376.77,86.18) .. controls (333.63,95.04) and (283.63,80.46) .. (246.49,76.18) .. controls (265.63,59.03) and (281.34,53.9) .. (305.53,51.9) -- cycle ;
\draw [color={rgb, 255:red, 96; green, 176; blue, 2 }  ,draw opacity=1 ]   (305.53,51.9) .. controls (289.82,106.75) and (316.96,199.32) .. (306.1,235.6) ;
\draw  [line width=1.5]  (238.54,83.83) .. controls (262.54,55.83) and (312.83,35.5) .. (361.04,70.83) .. controls (409.25,106.17) and (404.54,179.33) .. (378.54,209.33) .. controls (352.54,239.33) and (269.04,249.33) .. (240.54,208.83) .. controls (212.04,168.33) and (214.54,111.83) .. (238.54,83.83) -- cycle ;
\draw  [draw opacity=0][fill={rgb, 255:red, 0; green, 14; blue, 191 }  ,fill opacity=1 ] (289.23,210.85) .. controls (289.23,210.31) and (289.66,209.88) .. (290.21,209.88) .. controls (290.75,209.88) and (291.18,210.31) .. (291.18,210.85) .. controls (291.18,211.39) and (290.75,211.83) .. (290.21,211.83) .. controls (289.66,211.83) and (289.23,211.39) .. (289.23,210.85) -- cycle ;
\draw  [draw opacity=0][fill={rgb, 255:red, 0; green, 14; blue, 191 }  ,fill opacity=1 ] (307.56,202.77) .. controls (307.56,202.23) and (308,201.79) .. (308.54,201.79) .. controls (309.08,201.79) and (309.52,202.23) .. (309.52,202.77) .. controls (309.52,203.31) and (309.08,203.75) .. (308.54,203.75) .. controls (308,203.75) and (307.56,203.31) .. (307.56,202.77) -- cycle ;
\draw  [draw opacity=0][fill={rgb, 255:red, 0; green, 14; blue, 191 }  ,fill opacity=1 ] (289.64,200.35) .. controls (289.64,199.81) and (290.08,199.38) .. (290.62,199.38) .. controls (291.16,199.38) and (291.6,199.81) .. (291.6,200.35) .. controls (291.6,200.89) and (291.16,201.33) .. (290.62,201.33) .. controls (290.08,201.33) and (289.64,200.89) .. (289.64,200.35) -- cycle ;
\draw  [draw opacity=0][fill={rgb, 255:red, 0; green, 14; blue, 191 }  ,fill opacity=1 ] (286.14,216.6) .. controls (286.14,216.06) and (286.58,215.63) .. (287.12,215.63) .. controls (287.66,215.63) and (288.1,216.06) .. (288.1,216.6) .. controls (288.1,217.14) and (287.66,217.58) .. (287.12,217.58) .. controls (286.58,217.58) and (286.14,217.14) .. (286.14,216.6) -- cycle ;
\draw  [draw opacity=0][fill={rgb, 255:red, 0; green, 14; blue, 191 }  ,fill opacity=1 ] (297.73,219.19) .. controls (297.73,218.65) and (298.16,218.21) .. (298.71,218.21) .. controls (299.25,218.21) and (299.68,218.65) .. (299.68,219.19) .. controls (299.68,219.73) and (299.25,220.17) .. (298.71,220.17) .. controls (298.16,220.17) and (297.73,219.73) .. (297.73,219.19) -- cycle ;
\draw  [draw opacity=0][fill={rgb, 255:red, 0; green, 14; blue, 191 }  ,fill opacity=1 ] (307.73,217.19) .. controls (307.73,216.65) and (308.16,216.21) .. (308.71,216.21) .. controls (309.25,216.21) and (309.68,216.65) .. (309.68,217.19) .. controls (309.68,217.73) and (309.25,218.17) .. (308.71,218.17) .. controls (308.16,218.17) and (307.73,217.73) .. (307.73,217.19) -- cycle ;
\draw  [draw opacity=0][fill={rgb, 255:red, 0; green, 14; blue, 191 }  ,fill opacity=1 ] (293.14,219.02) .. controls (293.14,218.48) and (293.58,218.04) .. (294.12,218.04) .. controls (294.66,218.04) and (295.1,218.48) .. (295.1,219.02) .. controls (295.1,219.56) and (294.66,220) .. (294.12,220) .. controls (293.58,220) and (293.14,219.56) .. (293.14,219.02) -- cycle ;
\draw  [draw opacity=0][fill={rgb, 255:red, 0; green, 14; blue, 191 }  ,fill opacity=1 ] (296.73,213.1) .. controls (296.73,212.56) and (297.16,212.13) .. (297.71,212.13) .. controls (298.25,212.13) and (298.68,212.56) .. (298.68,213.1) .. controls (298.68,213.64) and (298.25,214.08) .. (297.71,214.08) .. controls (297.16,214.08) and (296.73,213.64) .. (296.73,213.1) -- cycle ;
\draw  [draw opacity=0][fill={rgb, 255:red, 0; green, 14; blue, 191 }  ,fill opacity=1 ] (307.73,208.77) .. controls (307.73,208.23) and (308.16,207.79) .. (308.71,207.79) .. controls (309.25,207.79) and (309.68,208.23) .. (309.68,208.77) .. controls (309.68,209.31) and (309.25,209.75) .. (308.71,209.75) .. controls (308.16,209.75) and (307.73,209.31) .. (307.73,208.77) -- cycle ;
\draw  [draw opacity=0][fill={rgb, 255:red, 0; green, 14; blue, 191 }  ,fill opacity=1 ] (295.73,205.85) .. controls (295.73,205.31) and (296.16,204.88) .. (296.71,204.88) .. controls (297.25,204.88) and (297.68,205.31) .. (297.68,205.85) .. controls (297.68,206.39) and (297.25,206.83) .. (296.71,206.83) .. controls (296.16,206.83) and (295.73,206.39) .. (295.73,205.85) -- cycle ;
\draw  [draw opacity=0][fill={rgb, 255:red, 0; green, 14; blue, 191 }  ,fill opacity=1 ] (302.31,229.44) .. controls (302.31,228.9) and (302.75,228.46) .. (303.29,228.46) .. controls (303.83,228.46) and (304.27,228.9) .. (304.27,229.44) .. controls (304.27,229.98) and (303.83,230.42) .. (303.29,230.42) .. controls (302.75,230.42) and (302.31,229.98) .. (302.31,229.44) -- cycle ;
\draw  [draw opacity=0][fill={rgb, 255:red, 0; green, 14; blue, 191 }  ,fill opacity=1 ] (306.14,231.94) .. controls (306.14,231.4) and (306.58,230.96) .. (307.12,230.96) .. controls (307.66,230.96) and (308.1,231.4) .. (308.1,231.94) .. controls (308.1,232.48) and (307.66,232.92) .. (307.12,232.92) .. controls (306.58,232.92) and (306.14,232.48) .. (306.14,231.94) -- cycle ;
\draw  [draw opacity=0][fill={rgb, 255:red, 0; green, 14; blue, 191 }  ,fill opacity=1 ] (319.23,229.52) .. controls (319.23,228.98) and (319.66,228.54) .. (320.21,228.54) .. controls (320.75,228.54) and (321.18,228.98) .. (321.18,229.52) .. controls (321.18,230.06) and (320.75,230.5) .. (320.21,230.5) .. controls (319.66,230.5) and (319.23,230.06) .. (319.23,229.52) -- cycle ;
\draw  [draw opacity=0][fill={rgb, 255:red, 0; green, 14; blue, 191 }  ,fill opacity=1 ] (300.23,224.85) .. controls (300.23,224.31) and (300.66,223.88) .. (301.21,223.88) .. controls (301.75,223.88) and (302.18,224.31) .. (302.18,224.85) .. controls (302.18,225.39) and (301.75,225.83) .. (301.21,225.83) .. controls (300.66,225.83) and (300.23,225.39) .. (300.23,224.85) -- cycle ;
\draw  [draw opacity=0][fill={rgb, 255:red, 0; green, 14; blue, 191 }  ,fill opacity=1 ] (297.56,231.02) .. controls (297.56,230.48) and (298,230.04) .. (298.54,230.04) .. controls (299.08,230.04) and (299.52,230.48) .. (299.52,231.02) .. controls (299.52,231.56) and (299.08,232) .. (298.54,232) .. controls (298,232) and (297.56,231.56) .. (297.56,231.02) -- cycle ;
\draw  [draw opacity=0][fill={rgb, 255:red, 0; green, 14; blue, 191 }  ,fill opacity=1 ] (301.31,209.35) .. controls (301.31,208.81) and (301.75,208.38) .. (302.29,208.38) .. controls (302.83,208.38) and (303.27,208.81) .. (303.27,209.35) .. controls (303.27,209.89) and (302.83,210.33) .. (302.29,210.33) .. controls (301.75,210.33) and (301.31,209.89) .. (301.31,209.35) -- cycle ;
\draw  [draw opacity=0][fill={rgb, 255:red, 0; green, 14; blue, 191 }  ,fill opacity=1 ] (317.39,213.44) .. controls (317.39,212.9) and (317.83,212.46) .. (318.37,212.46) .. controls (318.91,212.46) and (319.35,212.9) .. (319.35,213.44) .. controls (319.35,213.98) and (318.91,214.42) .. (318.37,214.42) .. controls (317.83,214.42) and (317.39,213.98) .. (317.39,213.44) -- cycle ;
\draw  [draw opacity=0][fill={rgb, 255:red, 0; green, 14; blue, 191 }  ,fill opacity=1 ] (322.81,214.94) .. controls (322.81,214.4) and (323.25,213.96) .. (323.79,213.96) .. controls (324.33,213.96) and (324.77,214.4) .. (324.77,214.94) .. controls (324.77,215.48) and (324.33,215.92) .. (323.79,215.92) .. controls (323.25,215.92) and (322.81,215.48) .. (322.81,214.94) -- cycle ;
\draw  [draw opacity=0][fill={rgb, 255:red, 0; green, 14; blue, 191 }  ,fill opacity=1 ] (290.56,215.1) .. controls (290.56,214.56) and (291,214.13) .. (291.54,214.13) .. controls (292.08,214.13) and (292.52,214.56) .. (292.52,215.1) .. controls (292.52,215.64) and (292.08,216.08) .. (291.54,216.08) .. controls (291,216.08) and (290.56,215.64) .. (290.56,215.1) -- cycle ;
\draw  [draw opacity=0][fill={rgb, 255:red, 0; green, 14; blue, 191 }  ,fill opacity=1 ] (307.64,211.6) .. controls (307.64,211.06) and (308.08,210.63) .. (308.62,210.63) .. controls (309.16,210.63) and (309.6,211.06) .. (309.6,211.6) .. controls (309.6,212.14) and (309.16,212.58) .. (308.62,212.58) .. controls (308.08,212.58) and (307.64,212.14) .. (307.64,211.6) -- cycle ;
\draw  [draw opacity=0][fill={rgb, 255:red, 0; green, 14; blue, 191 }  ,fill opacity=1 ] (293.64,225.52) .. controls (293.64,224.98) and (294.08,224.54) .. (294.62,224.54) .. controls (295.16,224.54) and (295.6,224.98) .. (295.6,225.52) .. controls (295.6,226.06) and (295.16,226.5) .. (294.62,226.5) .. controls (294.08,226.5) and (293.64,226.06) .. (293.64,225.52) -- cycle ;
\draw  [draw opacity=0][fill={rgb, 255:red, 0; green, 14; blue, 191 }  ,fill opacity=1 ] (289.14,225.35) .. controls (289.14,224.81) and (289.58,224.38) .. (290.12,224.38) .. controls (290.66,224.38) and (291.1,224.81) .. (291.1,225.35) .. controls (291.1,225.89) and (290.66,226.33) .. (290.12,226.33) .. controls (289.58,226.33) and (289.14,225.89) .. (289.14,225.35) -- cycle ;
\draw  [draw opacity=0][fill={rgb, 255:red, 0; green, 14; blue, 191 }  ,fill opacity=1 ] (314.89,224.77) .. controls (314.89,224.23) and (315.33,223.79) .. (315.87,223.79) .. controls (316.41,223.79) and (316.85,224.23) .. (316.85,224.77) .. controls (316.85,225.31) and (316.41,225.75) .. (315.87,225.75) .. controls (315.33,225.75) and (314.89,225.31) .. (314.89,224.77) -- cycle ;
\draw  [draw opacity=0][fill={rgb, 255:red, 0; green, 14; blue, 191 }  ,fill opacity=1 ] (314.23,218.69) .. controls (314.23,218.15) and (314.66,217.71) .. (315.21,217.71) .. controls (315.75,217.71) and (316.18,218.15) .. (316.18,218.69) .. controls (316.18,219.23) and (315.75,219.67) .. (315.21,219.67) .. controls (314.66,219.67) and (314.23,219.23) .. (314.23,218.69) -- cycle ;
\draw  [draw opacity=0][fill={rgb, 255:red, 0; green, 14; blue, 191 }  ,fill opacity=1 ] (310.89,230.02) .. controls (310.89,229.48) and (311.33,229.04) .. (311.87,229.04) .. controls (312.41,229.04) and (312.85,229.48) .. (312.85,230.02) .. controls (312.85,230.56) and (312.41,231) .. (311.87,231) .. controls (311.33,231) and (310.89,230.56) .. (310.89,230.02) -- cycle ;
\draw  [draw opacity=0][fill={rgb, 255:red, 0; green, 14; blue, 191 }  ,fill opacity=1 ] (320.06,224.19) .. controls (320.06,223.65) and (320.5,223.21) .. (321.04,223.21) .. controls (321.58,223.21) and (322.02,223.65) .. (322.02,224.19) .. controls (322.02,224.73) and (321.58,225.17) .. (321.04,225.17) .. controls (320.5,225.17) and (320.06,224.73) .. (320.06,224.19) -- cycle ;
\draw  [draw opacity=0][fill={rgb, 255:red, 0; green, 14; blue, 191 }  ,fill opacity=1 ] (317.81,206.1) .. controls (317.81,205.56) and (318.25,205.13) .. (318.79,205.13) .. controls (319.33,205.13) and (319.77,205.56) .. (319.77,206.1) .. controls (319.77,206.64) and (319.33,207.08) .. (318.79,207.08) .. controls (318.25,207.08) and (317.81,206.64) .. (317.81,206.1) -- cycle ;
\draw  [draw opacity=0][fill={rgb, 255:red, 0; green, 14; blue, 191 }  ,fill opacity=1 ] (326.31,203.94) .. controls (326.31,203.4) and (326.75,202.96) .. (327.29,202.96) .. controls (327.83,202.96) and (328.27,203.4) .. (328.27,203.94) .. controls (328.27,204.48) and (327.83,204.92) .. (327.29,204.92) .. controls (326.75,204.92) and (326.31,204.48) .. (326.31,203.94) -- cycle ;
\draw  [draw opacity=0][fill={rgb, 255:red, 0; green, 14; blue, 191 }  ,fill opacity=1 ] (306.89,192.6) .. controls (306.89,192.06) and (307.33,191.63) .. (307.87,191.63) .. controls (308.41,191.63) and (308.85,192.06) .. (308.85,192.6) .. controls (308.85,193.14) and (308.41,193.58) .. (307.87,193.58) .. controls (307.33,193.58) and (306.89,193.14) .. (306.89,192.6) -- cycle ;
\draw  [draw opacity=0][fill={rgb, 255:red, 0; green, 14; blue, 191 }  ,fill opacity=1 ] (299.14,191.52) .. controls (299.14,190.98) and (299.58,190.54) .. (300.12,190.54) .. controls (300.66,190.54) and (301.1,190.98) .. (301.1,191.52) .. controls (301.1,192.06) and (300.66,192.5) .. (300.12,192.5) .. controls (299.58,192.5) and (299.14,192.06) .. (299.14,191.52) -- cycle ;
\draw  [draw opacity=0][fill={rgb, 255:red, 0; green, 14; blue, 191 }  ,fill opacity=1 ] (279.81,223.02) .. controls (279.81,222.48) and (280.25,222.04) .. (280.79,222.04) .. controls (281.33,222.04) and (281.77,222.48) .. (281.77,223.02) .. controls (281.77,223.56) and (281.33,224) .. (280.79,224) .. controls (280.25,224) and (279.81,223.56) .. (279.81,223.02) -- cycle ;
\draw  [draw opacity=0][fill={rgb, 255:red, 0; green, 14; blue, 191 }  ,fill opacity=1 ] (302.39,217.69) .. controls (302.39,217.15) and (302.83,216.71) .. (303.37,216.71) .. controls (303.91,216.71) and (304.35,217.15) .. (304.35,217.69) .. controls (304.35,218.23) and (303.91,218.67) .. (303.37,218.67) .. controls (302.83,218.67) and (302.39,218.23) .. (302.39,217.69) -- cycle ;
\draw  [draw opacity=0][fill={rgb, 255:red, 0; green, 14; blue, 191 }  ,fill opacity=1 ] (303.98,223.77) .. controls (303.98,223.23) and (304.41,222.79) .. (304.96,222.79) .. controls (305.5,222.79) and (305.93,223.23) .. (305.93,223.77) .. controls (305.93,224.31) and (305.5,224.75) .. (304.96,224.75) .. controls (304.41,224.75) and (303.98,224.31) .. (303.98,223.77) -- cycle ;
\draw  [draw opacity=0][fill={rgb, 255:red, 0; green, 14; blue, 191 }  ,fill opacity=1 ] (284.81,230.19) .. controls (284.81,229.65) and (285.25,229.21) .. (285.79,229.21) .. controls (286.33,229.21) and (286.77,229.65) .. (286.77,230.19) .. controls (286.77,230.73) and (286.33,231.17) .. (285.79,231.17) .. controls (285.25,231.17) and (284.81,230.73) .. (284.81,230.19) -- cycle ;
\draw  [draw opacity=0][fill={rgb, 255:red, 0; green, 14; blue, 191 }  ,fill opacity=1 ] (315.98,232.19) .. controls (315.98,231.65) and (316.41,231.21) .. (316.96,231.21) .. controls (317.5,231.21) and (317.93,231.65) .. (317.93,232.19) .. controls (317.93,232.73) and (317.5,233.17) .. (316.96,233.17) .. controls (316.41,233.17) and (315.98,232.73) .. (315.98,232.19) -- cycle ;
\draw  [draw opacity=0][fill={rgb, 255:red, 0; green, 14; blue, 191 }  ,fill opacity=1 ] (309.78,221.9) .. controls (309.78,221.36) and (310.21,220.93) .. (310.76,220.93) .. controls (311.3,220.93) and (311.73,221.36) .. (311.73,221.9) .. controls (311.73,222.44) and (311.3,222.88) .. (310.76,222.88) .. controls (310.21,222.88) and (309.78,222.44) .. (309.78,221.9) -- cycle ;
\draw  [draw opacity=0][fill={rgb, 255:red, 0; green, 14; blue, 191 }  ,fill opacity=1 ] (304.81,226.77) .. controls (304.81,226.23) and (305.25,225.79) .. (305.79,225.79) .. controls (306.33,225.79) and (306.77,226.23) .. (306.77,226.77) .. controls (306.77,227.31) and (306.33,227.75) .. (305.79,227.75) .. controls (305.25,227.75) and (304.81,227.31) .. (304.81,226.77) -- cycle ;
\draw  [draw opacity=0][fill={rgb, 255:red, 0; green, 14; blue, 191 }  ,fill opacity=1 ] (304.56,221.6) .. controls (304.56,221.06) and (305,220.63) .. (305.54,220.63) .. controls (306.08,220.63) and (306.52,221.06) .. (306.52,221.6) .. controls (306.52,222.14) and (306.08,222.58) .. (305.54,222.58) .. controls (305,222.58) and (304.56,222.14) .. (304.56,221.6) -- cycle ;
\draw  [draw opacity=0][fill={rgb, 255:red, 0; green, 14; blue, 191 }  ,fill opacity=1 ] (301.06,220.69) .. controls (301.06,220.15) and (301.5,219.71) .. (302.04,219.71) .. controls (302.58,219.71) and (303.02,220.15) .. (303.02,220.69) .. controls (303.02,221.23) and (302.58,221.67) .. (302.04,221.67) .. controls (301.5,221.67) and (301.06,221.23) .. (301.06,220.69) -- cycle ;
\draw  [draw opacity=0][fill={rgb, 255:red, 0; green, 14; blue, 191 }  ,fill opacity=1 ] (312.14,220.19) .. controls (312.14,219.65) and (312.58,219.21) .. (313.12,219.21) .. controls (313.66,219.21) and (314.1,219.65) .. (314.1,220.19) .. controls (314.1,220.73) and (313.66,221.17) .. (313.12,221.17) .. controls (312.58,221.17) and (312.14,220.73) .. (312.14,220.19) -- cycle ;
\draw  [draw opacity=0][fill={rgb, 255:red, 0; green, 14; blue, 191 }  ,fill opacity=1 ] (313.73,223.35) .. controls (313.73,222.81) and (314.16,222.38) .. (314.71,222.38) .. controls (315.25,222.38) and (315.68,222.81) .. (315.68,223.35) .. controls (315.68,223.89) and (315.25,224.33) .. (314.71,224.33) .. controls (314.16,224.33) and (313.73,223.89) .. (313.73,223.35) -- cycle ;
\draw  [draw opacity=0][fill={rgb, 255:red, 0; green, 14; blue, 191 }  ,fill opacity=1 ] (301.89,227.02) .. controls (301.89,226.48) and (302.33,226.04) .. (302.87,226.04) .. controls (303.41,226.04) and (303.85,226.48) .. (303.85,227.02) .. controls (303.85,227.56) and (303.41,228) .. (302.87,228) .. controls (302.33,228) and (301.89,227.56) .. (301.89,227.02) -- cycle ;
\draw  [draw opacity=0][fill={rgb, 255:red, 0; green, 14; blue, 191 }  ,fill opacity=1 ] (313.81,227.27) .. controls (313.81,226.73) and (314.25,226.29) .. (314.79,226.29) .. controls (315.33,226.29) and (315.77,226.73) .. (315.77,227.27) .. controls (315.77,227.81) and (315.33,228.25) .. (314.79,228.25) .. controls (314.25,228.25) and (313.81,227.81) .. (313.81,227.27) -- cycle ;
\draw  [draw opacity=0][fill={rgb, 255:red, 0; green, 14; blue, 191 }  ,fill opacity=1 ] (312.64,205.85) .. controls (312.64,205.31) and (313.08,204.88) .. (313.62,204.88) .. controls (314.16,204.88) and (314.6,205.31) .. (314.6,205.85) .. controls (314.6,206.39) and (314.16,206.83) .. (313.62,206.83) .. controls (313.08,206.83) and (312.64,206.39) .. (312.64,205.85) -- cycle ;
\draw  [draw opacity=0][fill={rgb, 255:red, 0; green, 14; blue, 191 }  ,fill opacity=1 ] (312.64,214.27) .. controls (312.64,213.73) and (313.08,213.29) .. (313.62,213.29) .. controls (314.16,213.29) and (314.6,213.73) .. (314.6,214.27) .. controls (314.6,214.81) and (314.16,215.25) .. (313.62,215.25) .. controls (313.08,215.25) and (312.64,214.81) .. (312.64,214.27) -- cycle ;
\draw  [draw opacity=0][fill={rgb, 255:red, 0; green, 14; blue, 191 }  ,fill opacity=1 ] (309.01,226.62) .. controls (309.01,226.08) and (309.45,225.64) .. (309.99,225.64) .. controls (310.53,225.64) and (310.97,226.08) .. (310.97,226.62) .. controls (310.97,227.16) and (310.53,227.6) .. (309.99,227.6) .. controls (309.45,227.6) and (309.01,227.16) .. (309.01,226.62) -- cycle ;
\draw  [draw opacity=0][fill={rgb, 255:red, 0; green, 14; blue, 191 }  ,fill opacity=1 ] (295.41,160.52) .. controls (295.41,159.98) and (295.85,159.54) .. (296.39,159.54) .. controls (296.93,159.54) and (297.37,159.98) .. (297.37,160.52) .. controls (297.37,161.06) and (296.93,161.5) .. (296.39,161.5) .. controls (295.85,161.5) and (295.41,161.06) .. (295.41,160.52) -- cycle ;
\draw  [draw opacity=0][fill={rgb, 255:red, 0; green, 14; blue, 191 }  ,fill opacity=1 ] (309.68,164.65) .. controls (309.68,164.11) and (310.11,163.68) .. (310.66,163.68) .. controls (311.2,163.68) and (311.63,164.11) .. (311.63,164.65) .. controls (311.63,165.19) and (311.2,165.63) .. (310.66,165.63) .. controls (310.11,165.63) and (309.68,165.19) .. (309.68,164.65) -- cycle ;
\draw  [draw opacity=0][fill={rgb, 255:red, 0; green, 14; blue, 191 }  ,fill opacity=1 ] (311.41,182.12) .. controls (311.41,181.58) and (311.85,181.14) .. (312.39,181.14) .. controls (312.93,181.14) and (313.37,181.58) .. (313.37,182.12) .. controls (313.37,182.66) and (312.93,183.1) .. (312.39,183.1) .. controls (311.85,183.1) and (311.41,182.66) .. (311.41,182.12) -- cycle ;
\draw  [draw opacity=0][fill={rgb, 255:red, 0; green, 14; blue, 191 }  ,fill opacity=1 ] (297.94,199.05) .. controls (297.94,198.51) and (298.38,198.08) .. (298.92,198.08) .. controls (299.46,198.08) and (299.9,198.51) .. (299.9,199.05) .. controls (299.9,199.59) and (299.46,200.03) .. (298.92,200.03) .. controls (298.38,200.03) and (297.94,199.59) .. (297.94,199.05) -- cycle ;
\draw  [draw opacity=0][fill={rgb, 255:red, 0; green, 14; blue, 191 }  ,fill opacity=1 ] (315.14,190.92) .. controls (315.14,190.38) and (315.58,189.94) .. (316.12,189.94) .. controls (316.66,189.94) and (317.1,190.38) .. (317.1,190.92) .. controls (317.1,191.46) and (316.66,191.9) .. (316.12,191.9) .. controls (315.58,191.9) and (315.14,191.46) .. (315.14,190.92) -- cycle ;
\draw  [draw opacity=0][fill={rgb, 255:red, 0; green, 14; blue, 191 }  ,fill opacity=1 ] (306.34,181.72) .. controls (306.34,181.18) and (306.78,180.74) .. (307.32,180.74) .. controls (307.86,180.74) and (308.3,181.18) .. (308.3,181.72) .. controls (308.3,182.26) and (307.86,182.7) .. (307.32,182.7) .. controls (306.78,182.7) and (306.34,182.26) .. (306.34,181.72) -- cycle ;
\draw  [draw opacity=0][fill={rgb, 255:red, 0; green, 14; blue, 191 }  ,fill opacity=1 ] (314.08,197.72) .. controls (314.08,197.18) and (314.51,196.74) .. (315.06,196.74) .. controls (315.6,196.74) and (316.03,197.18) .. (316.03,197.72) .. controls (316.03,198.26) and (315.6,198.7) .. (315.06,198.7) .. controls (314.51,198.7) and (314.08,198.26) .. (314.08,197.72) -- cycle ;
\draw  [draw opacity=0][fill={rgb, 255:red, 0; green, 14; blue, 191 }  ,fill opacity=1 ] (304.21,162.92) .. controls (304.21,162.38) and (304.65,161.94) .. (305.19,161.94) .. controls (305.73,161.94) and (306.17,162.38) .. (306.17,162.92) .. controls (306.17,163.46) and (305.73,163.9) .. (305.19,163.9) .. controls (304.65,163.9) and (304.21,163.46) .. (304.21,162.92) -- cycle ;
\draw  [draw opacity=0][fill={rgb, 255:red, 0; green, 14; blue, 191 }  ,fill opacity=1 ] (287.28,186.25) .. controls (287.28,185.71) and (287.71,185.28) .. (288.26,185.28) .. controls (288.8,185.28) and (289.23,185.71) .. (289.23,186.25) .. controls (289.23,186.79) and (288.8,187.23) .. (288.26,187.23) .. controls (287.71,187.23) and (287.28,186.79) .. (287.28,186.25) -- cycle ;
\draw  [draw opacity=0][fill={rgb, 255:red, 0; green, 14; blue, 191 }  ,fill opacity=1 ] (301.06,222.65) .. controls (301.06,222.11) and (301.5,221.67) .. (302.04,221.67) .. controls (302.58,221.67) and (303.02,222.11) .. (303.02,222.65) .. controls (303.02,223.19) and (302.58,223.62) .. (302.04,223.62) .. controls (301.5,223.62) and (301.06,223.19) .. (301.06,222.65) -- cycle ;
\draw [color={rgb, 255:red, 0; green, 85; blue, 186 }  ,draw opacity=1 ]   (213.18,205.79) .. controls (229.88,219.8) and (243.73,215.49) .. (255.72,202.03) ;
\draw [shift={(257.6,199.83)}, rotate = 129.34] [fill={rgb, 255:red, 0; green, 85; blue, 186 }  ,fill opacity=1 ][line width=0.08]  [draw opacity=0] (8.04,-3.86) -- (0,0) -- (8.04,3.86) -- (5.34,0) -- cycle    ;
\draw  [draw opacity=0][fill={rgb, 255:red, 0; green, 14; blue, 191 }  ,fill opacity=1 ] (331.41,194.85) .. controls (331.41,194.31) and (331.85,193.88) .. (332.39,193.88) .. controls (332.93,193.88) and (333.37,194.31) .. (333.37,194.85) .. controls (333.37,195.39) and (332.93,195.83) .. (332.39,195.83) .. controls (331.85,195.83) and (331.41,195.39) .. (331.41,194.85) -- cycle ;
\draw  [draw opacity=0][fill={rgb, 255:red, 0; green, 14; blue, 191 }  ,fill opacity=1 ] (352.12,166.57) .. controls (352.12,166.03) and (352.56,165.59) .. (353.1,165.59) .. controls (353.64,165.59) and (354.08,166.03) .. (354.08,166.57) .. controls (354.08,167.11) and (353.64,167.55) .. (353.1,167.55) .. controls (352.56,167.55) and (352.12,167.11) .. (352.12,166.57) -- cycle ;
\draw  [draw opacity=0][fill={rgb, 255:red, 0; green, 14; blue, 191 }  ,fill opacity=1 ] (347.64,218.86) .. controls (347.64,218.32) and (348.08,217.88) .. (348.62,217.88) .. controls (349.16,217.88) and (349.6,218.32) .. (349.6,218.86) .. controls (349.6,219.4) and (349.16,219.84) .. (348.62,219.84) .. controls (348.08,219.84) and (347.64,219.4) .. (347.64,218.86) -- cycle ;
\draw  [draw opacity=0][fill={rgb, 255:red, 0; green, 14; blue, 191 }  ,fill opacity=1 ] (347.76,197.69) .. controls (347.76,197.15) and (348.2,196.71) .. (348.74,196.71) .. controls (349.28,196.71) and (349.72,197.15) .. (349.72,197.69) .. controls (349.72,198.23) and (349.28,198.67) .. (348.74,198.67) .. controls (348.2,198.67) and (347.76,198.23) .. (347.76,197.69) -- cycle ;
\draw  [draw opacity=0][fill={rgb, 255:red, 0; green, 14; blue, 191 }  ,fill opacity=1 ] (356,197.1) .. controls (356,196.56) and (356.44,196.12) .. (356.98,196.12) .. controls (357.52,196.12) and (357.96,196.56) .. (357.96,197.1) .. controls (357.96,197.64) and (357.52,198.08) .. (356.98,198.08) .. controls (356.44,198.08) and (356,197.64) .. (356,197.1) -- cycle ;
\draw  [draw opacity=0][fill={rgb, 255:red, 0; green, 14; blue, 191 }  ,fill opacity=1 ] (322.94,182.75) .. controls (322.94,182.21) and (323.38,181.77) .. (323.92,181.77) .. controls (324.46,181.77) and (324.9,182.21) .. (324.9,182.75) .. controls (324.9,183.29) and (324.46,183.73) .. (323.92,183.73) .. controls (323.38,183.73) and (322.94,183.29) .. (322.94,182.75) -- cycle ;
\draw  [draw opacity=0][fill={rgb, 255:red, 0; green, 14; blue, 191 }  ,fill opacity=1 ] (372.23,167.57) .. controls (372.23,167.03) and (372.67,166.59) .. (373.21,166.59) .. controls (373.75,166.59) and (374.19,167.03) .. (374.19,167.57) .. controls (374.19,168.11) and (373.75,168.55) .. (373.21,168.55) .. controls (372.67,168.55) and (372.23,168.11) .. (372.23,167.57) -- cycle ;
\draw  [draw opacity=0][fill={rgb, 255:red, 0; green, 14; blue, 191 }  ,fill opacity=1 ] (329.64,172.16) .. controls (329.64,171.62) and (330.08,171.18) .. (330.62,171.18) .. controls (331.16,171.18) and (331.6,171.62) .. (331.6,172.16) .. controls (331.6,172.7) and (331.16,173.14) .. (330.62,173.14) .. controls (330.08,173.14) and (329.64,172.7) .. (329.64,172.16) -- cycle ;
\draw  [draw opacity=0][fill={rgb, 255:red, 0; green, 14; blue, 191 }  ,fill opacity=1 ] (336,219.65) .. controls (336,219.11) and (336.44,218.67) .. (336.98,218.67) .. controls (337.52,218.67) and (337.96,219.11) .. (337.96,219.65) .. controls (337.96,220.19) and (337.52,220.63) .. (336.98,220.63) .. controls (336.44,220.63) and (336,220.19) .. (336,219.65) -- cycle ;
\draw  [draw opacity=0][fill={rgb, 255:red, 0; green, 14; blue, 191 }  ,fill opacity=1 ] (359.76,209.53) .. controls (359.76,208.99) and (360.2,208.55) .. (360.74,208.55) .. controls (361.28,208.55) and (361.72,208.99) .. (361.72,209.53) .. controls (361.72,210.07) and (361.28,210.51) .. (360.74,210.51) .. controls (360.2,210.51) and (359.76,210.07) .. (359.76,209.53) -- cycle ;
\draw  [draw opacity=0][fill={rgb, 255:red, 0; green, 14; blue, 191 }  ,fill opacity=1 ] (298.68,203.65) .. controls (298.68,203.11) and (299.12,202.67) .. (299.66,202.67) .. controls (300.2,202.67) and (300.64,203.11) .. (300.64,203.65) .. controls (300.64,204.19) and (300.2,204.63) .. (299.66,204.63) .. controls (299.12,204.63) and (298.68,204.19) .. (298.68,203.65) -- cycle ;
\draw  [draw opacity=0][fill={rgb, 255:red, 0; green, 14; blue, 191 }  ,fill opacity=1 ] (235.62,168.24) .. controls (235.62,167.7) and (236.06,167.26) .. (236.6,167.26) .. controls (237.14,167.26) and (237.58,167.7) .. (237.58,168.24) .. controls (237.58,168.78) and (237.14,169.22) .. (236.6,169.22) .. controls (236.06,169.22) and (235.62,168.78) .. (235.62,168.24) -- cycle ;
\draw  [draw opacity=0][fill={rgb, 255:red, 0; green, 14; blue, 191 }  ,fill opacity=1 ] (257.27,157.65) .. controls (257.27,157.11) and (257.71,156.67) .. (258.25,156.67) .. controls (258.79,156.67) and (259.23,157.11) .. (259.23,157.65) .. controls (259.23,158.19) and (258.79,158.63) .. (258.25,158.63) .. controls (257.71,158.63) and (257.27,158.19) .. (257.27,157.65) -- cycle ;
\draw  [draw opacity=0][fill={rgb, 255:red, 0; green, 14; blue, 191 }  ,fill opacity=1 ] (269.62,205.06) .. controls (269.62,204.52) and (270.06,204.08) .. (270.6,204.08) .. controls (271.14,204.08) and (271.58,204.52) .. (271.58,205.06) .. controls (271.58,205.6) and (271.14,206.04) .. (270.6,206.04) .. controls (270.06,206.04) and (269.62,205.6) .. (269.62,205.06) -- cycle ;
\draw  [draw opacity=0][fill={rgb, 255:red, 0; green, 14; blue, 191 }  ,fill opacity=1 ] (241.5,198) .. controls (241.5,197.46) and (241.94,197.02) .. (242.48,197.02) .. controls (243.02,197.02) and (243.46,197.46) .. (243.46,198) .. controls (243.46,198.54) and (243.02,198.98) .. (242.48,198.98) .. controls (241.94,198.98) and (241.5,198.54) .. (241.5,198) -- cycle ;
\draw  [draw opacity=0][fill={rgb, 255:red, 0; green, 14; blue, 191 }  ,fill opacity=1 ] (261.62,217.53) .. controls (261.62,216.99) and (262.06,216.55) .. (262.6,216.55) .. controls (263.14,216.55) and (263.58,216.99) .. (263.58,217.53) .. controls (263.58,218.07) and (263.14,218.51) .. (262.6,218.51) .. controls (262.06,218.51) and (261.62,218.07) .. (261.62,217.53) -- cycle ;
\draw  [draw opacity=0][fill={rgb, 255:red, 0; green, 14; blue, 191 }  ,fill opacity=1 ] (266.33,190.24) .. controls (266.33,189.7) and (266.76,189.26) .. (267.31,189.26) .. controls (267.85,189.26) and (268.28,189.7) .. (268.28,190.24) .. controls (268.28,190.78) and (267.85,191.22) .. (267.31,191.22) .. controls (266.76,191.22) and (266.33,190.78) .. (266.33,190.24) -- cycle ;
\draw  [draw opacity=0][fill={rgb, 255:red, 0; green, 14; blue, 191 }  ,fill opacity=1 ] (242.8,188.24) .. controls (242.8,187.7) and (243.23,187.26) .. (243.78,187.26) .. controls (244.32,187.26) and (244.75,187.7) .. (244.75,188.24) .. controls (244.75,188.78) and (244.32,189.22) .. (243.78,189.22) .. controls (243.23,189.22) and (242.8,188.78) .. (242.8,188.24) -- cycle ;
\draw  [draw opacity=0][fill={rgb, 255:red, 0; green, 14; blue, 191 }  ,fill opacity=1 ] (232.91,153.89) .. controls (232.91,153.34) and (233.35,152.91) .. (233.89,152.91) .. controls (234.43,152.91) and (234.87,153.34) .. (234.87,153.89) .. controls (234.87,154.43) and (234.43,154.86) .. (233.89,154.86) .. controls (233.35,154.86) and (232.91,154.43) .. (232.91,153.89) -- cycle ;
\draw [color={rgb, 255:red, 154; green, 161; blue, 1 }  ,draw opacity=1 ]   (421.36,126.44) .. controls (408.06,121.38) and (398.07,120.52) .. (384.73,131.36) ;
\draw [shift={(382.6,133.17)}, rotate = 318.59] [fill={rgb, 255:red, 154; green, 161; blue, 1 }  ,fill opacity=1 ][line width=0.08]  [draw opacity=0] (8.04,-3.86) -- (0,0) -- (8.04,3.86) -- (5.34,0) -- cycle    ;
\draw [color={rgb, 255:red, 160; green, 125; blue, 0 }  ,draw opacity=1 ]   (221.36,85.13) .. controls (234.27,76) and (248.97,78.77) .. (259.76,91.23) ;
\draw [shift={(261.6,93.5)}, rotate = 232.78] [fill={rgb, 255:red, 160; green, 125; blue, 0 }  ,fill opacity=1 ][line width=0.08]  [draw opacity=0] (8.04,-3.86) -- (0,0) -- (8.04,3.86) -- (5.34,0) -- cycle    ;
\draw [color={rgb, 255:red, 161; green, 68; blue, 0 }  ,draw opacity=1 ]   (391.03,72.44) .. controls (377.73,67.38) and (367.74,66.52) .. (354.4,77.36) ;
\draw [shift={(352.27,79.17)}, rotate = 318.59] [fill={rgb, 255:red, 161; green, 68; blue, 0 }  ,fill opacity=1 ][line width=0.08]  [draw opacity=0] (8.04,-3.86) -- (0,0) -- (8.04,3.86) -- (5.34,0) -- cycle    ;

\draw (300.71,131.54) node [anchor=east] [inner sep=0.75pt]  [font=\small,color={rgb, 255:red, 96; green, 176; blue, 2 }  ,opacity=1 ,xscale=1.2,yscale=1.2]  {$f_{*} \lambda =\lambda $};
\draw (352.98,179.95) node  [font=\small,color={rgb, 255:red, 0; green, 14; blue, 191 }  ,opacity=1 ,xscale=1.2,yscale=1.2]  {$h_{top} =0$};
\draw (218.22,203) node [anchor=south east] [inner sep=0.75pt]  [font=\footnotesize,color={rgb, 255:red, 0; green, 85; blue, 186 }  ,opacity=1 ,xscale=1.2,yscale=1.2]  {$\mathrm{dim}( \mathrm{span}(\rot)) \leq g$};
\draw (409.07,136.63) node [anchor=west] [inner sep=0.75pt]  [font=\footnotesize,color={rgb, 255:red, 154; green, 161; blue, 1 }  ,opacity=1 ,xscale=1.2,yscale=1.2]  {$\mathrm{dim}( \mathrm{span}(\rot)) > g$};
\draw (388,205) node [anchor=west] [inner sep=0.75pt]  [font=\footnotesize,color=black ,opacity=1 ,xscale=1.2,yscale=1.2]  {$\Homeo_0(S)$};
\draw (224.85,99) node [anchor=east] [inner sep=0.75pt]  [font=\footnotesize,color={rgb, 255:red, 160; green, 125; blue, 0 }  ,opacity=1 ,xscale=1.2,yscale=1.2]  {\underline{$\mathrm{int}(\conv(\rote)) \neq\emptyset$}};
\draw (393.03,72.44) node [anchor=west] [inner sep=0.75pt]  [font=\footnotesize,color={rgb, 255:red, 161; green, 68; blue, 0 }  ,opacity=1 ,xscale=1.2,yscale=1.2]  {$\mathrm{int}(\rote) \neq\emptyset$};

\end{tikzpicture}
\caption{The space $\Homeo_0(S)$ from a rotational viewpoint. Here $\lambda$ is Lebesgue measure. We have the inclusions $\{h_{top} = 0\}\subset \{\mathrm{dim}( \mathrm{span}(\rot)) \leq g\}$ and $\{\mathrm{int}(\rote) \neq\emptyset\}\subset \{\mathrm{int}(\conv(\rote)) \neq\emptyset\}\subset \{\mathrm{dim}( \mathrm{span}(\rot)) > g\}$. The dots represent the set $\{h_{top} = 0\}$.
The present work is focused on the underlined case $\mathrm{int}(\conv(\rote)) \neq\emptyset$. Note that an open and dense subset of $\Homeo_0(S)$ was studied in depth from a rotational viewpoint in \cite{MR4578317}.}\label{FigSpaceHomeo}
\end{center}
\end{figure}
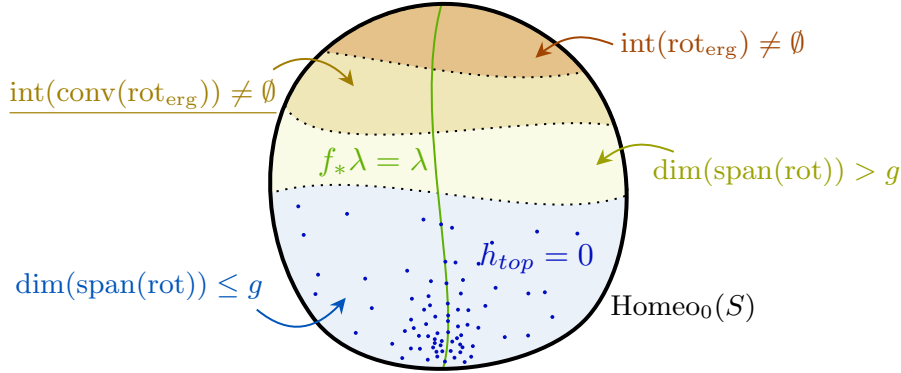

The results of the present work suggest that homeomorphisms with big rotation set should be seen as ``axiom A from the rotational viewpoint'': the whole rotational dynamics is described by a finite number of pieces bearing some hyperbolic-like features together with the heteroclinic connections between them (see Figure~\ref{FigExIntro}). 
\bigskip

The proofs will be made of rather direct applications of the companion article \cite{G25Cvx1}, itself based on a bounded deviations result \cite{paper1PAF} as well as the study of rotational properties of invariant measures \cite{alepablo}. The use of invariant measures as a line of weakness in breaking down the dynamics and of the forcing theory of Le Calvez and Tal \cite{lct1, lct2} as a technical tool are the keys to the recent progresses in the study of rotational behaviour of homeomorphisms in higher genus. 

The actual paper is an illustration of the potential applications of the hyperbolic-like techniques developed in the companion paper \cite{G25Cvx1}, which may be an important tool in the description of the rotational dynamics of any homeomorphism. 

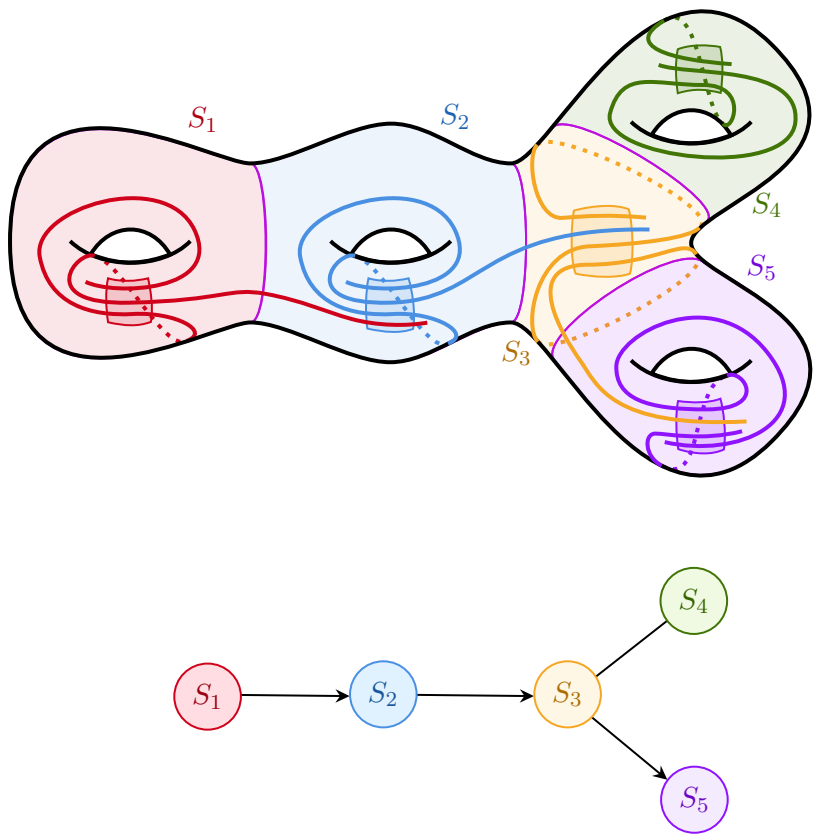
\begin{figure}[!t]
\begin{center}
\tikzset{every picture/.style={line width=0.75pt}} 
\vspace{-25pt}
\begin{tikzpicture}[x=0.75pt,y=0.75pt,yscale=-1,xscale=1]

\draw [color={rgb, 255:red, 144; green, 19; blue, 254 }  ,draw opacity=1 ][line width=1.5]  [dash pattern={on 1.69pt off 2.76pt}]  (365.28,232.01) .. controls (387.12,244.01) and (379.05,191.88) .. (397.87,186.65) ;
\draw [color={rgb, 255:red, 245; green, 166; blue, 35 }  ,draw opacity=1 ][line width=1.5]  [dash pattern={on 1.69pt off 2.76pt}]  (303.47,169.71) .. controls (311.35,179.56) and (394.43,135.56) .. (382.58,125.71) ;
\draw [color={rgb, 255:red, 245; green, 166; blue, 35 }  ,draw opacity=1 ][line width=1.5]  [dash pattern={on 1.69pt off 2.76pt}]  (303.11,70.63) .. controls (311.62,60.63) and (397.2,102.94) .. (384.08,112.33) ;
\draw [color={rgb, 255:red, 65; green, 117; blue, 5 }  ,draw opacity=1 ][line width=1.5]  [dash pattern={on 1.69pt off 2.76pt}]  (366.47,8.19) .. controls (384.4,-3.93) and (385.77,53.55) .. (395.99,59.99) ;
\draw [color={rgb, 255:red, 74; green, 144; blue, 226 }  ,draw opacity=1 ][line width=1.5]  [dash pattern={on 1.69pt off 2.76pt}]  (212.23,126.52) .. controls (226.07,129.9) and (247.33,176.43) .. (258.65,170.58) ;
\draw [color={rgb, 255:red, 208; green, 2; blue, 27 }  ,draw opacity=1 ][line width=1.5]  [dash pattern={on 1.69pt off 2.76pt}]  (82.23,126.52) .. controls (96.07,129.9) and (116.84,175.13) .. (128.16,169.29) ;
\draw  [color={rgb, 255:red, 189; green, 16; blue, 224 }  ,draw opacity=1 ][fill={rgb, 255:red, 144; green, 19; blue, 254 }  ,fill opacity=0.1 ] (386.74,130.02) .. controls (373.02,117.02) and (300.62,165.95) .. (311.85,179.44) .. controls (340.01,216.6) and (377.72,261.75) .. (420,220) .. controls (465.44,174.32) and (420.87,150.89) .. (386.74,130.02) -- cycle ;
\draw  [color={rgb, 255:red, 189; green, 16; blue, 224 }  ,draw opacity=1 ][fill={rgb, 255:red, 65; green, 117; blue, 5 }  ,fill opacity=0.1 ] (310.7,62.01) .. controls (320.29,50.63) and (400.66,100.63) .. (387.33,109.62) .. controls (430.43,82.1) and (459.99,61.88) .. (420,20) .. controls (379.77,-21.9) and (338.66,24.77) .. (310.7,62.01) -- cycle ;
\draw  [color={rgb, 255:red, 189; green, 16; blue, 224 }  ,draw opacity=1 ][fill={rgb, 255:red, 245; green, 166; blue, 35 }  ,fill opacity=0.1 ] (310.7,62.01) .. controls (320.02,50.63) and (400.84,100.81) .. (387.33,109.62) .. controls (381.93,115.02) and (380.7,114.01) .. (380,120) .. controls (380.7,126.3) and (384.73,128.02) .. (386.73,130.02) .. controls (373.02,117.15) and (300.22,166.35) .. (311.84,179.44) .. controls (305.43,172.86) and (298.13,160.59) .. (290,160) .. controls (300.13,159.82) and (300.13,79.82) .. (290,80) .. controls (299.13,78.22) and (304.13,68.87) .. (310.7,62.01) -- cycle ;
\draw  [color={rgb, 255:red, 189; green, 16; blue, 224 }  ,draw opacity=1 ][fill={rgb, 255:red, 74; green, 144; blue, 226 }  ,fill opacity=0.1 ] (230,60) .. controls (250,59.53) and (270,80.14) .. (290,80) .. controls (300.13,79.62) and (300.13,159.82) .. (290,160) .. controls (270.67,160.8) and (251,180.19) .. (230,180) .. controls (209,179.8) and (180.67,160.19) .. (160,160) .. controls (170.67,159.86) and (170.34,80.19) .. (160,80) .. controls (178.67,80.14) and (210,60.47) .. (230,60) -- cycle ;
\draw  [color={rgb, 255:red, 189; green, 16; blue, 224 }  ,draw opacity=1 ][fill={rgb, 255:red, 208; green, 2; blue, 27 }  ,fill opacity=0.1 ] (40,120) .. controls (40.21,21.33) and (139.73,80.32) .. (160,80) .. controls (169.84,80.08) and (170.34,159.83) .. (160,160) .. controls (139.42,159.92) and (40.58,220.6) .. (40,120) -- cycle ;
\draw  [fill={rgb, 255:red, 255; green, 255; blue, 255 }  ,fill opacity=1 ][line width=1.5]  (119.89,125.85) .. controls (108.41,129.97) and (96.62,132.39) .. (80.3,125.66) .. controls (90.28,109.13) and (110.17,109.02) .. (119.89,125.85) -- cycle ;
\draw [line width=1.5]    (70,119.08) .. controls (85.09,133.8) and (116.09,133.47) .. (130,119.08) ;

\draw  [color={rgb, 255:red, 208; green, 2; blue, 27 }  ,draw opacity=1 ][fill={rgb, 255:red, 208; green, 2; blue, 27 }  ,fill opacity=0.15 ] (88.93,137.79) .. controls (95.54,139.75) and (102.77,139.29) .. (108.93,137.79) .. controls (110.01,143.59) and (111.7,151.56) .. (110.31,159.75) .. controls (103.59,162.17) and (97.39,161.75) .. (88.93,159.9) .. controls (87.85,152.48) and (87.54,144.2) .. (88.93,137.79) -- cycle ;
\draw [color={rgb, 255:red, 208; green, 2; blue, 27 }  ,draw opacity=1 ][line width=1.5]    (77.71,139.54) .. controls (91.37,146.54) and (146.91,144.82) .. (132.03,113.43) .. controls (117.15,82.04) and (50.38,102.67) .. (54.82,126.89) .. controls (59.26,151.11) and (80.2,156.93) .. (101.39,155.75) .. controls (122.58,154.56) and (140.78,164.67) .. (128.16,169.29) ;
\draw  [line width=1.5]  (160,80) .. controls (179.42,79.96) and (209,60.14) .. (230,60) .. controls (251,59.86) and (270.34,79.86) .. (290,80) .. controls (309.67,80.14) and (360.39,-39.28) .. (420,20) .. controls (479.61,79.28) and (380.28,100.94) .. (380,120) .. controls (379.72,139.06) and (480.82,159.39) .. (420,220) .. controls (359.18,280.61) and (310.61,159.86) .. (290,160) .. controls (269.39,160.14) and (250.67,179.86) .. (230,180) .. controls (209.34,180.14) and (178.85,159.97) .. (160,160) .. controls (141.15,160.03) and (40.58,219.75) .. (40,120) .. controls (39.42,20.25) and (140.58,80.04) .. (160,80) -- cycle ;
\draw  [fill={rgb, 255:red, 255; green, 255; blue, 255 }  ,fill opacity=1 ][line width=1.5]  (249.89,125.85) .. controls (238.41,129.97) and (226.62,132.39) .. (210.3,125.66) .. controls (220.28,109.13) and (240.17,109.02) .. (249.89,125.85) -- cycle ;
\draw [line width=1.5]    (200,119.08) .. controls (215.09,133.8) and (246.09,133.47) .. (260,119.08) ;

\draw  [fill={rgb, 255:red, 255; green, 255; blue, 255 }  ,fill opacity=1 ][line width=1.5]  (399.89,65.85) .. controls (388.41,69.97) and (376.62,72.39) .. (360.3,65.66) .. controls (370.28,49.13) and (390.17,49.03) .. (399.89,65.85) -- cycle ;
\draw [line width=1.5]    (350,59.08) .. controls (365.08,73.81) and (396.08,73.47) .. (410,59.08) ;

\draw  [fill={rgb, 255:red, 255; green, 255; blue, 255 }  ,fill opacity=1 ][line width=1.5]  (399.89,185.85) .. controls (388.41,189.97) and (376.62,192.39) .. (360.3,185.66) .. controls (370.28,169.13) and (390.17,169.03) .. (399.89,185.85) -- cycle ;
\draw [line width=1.5]    (350,179.08) .. controls (365.08,193.81) and (396.08,193.47) .. (410,179.08) ;

\draw  [color={rgb, 255:red, 74; green, 144; blue, 226 }  ,draw opacity=1 ][fill={rgb, 255:red, 74; green, 144; blue, 226 }  ,fill opacity=0.15 ] (218.93,137.79) .. controls (225.54,139.75) and (232.77,139.29) .. (238.93,137.79) .. controls (240.69,143.75) and (242.58,155.06) .. (241.19,163.25) .. controls (234.47,165.67) and (226.82,164.93) .. (218.36,163.08) .. controls (217.28,155.66) and (217.54,144.2) .. (218.93,137.79) -- cycle ;
\draw [color={rgb, 255:red, 74; green, 144; blue, 226 }  ,draw opacity=1 ][line width=1.5]    (207.71,139.54) .. controls (221.37,146.54) and (276.91,144.82) .. (262.03,113.43) .. controls (247.15,82.04) and (180.38,102.67) .. (184.82,126.89) .. controls (189.26,151.11) and (210.2,156.93) .. (231.39,155.75) .. controls (252.58,154.56) and (271.28,165.97) .. (258.65,170.58) ;
\draw [color={rgb, 255:red, 208; green, 2; blue, 27 }  ,draw opacity=1 ][line width=1.5]    (82.23,126.52) .. controls (73.61,123.29) and (51.46,151.13) .. (99.46,150.06) .. controls (147.46,148.98) and (149.17,142.2) .. (167.57,145.6) .. controls (185.97,149) and (197.25,154.59) .. (210.25,158.19) .. controls (223.25,161.79) and (238.47,161.88) .. (248.27,160.08) ;
\draw  [color={rgb, 255:red, 65; green, 117; blue, 5 }  ,draw opacity=1 ][fill={rgb, 255:red, 65; green, 117; blue, 5 }  ,fill opacity=0.15 ] (373.33,22.59) .. controls (379.67,20.03) and (388.44,20.34) .. (394.44,21.42) .. controls (395.98,27.11) and (395.98,36.03) .. (394.71,44.55) .. controls (388.21,42.88) and (381.99,41.99) .. (373.33,44.7) .. controls (372.25,37.28) and (371.94,29) .. (373.33,22.59) -- cycle ;
\draw [color={rgb, 255:red, 65; green, 117; blue, 5 }  ,draw opacity=1 ][line width=1.5]    (363.77,30.44) .. controls (378.09,36.09) and (435.1,29.77) .. (431.99,59.33) .. controls (428.88,88.88) and (333.54,77.33) .. (339.54,54.66) .. controls (345.54,31.99) and (362.04,40.66) .. (385.79,40.55) .. controls (409.54,40.44) and (400.43,65.55) .. (395.99,59.99) ;
\draw [color={rgb, 255:red, 65; green, 117; blue, 5 }  ,draw opacity=1 ][line width=1.5]    (400.18,30.06) .. controls (369.34,27.94) and (345.59,20.94) .. (366.47,8.19) ;
\draw  [color={rgb, 255:red, 245; green, 166; blue, 35 }  ,draw opacity=1 ][fill={rgb, 255:red, 245; green, 166; blue, 35 }  ,fill opacity=0.15 ] (321.49,103.73) .. controls (330.91,101.16) and (342.63,100.3) .. (348.91,103.16) .. controls (350.91,111.73) and (351.49,126.87) .. (349.49,135.44) .. controls (342.76,137.86) and (330.8,137.29) .. (322.34,135.44) .. controls (319.49,128.3) and (320.1,110.15) .. (321.49,103.73) -- cycle ;
\draw [color={rgb, 255:red, 245; green, 166; blue, 35 }  ,draw opacity=1 ][line width=1.5]    (303.11,70.63) .. controls (297.26,76.79) and (301.93,108.48) .. (315.78,107.4) .. controls (329.62,106.33) and (339.01,105.71) .. (357.47,107.25) ;
\draw [color={rgb, 255:red, 245; green, 166; blue, 35 }  ,draw opacity=1 ][line width=1.5]    (303.47,169.71) .. controls (296.12,162.02) and (302.39,122.94) .. (329.2,121.71) .. controls (356.01,120.48) and (379.66,117.4) .. (384.08,112.33) ;
\draw [color={rgb, 255:red, 74; green, 144; blue, 226 }  ,draw opacity=1 ][line width=1.5]    (212.23,126.52) .. controls (203.61,123.29) and (181.46,151.13) .. (229.46,150.06) .. controls (277.46,148.98) and (273.7,112.02) .. (359.2,113.25) ;
\draw  [color={rgb, 255:red, 144; green, 19; blue, 254 }  ,draw opacity=1 ][fill={rgb, 255:red, 144; green, 19; blue, 254 }  ,fill opacity=0.15 ] (373.48,199.51) .. controls (380.29,200.49) and (388.91,200.65) .. (394.6,198.34) .. controls (396.14,204.03) and (397.6,215.36) .. (396.33,223.88) .. controls (390.64,226.65) and (382.17,226.49) .. (373.56,223.72) .. controls (372.48,216.3) and (372.09,205.93) .. (373.48,199.51) -- cycle ;
\draw [color={rgb, 255:red, 144; green, 19; blue, 254 }  ,draw opacity=1 ][line width=1.5]    (366.79,220.18) .. controls (402.94,228.49) and (445.41,210.34) .. (418.02,175.11) .. controls (390.64,139.88) and (342.94,165.88) .. (342.64,182.03) .. controls (342.33,198.18) and (357.25,203.88) .. (388.02,203.72) .. controls (418.79,203.57) and (406.33,183.26) .. (397.87,186.65) ;
\draw [color={rgb, 255:red, 144; green, 19; blue, 254 }  ,draw opacity=1 ][line width=1.5]    (365.28,232.01) .. controls (356.57,227.54) and (356.34,215.9) .. (362.1,215.9) .. controls (367.87,215.9) and (385.32,219.79) .. (405.4,214.72) ;
\draw [color={rgb, 255:red, 245; green, 166; blue, 35 }  ,draw opacity=1 ][line width=1.5]    (407.63,209.78) .. controls (339.63,213.19) and (333.07,191.32) .. (324.58,179.14) .. controls (316.1,166.96) and (294.05,131.09) .. (328.44,130.95) .. controls (362.84,130.82) and (376.89,116.63) .. (382.58,125.71) ;

\draw (127.08,65.48) node [anchor=south west] [inner sep=0.75pt]  [color={rgb, 255:red, 174; green, 0; blue, 21 }  ,opacity=1 ]  {$S_{1}$};
\draw (253.08,63.98) node [anchor=south west] [inner sep=0.75pt]  [color={rgb, 255:red, 30; green, 99; blue, 179 }  ,opacity=1 ]  {$S_{2}$};
\draw (301.98,167.9) node [anchor=north east] [inner sep=0.75pt]  [color={rgb, 255:red, 174; green, 110; blue, 5 }  ,opacity=1 ]  {$S_{3}$};
\draw (408.58,93.4) node [anchor=north west][inner sep=0.75pt]  [color={rgb, 255:red, 65; green, 117; blue, 5 }  ,opacity=1 ]  {$S_{4}$};
\draw (406.08,139.98) node [anchor=south west] [inner sep=0.75pt]  [color={rgb, 255:red, 118; green, 3; blue, 220 }  ,opacity=1 ]  {$S_{5}$};
\end{tikzpicture}

\begin{tikzpicture}[x=0.75pt,y=0.75pt,yscale=-1,xscale=1]

\draw    (118,133) -- (185.41,133.71) ;
\draw [shift={(188.41,133.74)}, rotate = 180.61] [fill={rgb, 255:red, 0; green, 0; blue, 0 }  ][line width=0.08]  [draw opacity=0] (7.14,-3.43) -- (0,0) -- (7.14,3.43) -- (4.74,0) -- cycle    ;
\draw    (210,133) -- (277.41,133.71) ;
\draw [shift={(280.41,133.74)}, rotate = 180.61] [fill={rgb, 255:red, 0; green, 0; blue, 0 }  ][line width=0.08]  [draw opacity=0] (7.14,-3.43) -- (0,0) -- (7.14,3.43) -- (4.74,0) -- cycle    ;
\draw    (297.71,134.71) -- (344.68,173.9) ;
\draw [shift={(346.98,175.83)}, rotate = 219.84] [fill={rgb, 255:red, 0; green, 0; blue, 0 }  ][line width=0.08]  [draw opacity=0] (7.14,-3.43) -- (0,0) -- (7.14,3.43) -- (4.74,0) -- cycle    ;
\draw    (360.05,85.47) -- (297.71,134.71) ;

\draw  [color={rgb, 255:red, 208; green, 2; blue, 27 }  ,draw opacity=1 ][fill={rgb, 255:red, 255; green, 221; blue, 227 }  ,fill opacity=1 ]  (117.5, 134) circle [x radius= 16.62, y radius= 16.62]   ;
\draw (117.5,134) node  [color={rgb, 255:red, 152; green, 0; blue, 17 }  ,opacity=1 ]  {$S_{1}$};
\draw  [color={rgb, 255:red, 74; green, 144; blue, 226 }  ,draw opacity=1 ][fill={rgb, 255:red, 224; green, 241; blue, 255 }  ,fill opacity=1 ]  (205.21, 132.86) circle [x radius= 16.62, y radius= 16.62]   ;
\draw (205.21,132.86) node  [color={rgb, 255:red, 26; green, 87; blue, 156 }  ,opacity=1 ]  {$S_{2}$};
\draw  [color={rgb, 255:red, 245; green, 166; blue, 35 }  ,draw opacity=1 ][fill={rgb, 255:red, 255; green, 247; blue, 229 }  ,fill opacity=1 ]  (297.21, 132.86) circle [x radius= 16.62, y radius= 16.62]   ;
\draw (297.21,132.86) node  [color={rgb, 255:red, 174; green, 108; blue, 0 }  ,opacity=1 ]  {$S_{3}$};
\draw  [color={rgb, 255:red, 144; green, 19; blue, 254 }  ,draw opacity=1 ][fill={rgb, 255:red, 245; green, 235; blue, 255 }  ,fill opacity=1 ]  (360.5, 185.86) circle [x radius= 16.62, y radius= 16.62]   ;
\draw (360.5,185.86) node  [color={rgb, 255:red, 107; green, 7; blue, 193 }  ,opacity=1 ]  {$S_{5}$};
\draw  [color={rgb, 255:red, 65; green, 117; blue, 5 }  ,draw opacity=1 ][fill={rgb, 255:red, 240; green, 250; blue, 226 }  ,fill opacity=1 ]  (360.21, 85.86) circle [x radius= 16.62, y radius= 16.62]   ;
\draw (360.21,85.86) node  [color={rgb, 255:red, 54; green, 100; blue, 1 }  ,opacity=1 ]  {$S_{4}$};

\end{tikzpicture}

\caption{An example of homeomorphism of $\Homeo_0(S)$ with big rotation set. It has 5 different rotational horseshoes having (or not) heteroclinic connections. In the paper we will define sub-surfaces $S_i$ associated to the dynamics and a graph whose vertices are those sub-surfaces and the edges are given by the adjacency; some will be oriented by the relation of heteroclinic connections. We will see that a homeomorphism with big rotation set resembles a lot the one of this example.}\label{FigExIntro}

\end{center}
\end{figure}

\subsection*{Rotation sets}

Let $S$ be an orientable closed hyperbolic surface. In other words, $S$ is homeomorphic to the connected sum of $g$ tori, with $g\ge 2$. We endow $S$ with a metric of constant curvature $-1$. The universal cover $\wt S$ of $S$ is isometric to the hyperbolic plane $\Hy^2$, the group of deck transformations of this cover will be denoted $\G$. 

We denote $\Homeo_0(S)$ the set of homeomorphisms of $S$ that are isotopic (or, equivalently, by \cite{zbMATH03221970}, homotopic) to the identity. Any element $f\in \Homeo_0(S)$ has a preferred lift $\wt f \in \Homeo(\wt S)$ that commutes with the deck transformations; this homeomorphism $\wt f$ has a contractible fixed point, and extends to the identity on $\partial \wt S$. 
We denote $\mathcal{M}(f)$ the set of $f$-invariant Borel probability measures, and $\mathcal{M}^{\textnormal{erg}}(f)$ the subset of $\mathcal{M}(f)$ made of $f$-ergodic measures. 

Let us define the homological rotation set of such a homeomorphism $f\in \Homeo_0(S)$; this definition is due to Schwarzman \cite{MR88720} and was adapted for surface homeomorphisms by Pollicott \cite{pollicott}.
We recall that as $S$ is a closed surface of genus $g$, the homology group $H_1(S,\R)$ is a real vector space of dimension $2g$.
We equip the homology $H_1(S,\R)$ with a norm $\|\mathord{\cdot}\|$ and the intersection form $\wedge$ coming from the classical cup product on cohomology via Poincar\'e duality. This intersection form coincides with the classical algebraic intersection number in restriction to elements of $H_1(S,\Z)$ (\emph{e.g.}\ \cite{lellouch}).
Given \(a \in \G\), we denote \([a] \in H_1(S,\R)\) its \emph{homology class}.

Fix a bounded and measurable fundamental domain \(D \subset \widetilde S\) for the action of \(\G\) on \(\widetilde S\) and denote \(\widetilde x\) the lift of \(x \in S\) to \(D\). 
For each \( y \in S\) let \(a_{y}\) be the unique element of $\G$ such that \(a_y^{-1}\wt f(\wt y) \in D\).\label{Defay}
For any path $\beta : [0,1]\to S$, we consider $\wt\beta : [0,1]\to \wt S$ the lift of $\beta$ such that $\wt\beta(0)\in D$, and $T_\beta\in \G$ such that $\wt\beta(1)\in T_\beta D$. This allows to define $[\beta] = [T_\beta]\in H_1(S,\Z)$ .

\begin{definition}\label{DefHomRotVect}
Given an \(f\)-invariant probability measure \(\mu\), the \emph{homological rotation vector} of \(\mu\) is
\begin{equation}\label{eq:homologyequation}
\rho(\mu) = \int_{S}[a_y]\dd\mu(y).
\end{equation}
\end{definition}

Note that by Birkhoff ergodic theorem, if moreover \(\mu\) is ergodic, then for \(\mu\)-almost every \(x \in S\)
\begin{equation}\label{eq:homologyequation2}
\rho(\mu) = \int_{S}[a_y]\dd\mu(y) = \lim\limits_{n \to +\infty}\frac{1}{n}\sum_{i=0}^{n-1}[a_{f^i(x)}].
\end{equation}
If $x\in S$ is such that the right equality of \eqref{eq:homologyequation2} holds, we will denote $\rho(x) = \rho(\mu)$. More generally, we will denote $\rho(x)$ as the set of accumulation points of the sequence
\[\left(\frac{1}{n}\sum_{i=0}^{n-1}[a_{f^i(x)}]\right)_n.\]

This definition is independent of the choice of the fundamental domain $D$, moreover the map $\mu\mapsto \rho(\mu)$ is affine and continuous (see \cite[Remark~1.2]{G25Cvx1}). 

\begin{definition}[Homological rotation sets]\label{DefErgHomRot}
Let $f \in \Homeo_0(S)$. 
The \emph{(homological) rotation set} $\rot(f)$ of $f$ is the set of vectors $\rho\in H_1(S,\R)$ such that there exists $(x_k)_k\in S^\N$ and $(n_k)_k\in\N^\N$ with $\lim_{k\to+\infty} n_k = +\infty$ and such that 
\[\lim\limits_{k \to +\infty}\frac{1}{n_k}\sum_{i=0}^{n_k-1}[a_{f^{i}(x_k)}] = \rho.\]
The \emph{ergodic (homological) rotation set} $\rote(f)$ of $f$ is
\[\rote(f) = \big\{\rho(\mu) \mid \mu \in \Me(f)\big\}.\] 
\end{definition}

Denote $\conv(A)$ the convex hull of a set $A$. The following lemma is straightforward.

\begin{lemma}\label{LemInclusEnsRot}
Let $f\in\Homeo_0(S)$. Then $\rot(f)$ is compact, and 
\[\rote(f)\subset \rot(f)\subset \conv\big(\rote(f)\big).\]
\end{lemma}

For a proof of the last inclusion, see \cite[Corollary 1.2]{pollicott}. 

This lemma implies that $\conv\rot(f)\subset \conv\rote(f)$. Therefore, the following properties are equivalent:
\begin{itemize}
\item $\inte\big(\conv\rot(f)\big)\neq\emptyset$;
\item $\inte\big(\conv\rote(f)\big)\neq\emptyset$;
\item $\inte\big\{\rot(\mu)\mid \mu\in\mathcal{M}(f)\big\}\neq\emptyset$;
\item $\operatorname{span}\big(\conv\rot(f)\big) = H_1(S,\R)$.
\end{itemize}

\begin{definition}\label{DefBig}
We say that $f\in \Homeo_0(S)$ has \emph{a big rotation set} if one of the above equivalent properties holds.
\end{definition}

\paragraph{Standing hypothesis:} For now on all considered homeomorphisms are elements of $\Homeo_0(S)$ having a big rotation set.

\subsection*{Torus homeomorphisms with big rotation set}

In the case of the torus, there is a more precise description of the relations between the two rotation sets $\rot(f)$ and $\rote(f)$.

Indeed, the same definition of rotation set can be applied in the case of the two torus $\T^2$, with the difference that the obtained sets do depend on the chosen lift $\wt f$ of $f$; however the sets associated to two different lifts only differ by an integer translation. 

For $f\in\Homeo_0(\T^2)$:
\begin{enumerate}[label= \Alph*)]
\item\label{ItemA} the set $\rot(\wt f)$ is convex \cite{zbMATH04084609} and hence $\rot(\wt f) =\conv(\rote(\wt f))$;
\item\label{ItemB} if $f$ has a big rotation set, then $\rot(\wt f) \setminus \rote(\wt f) \subset \partial\rot(f) \setminus \operatorname{extr}(\rot(f))$ \cite{zbMATH00009916};
\item\label{ItemC} if $f$ has a big rotation set, then $f$ has bounded deviations with respect to $\rot(f)$ \cite{zbMATH06425076, lct1} (see Theorem~\ref{TheoBndedHomoLellou} for a definition);
\item\label{ItemD} the so-called \emph{Boyland conjecture} holds: if $f$ has big rotation set and preserves area, then the rotation vector of the Lebesgue measure belongs to the interior of the rotation set \cite{zbMATH06425076, lct1}.
\end{enumerate}


\subsection*{The shape of big rotation sets}

For higher genus surfaces, the last inclusion of Lemma~\ref{LemInclusEnsRot} is not an equality in general, as the convexity of the rotation set $\rot(f)$ does not always hold (see Matsumoto's example: \cite[Proposition 3.2]{matsumoto}, \cite[Figure 10]{alepablo}; another example is depicted in Figure~\ref{FigTree} page \pageref{FigTree}).

Properties similar to \ref{ItemA} and \ref{ItemB} were already proved in higher genus under the hypothesis that $\inte(\rote(f))\neq\emptyset$ (see \cite{guiheneuf2023hyperbolic, alepablo}).

In this article, we first prove a counterpart of \ref{ItemA} under the weaker hypothesis of big rotation sets:

\begin{theo}\label{MainTheoUnionConvex}
Let $f\in\Homeo_0(S)$ with big rotation set (Definition~\ref{DefBig}). Then $\rot(f)$ is the union of at most $\max(g, \lfloor g^2/4\rfloor)$ convex sets $\rho'_j$ containing 0.
Each $\rho'_j$ is of the form (see \cite[Subsection~2.2]{G25Cvx1} for the definition of the objects)
\[\rho'_j = \conv\left(\bigcup_{i\in I^j} \rho_i\right),\]
where $I^j$ is a subset of $I_{\mathrm{h}}$.
\end{theo}

Note that the estimate on the number of pieces is sharp: for any $g$ there exists a homeomorphism of a surface of genus $g$ with big rotation set and whose rotation set is not the union of $\max(g, \lfloor g^2/4\rfloor) -1$ convex sets (see Subsection~\ref{paragSharpness}).

The proof of Theorem~\ref{MainTheoUnionConvex} has the three following consequences.
The first one is under a hypothesis similar to \cite[Theorem D]{zbMATH06908424} (although weaker), which improves \cite[Theorem~3]{zbMATH07282570}.

\begin{theo}\label{MainTheoConvex}
Let $f\in\Homeo_0(S)$ such that $\inte\big(\rot(f)\big)\neq\emptyset$. Then $\rot(f)$ is convex. 

More precisely, if $\rot(f)$ is not included in a hyperplane, or in a union of $2$ vector subspaces of codimension 2, then $\rot(f)$ is convex and $\rot(f) = \overline{\rote(f)}$. 
\end{theo}

Note that the first conclusion of this theorem was proved for an open and dense set of homeomorphisms satisfying $\inte(\rot(f))\neq\emptyset$ in \cite{MR4578317}. This first conclusion also gives some constraints on rotation sets: it proves that some compact sets of $R^{2g}$ cannot be realised as rotation sets of higher genus surface homeomorphisms.

The second consequence is that the same conclusion holds when $f$ is transitive.

\begin{propo}\label{CoroTheoConvexTrans}
Let $f\in\Homeo_0(S)$ be transitive and with big rotation set. Then $\rot(f)$ is convex and $\rot(f) = \overline{\rote(f)}$ (and in particular $\inte \rote(f)\neq\emptyset$).
\end{propo}

Note that using \cite{guiheneuf2023hyperbolic} or \cite{alepablo}, we deduce that if $f\in\Homeo_0(S)$ is transitive and has big rotation set, then it is in the class of a pseudo-Anosov relative to a finite number of periodic orbits. 

The last consequence concerns the non wandering case.

\begin{propo}\label{CoroTheoConvexNW}
Let $f\in\Homeo_0(S)$ be nonwandering and with big rotation set. Then $\rot(f)$ is the union of at most $g$ convex sets $\rho_1,\dots,\rho_k$, with $H_1(S,\R) = \bigoplus_i \operatorname{span}(\rho_i)$. 
\end{propo}

This proposition, combined with the example of Paragraph~\ref{paragSharpness}, shows in particular that for any $g\ge 5$, there is a homeomorphism of the closed surface of genus $g$ whose (big) rotation set is not the rotation set of a nonwandering homeomorphism. This contrasts with the torus case where any nonempty interior rotation set is realised as the rotation set of an area-preserving homeomorphism \cite{garcia2024fullychaotic}.
\bigskip

Let us come back to Theorem~\ref{MainTheoUnionConvex}. It suggests that a similar conclusion may persist if the hypothesis of having a big rotation set is removed:

\begin{conj}
There exists a map $m: \N_{\ge 2} \to\N$ such that for any surface $S$ of genus $g\ge 2$ and any $f\in\Homeo_0(S)$, the rotation set $\rot(f)$ is the union of at most $m(g)$ convex sets. Moreover, $m(g)\le 2^{5g-5}$.
\end{conj}

Of course, if this conjecture is true, one can wonder what is the optimal map $m(g)$. The $5g-5$ constant comes from \cite[Theorem~F]{alepablo}. This conjecture is supported by this present work, which proves this conjecture in a special case, as well as the study \cite{MR4578317} of the shape of rotation sets for an open and dense set of homeomorphisms (and this set is understood \emph{via} Axiom A diffeomorphisms). In this work the authors prove that if it exists, the optimal bound $m(g)$ should be at least $2^{\lfloor g/2\rfloor}$ \cite[Section~10.4]{MR4578317}.

\subsection*{Existence of $f$-invariant open subsets}

When the rotation set can be split into supplementary pieces, there are associated pairwise disjoint open and $f$-invariant sets.

\begin{propo}\label{PropExistBj}
Let $f\in\Homeo_0(S)$ with big rotation set, and suppose that there are linear subspaces $(V_j)_{1\le j\le k}$ of $H_1(S,\R)$ such that $H_1(S,\R) = \bigoplus_j V_j$ and $\rot(f)\subset \bigcup_j V_j$.
Then there exists a family $B_1,\dots, B_k\subset S$ of pairwise disjoint $f$-invariant open sets such that for any $1\le j\le k$ one has $\rot(f|_{B_j} ) = V_j\cap \rot(f)$ and that, denoting $K$ the complement of the union of the $B_j$, one has $\rot(f|_{K} ) = \{0\}$.
\end{propo}

Note that this proposition can be applied to the case of nonwandering homeomorphisms with big rotation set thanks to Proposition~\ref{CoroTheoConvexNW}: in this case we get $\rot(f|_{B_i} ) = \rho_i$. 

This proposition is a direct application of \cite[Proposition~4.17]{G25Cvx1}, which also gives the existence of positively/negatively invariant sets for any big rotation set homeomorphism. 

One may compare Proposition~\ref{PropExistBj} with the main theorem of \cite{GlCPT} that also states, for an area-preserving zero entropy homeomorphism, the existence of an $f$-invariant decomposition of $S$ into open sets whose union contains all the rotational behaviour of $f$.

\subsection*{Bounded deviations for homeomorphisms with big rotation set}

The topic of bounded deviations for torus homeomorphisms is now quite developed, see \cite{zbMATH06304088, zbMATH06296542, zbMATH06914177, zbMATH07548721, zbMATH07488214, zbMATH07867510}. Some of these results deal with homeomorphisms whose rotation set has nonempty interior: \cite{zbMATH06425076, lct1}. However, only few results are known for higher genus surfaces: the only references we know are \cite{zbMATH07282570}, that treats the case of $C^{1+\alpha}$ diffeomorphisms under the hypothesis of existence of what the authors call a \emph{fully essential system of curves}, and \cite{lellouch} which concerns homeomorphisms such that 0 belongs to the interior of their rotation set. A study of bounded deviations in the homotopical sense is made in \cite{paper1PAF, paper2PAF}; we will use this work here. 

Let us first set a notation.
We define $a_x^n\in\G$ as the product of the $a_{f^i(x)}$ for $0\le i\le n-1$, hence
\begin{equation}\label{Defaxn}
[a_x^n] = \sum_{i=0}^{n-1}[a_{f^i(x)}].
\end{equation}
Let us recall a theorem of Lellouch, which already improves \cite[Theorem~4]{zbMATH07282570}:

\begin{theorem}[{\cite[Th\'eor\`eme G]{lellouch}}]\label{TheoBndedHomoLellou}
Let $f\in\Homeo_0(S)$ such that $0\in\inte\big(\conv(\rot(f))\big)$. Then the deviations with respect to $\conv(\rot(f))$ are bounded: there exists $L>0$ such that for any $x\in S$ and any $n\in\N$, we have:
\[d\big([a_x^n],\, n\conv(\rot(f))  \big)\le L.\]
\end{theorem}

We improve his result by weakening its hypothesis and strengthening its conclusion:

\begin{theo}\label{TheoBndedHomo}
Let $f\in\Homeo_0(S)$ such that $\inte\big(\conv(\rot(f))\big)\neq\emptyset$ (\emph{i.e.}~with big rotation set). Then the deviations with respect to $\rot(f)$ are bounded: there exists $L>0$ such that for any $x\in S$ and any $n\in\N$, we have :
\[d\big([a_x^n],\, n\rot(f)  \big)\le L.\]
\end{theo}

The hypothesis $\inte\big(\conv(\rot(f))\big)\neq\emptyset$ is necessary as there are counterexamples due to Koropecki and Tal \cite{zbMATH06345227} (see also \cite[Section~7.2]{pa}).
However, one may ask the following:

\begin{question}
Is it possible to get bounded deviations results with respect to the homological rotation set under the additional hypothesis that the fixed point set is inessential \cite{zbMATH06908424, zbMATH06997839}?
\end{question}

It is now classical that bounded deviation results imply some versions of Boyland conjecture (see \cite{zbMATH06425076, lct1, zbMATH07282570,lellouch}). Here is what implies Theorem~\ref{TheoBndedHomo} under our hypotheses (improving \cite[Theorem~5]{zbMATH07282570} and \cite[Th\'eor\`eme~I]{lellouch}):

\begin{corol}\label{CoroBoyland1}
Let $f\in\Homeo_0(S)$ preserving a Borel probability measure $\mu$ that has total support. If $f$ has big rotation set, then $\rho(\mu) \in\inte\big(\conv(\rot(f))\big)$.
\end{corol}

\subsection*{Realization of rotation vectors}

By Proposition~\ref{Prop:decomposition} (due to \cite{guiheneuf2023hyperbolic}), if $f$ has a big rotation set, then $\rote(f) = \bigcup_{i\in I_{\mathrm h}} \rho_i$, with $\card(I_{\mathrm h}) \le 2g-2$, and $\rho_i$ is such that $\conv(\rho_i)\setminus\rho_i \subset \partial \conv(\rho_i)$. This allows to apply \cite[Proposition~A]{G25Cvx1} to any piece $\rho_i$ of the set $\rote(f)$; this gives the following, which improves \cite[Theorem~3]{zbMATH07282570}:

\begin{propo}\label{PropRealByCompact2}
Let $f\in\Homeo_0(S)$ with big rotation set and $i\in I_{\mathrm h}$. Then:
\begin{enumerate}
\item For any $\rho\in \rho_i$, there exists $x\in S$, such that $\rho(x) = \rho$.
\item For $\rho\in \inte(\rho_i)$ (the interior is taken inside the span of the convex set), there exists a compact $f$-invariant set $K_\rho\subset S$ and $L_\rho>0$ such that for any $x\in K_\rho$ and any $n\in\N$,
\[d\big([a_x^n],\, n\rho \big)\le L_\rho.\]
\item If $C\subset \inte(\rho_i)$ is a compact connected set, then there exists $x\in S$ such that $\rho(x) = C$. 
\end{enumerate}
\end{propo}

Note that this answers Problem 1 of \cite{pollicott} in the case where $\rot(f) = \rho_i$: in this case any rotation vector is the rotation vector of some orbit. This problem remains open even in the more general case of a homeomorpihsm with big rotation set.
\bigskip

Given a compact convex set $C\subset \R^n$, a point $x\in C$ is called \emph{exposed} if there is an affine hyperplane $H\subset \R^n$ such that $C\cap H = \{x\}$; we write $\exp(C)$ the set of exposed point of $C$. A theorem of Straszewicz \cite{Straszewicz1935} asserts that $\overline{\exp(C)} = \extr(C)$.
The following is a direct consequence of boundedness of deviations (Theorem~\ref{TheoBndedHomo}):

\begin{corol}\label{CoroBoyland2}
Let $f\in\Homeo_0(S)$ with big rotation set. Then for any $\rho\in \exp\big(\conv(\rot(f))\big)$, there exists a compact $f$-invariant set $K_\rho\subset S$ such that for any $x\in K_\rho$, we have $\rho(x) = \rho$. 
Moreover, one can suppose that there are bounded deviations: there exists $L_0>0$ (independent of $\rho$) such that for any $x\in K_\rho$ and any $n\in\N$,
\[d\big([a_x^n],\, n\rho \big)\le L_0.\]
\end{corol}

A similar result was stated in \cite[Corollary 67]{lct1} for the torus case, under the stronger assumption that $\rho$ is a vertex of $C$ (\emph{i.e}\ if there exist $n$ independent linear forms in $\R^n$ such that $\rho$ belongs to the supporting lines defined by these forms).

\subsection*{Proof strategy and plan of the paper}

The proof of these results is based on the first part of this work \cite{G25Cvx1}: there is an equivalence relation on ergodic measures having nontrivial rotational behaviour; these equivalence classes (whose number is finite) act a bit as homoclinic classes of diffeomorphisms: they are associated to a lot of (topological and rotational) horseshoes (see \cite[Subsection~3.1]{G25Cvx1}), all having heteroclinic connections between them. 
There can also be heteroclinic connections between these classes, and these have numerous different characterizations, including one in terms of large deviations Proposition~\ref{PropDevImpliesConnec} and one in terms of heteroclinic connections between the horseshoes (\cite[Definition~4.9 and Theorem~B]{G25Cvx1}). These heteroclinic connections are represented by a finite oriented graph $\Tr$ (see Subsection~\ref{SubSecGrafT}) to which is associated a rotation set which coincides with the rotation set of the homeomorphism (Corollary~\ref{CoroPropRotFRotG}).

The whole picture is similar to the one painted in \cite{MR4578317}, but whereas this previous paper has hypotheses about the kind of dynamics involved (a true axiom A homeomorphism, hypothesis that is $C^0$ generic), our hypotheses involve the rotational behaviour of the homeomorphism (and we prove that indeed, from the rotational viewpoint, it really behaves as an axiom A). 

The plan of the paper is as follows. We first recall some properties of ergodic rotation sets of homeomorphisms with big rotation set due to \cite{guiheneuf2023hyperbolic} and \cite{alepablo} in Section~\ref{SecRotSetbig}. The next section is devoted to the adaptation of results of \cite{G25Cvx1} and \cite{paper1PAF} to two intermediate results about boundedness of homological rotational deviations, the first one with respect to some union of vector subspaces and the second one with respect to $\conv(\rot(f))$. The next section (Section~\ref{SecRefRefRef}) is devoted to the proof of the four first main results (involving the shape of rotation sets). We first study an oriented graph $G$ of Markovian intersections associated to the dynamics, and the quotient of this graph by its strong connected components denoted $\Tr$. This allows to get the results; we finish the section with some examples showing that the bound of Theorem~\ref{MainTheoUnionConvex} is sharp. 
We then prove the main result about bounded deviations in homology and its consequences in Section~\ref{SecBnddDev2}, and finish with some comments about the continuity properties of the various notions of rotation set.

\subsection*{Acknowledgements}

The author thanks Fabio Tal for the discussion about Theorem~\ref{TheoBndedHomo}.

\section{Ergodic rotation sets of homeomorphisms with big rotation set}\label{SecRotSetbig}

The following is a combination of Proposition~10 and Lemma~11 of \cite{guiheneuf2023hyperbolic}, whose proof is mainly based on \cite[Th\'eor\`eme~C]{lellouch}.

\begin{prop} \label{Prop:decomposition}
Let $f\in \Homeo_0(S)$ with big rotation set.
Then there exist pairwise orthogonal vector subspaces $V_1,V_2, \ldots, V_n$ of $H_1(S,\mathbb{R})$, with $n \le g$ such that:
\begin{itemize}
\item $\rote(f) \subset \bigcup_{j=1}^n V_j.$
\item For any $1 \leq j \leq n$, $V_j$ is a rational symplectic subspace of $H_1(S,\mathbb{R})$.
\item $H_1(S,\R) = \bigoplus_{j=1}^n V_j$.
\item For any $j$, the set $\overline{\rote(f)\cap V_j}$ is a convex set containing $0$, with nonempty interior in $V_j$, and with a dense subset of elements realised by periodic orbits.
\item For any two vectors $v, w\in \rote(f) \setminus \left\{ 0 \right\}$, the following are equivalent:
\begin{itemize}
\item there exist $v=v_1,v_2,\dots, v_k = w \in \rote(f)$ such that for any $1\le m<k$, we have $v_m \wedge v_{m+1}\neq 0$;
\item there exists $1 \leq j \leq n$ such that $v,w\in \rote(f) \cap V_j \setminus \left\{ 0 \right\}$.
\end{itemize}
\end{itemize} 
\end{prop}

By \cite[Proposition~4.7]{alepablo} (that roughly states that intersection in homology implies transverse intersection of some tracking geodesics), the last point implies that if $\mu_1,\mu_2\in \Merg(f)$ are such that there exists $1\le j\le n$ such that $\rho(\mu_1), \rho(\mu_2)\in V_j\setminus \{0\}$, then $\mu_1\sim\mu_2$ in the sense of \cite[Definition~1.6]{G25Cvx1}. 

Reciprocally, if $\mu_1,\mu_2\in \Merg(f)$ are such that $\rho(\mu_1), \rho(\mu_2)\neq 0$ and $\mu_1\sim\mu_2$, then by \cite[Theorem~F]{alepablo}, the interior of the triangle spanned by $\rho(\mu_1), \rho(\mu_2)$ and 0 is included in $\rote(f)$. Proposition~\ref{Prop:decomposition} implies that this triangle is included in a single $V_j$ (as these spaces are in direct sum) and hence there exist $\rho(\mu_1)=v_1,v_2,\dots, v_k = \rho(\mu_2) \in \rote(f)$ such that for any $m$, we have $v_m \wedge v_{m+1}\neq 0$. 

This implies that when $f$ has big rotation set, then the two decompositions of $\rote(f)$ given by Proposition~\ref{Prop:decomposition} and \cite[Theorem~F]{alepablo} (stated in \cite[Theorem~2.4]{G25Cvx1}) coincide, up to the fact that for some $i\in I_{\mathrm{h}}$ the set $\rho_i$ might be trivial: the subspaces $V_j$ of Proposition~\ref{Prop:decomposition} coincide with the $\operatorname{span}(\rho_i)$ for the $\rho_i$ defined in \cite[(4)]{G25Cvx1} (with $i\in I_{\mathrm h}$ and $\rho_i\neq\{0\}$), and the sets $V_j\cap \rot_{erg}(f)$ coincide with the $\rho_i$ (again, with $i\in I_{\mathrm h}$ and $\rho_i\neq\{0\}$).
Moreover, by \cite[Proposition~4.6]{alepablo}, we have $i_*(H_1(S_i,\R)) = V_i$. 
\bigskip

After \cite[Theorem~2.4]{G25Cvx1} is recalled the definition of the surfaces $(S_i)_{i\in I_{\mathrm h}}$ due to \cite{alepablo}; those are pairwise disjoint and satisfy $i_*H_1(S_i, \R) = V_i$.
Using $H_1(S,\R) = \bigoplus_{j=1}^n V_j$, we deduce that the complement of $\bigcup_i S_i$ is a disjoint union of surfaces of zero genus, \emph{i.e.}\ closed separating geodesics or surfaces homeomorphic to a sphere with a finite number of boundary components, all of them separating. 
None of these surfaces bear orientable minimal lamination different from a single closed geodesic, hence by \cite[Theorem~A]{paper2PAF} and \cite[Theorem~D]{alepablo} this implies that for any  $i\in I^1$ (defined in \cite[Definition~1.7]{G25Cvx1}), the lamination $\Lambda_i$ is made of a single separating geodesic (in particular, $\rho_i = \{0\}$). Hence, we do not have to bother neither with these classes (which do not affect the rotation of the homeomorphism), nor with classes whose associated laminations are minimal but not made of a single closed geodesic (which do not appear here).

\section{Bounded deviations I: two intermediate results}\label{SecBnddDev1}

In this section we start the study of bounded deviations for homeomorphisms with big rotation sets. However we won't be able to finish the proof of Theorem~\ref{MainTheoUnionConvex} yet, because this requires the structure theorem for rotation sets (Theorem~\ref{MainTheoUnionConvex}); this final proof will be made in Section~\ref{SecBnddDev2}. 

More specifically, we will prove two results of bounded deviations. 
The first one asserts that if there exists an orbit segment (for $f$) whose displacement is far away from the linear subspaces $V_i$, then it forces the existence of connections (in the sense of $\F$-transverse intersections, see \cite[Definition~4.6]{G25Cvx1}) between classes (Proposition~\ref{PropDevImpliesConnec}).
The second one is a result of bounded deviations with respect to the set $\conv(\rot(f))$ (Proposition~\ref{PropBnddDevConvHull}).

The proofs will consist in adapting the results of \cite{paper1PAF} to the homological setting.

\subsection{Preliminaries}

Let us state a variation of \cite[Definition~4.8]{G25Cvx1} due to \cite[Proposition~4.17]{G25Cvx1}:

\begin{definition}\label{DefToOpen2}
To any $i\in I_{\mathrm h}$, one can associate open subsets $B_i^-, B_i^+$ of $S$ such that:
\begin{itemize}
\item for any $\mu\in \cl_i$, we have $\mu(B_i^-) = \mu(B_i^+) = 1$;
\item $f^{-1}(B_i^-) \subset B_i^-$ and $f(B_i^+)\subset B_i^+$;
\item $i_* \pi_1(S_i,\R)\subset i_* \pi_1(B_i^-, \R)\cap  i_* \pi_1(B_i^+, \R)$.
\end{itemize}
For $i,j\in I_{\mathrm{h}}$, we write $\cl_i \to \cl_j$ if there exists $n\ge 0$ such that $f^n(B_i^-)\cap B_j^+ \neq\emptyset$.
\end{definition}

Note that \cite[Theorem~B]{G25Cvx1} gives 4 alternative characterizations of this order relation $\to$ (\cite[Proposition 4.18]{G25Cvx1}). 
\bigskip

As $H_1(S,\R) = \bigoplus_{i}V_i$ (Proposition~\ref{Prop:decomposition}), we can endow $H_1(S,\R)$ with a norm 
\[\|v\| = \sup_i \|v_i\|_i,\]
where $v = \sum_i v_i$, with $v_i\in V_i$, and $\|\cdot\|_i$ is a norm on $V_i$. To this norm is associated the distance to a set:
\[d(a,E) = \inf\big\{\|a-e\|\mid e\in E\big\}.\]

As the sets $\rho_i$ span the subspaces $V_i$, there exists a basis $w_1,\dots,w_{2g}$ of $H_1(S,\R)$ that is adapted to the decomposition $H_1(S,\R) = \bigoplus_{i\in I_{\mathrm{h}}} V_i$ such that for any $j$ we have $w_j\in \rho_{i_j}$ for some ${i\in I_{\mathrm{h}}}$. 

As for any $i\in I_{\mathrm h}$ we have $i_*H_1(S_i, \R) = V_i$,  $i_* \pi_1(S_i,\R)\subset i_* \pi_1(B_i^-, \R) \cap i_* \pi_1(B_i^+, \R)$, 
for any $1\le j\le 2g$ there exist two freely homotopic loops $\alpha_j^- \subset B_{i_j}^-$ and $\alpha_j^+ \subset B_{i_j}^+$ such that $[\alpha_j^-], [\alpha_j^+]\in \R w_j$. We suppose that the lifts of these loops to $\wt S$ are simple. 
Let $\gamma_j$ be the geodesic loop freely homotopic to $\alpha_j^-$. 

Let $N_0\in\N$ be such that if a path of $\wt S$ crosses $N_0$ different lifts of $\gamma_j$, then it crosses at least one lift of $\alpha_j^-$ and one lift of $\alpha_j^+$. Set $M = g N_0$.

Let $M'>0$ be such that if fr some $1\le j\le 2g$, some $x\in S$ and some $n\in\N$ one has $\big|[a_x^n]\wedge [\gamma_j]\big|\ge M'$, then for any $t,t'\in [-1, 1]$, any lift of the isotopy path $I_x^{[t,n+t']}$ crosses $M$ different lifts of $\gamma_j$ (the existence of such $M'$ comes from the boundedness of the fundamental domain $D$).

The maps $\cdot \wedge [\gamma_j]$ form a basis of the dual of $H_1(S,\R)$, hence (because the $V_i$ are symplectic) there exists $L>0$ such that for any $i\in I_{\mathrm{h}}$ and any $v_i\in V_i$ such that $\|v_i\|\ge L$, there exists $1\le j\le 2g$ such that $[\gamma_j]\in V_i$ and
\begin{equation}\label{EqBigWedge0}
\big|v_i\wedge [\gamma_j]\big|\ge M'.
\end{equation}

\subsection{Big deviations and connections between classes}

Let $x\in S$ and $n\in\N$.
We can write in a unique way $[a_x^n] = \sum_i v_i$ (see \eqref{Defaxn} for a definition of $a_x^n$), with $v_i\in V_i$. 
Recall that the constant $L$ is defined before \eqref{EqBigWedge0}.

\begin{prop}\label{PropDevImpliesConnec}
Let $I_L = \{i\in I_{\mathrm{h}} \mid \|v_i\|\ge L\}$. Then there exists an ordering $i_1\prec \dots \prec i_\ell$ of $I_L$ such that $\cl_{i_1}\to\cdots \to \cl_{i_\ell}$. 
\end{prop}

\begin{rem}
In fact we prove that the conclusion of the proposition is valid when replacing $I_L$ with the following set (involving $\Lambda_i$ defined in \cite[(4)]{G25Cvx1}):
\[I'_M = \Big\{i\in I_{\mathrm{h}}\ \Big|\ \exists j_i : \big|a_x^n\wedge[\gamma_{j_i}]\big|\ge M, \gamma_{j_i}\subset \Lambda_i\Big\}. \]
\end{rem}

\begin{proof}
By \eqref{EqBigWedge0}, for any $i\in I_L$ there exists $j_i$ such that $|v_i\wedge [\gamma_{j_i}]|\ge M'$ and $[\gamma_{j_i}]\in V_i$. Because the linear subspaces $V_k$ are pairwise orthogonal for $\wedge$, this implies that $|a_x^n \wedge [\gamma_{j_i}]|\ge M'$. The condition $a_x^n \wedge [\gamma_{j_i}]|\ge M'$ implies that for any $i\in I_L$, any lift of the isotopy path $I_x^{[0,n]}$ to $\wt S$ crosses at least $M$ different lifts of $\gamma_{j_i}$, denoted $\wt\gamma_{j_i}^1,\dots, \wt\gamma_{j_i}^{M}$


We now consider all the $i\in I_L$. Denote $\ell = \card(I_L)$. 
Define $t_1\in(0,1]$ as the smallest time $t\in [0,n]$ satisfying: there exists $i\in I_L$ such that the path $I_x^{[0,t]}$ crosses $\lfloor M/\ell\rfloor$ lifts among $\wt\gamma_{j_i}^1,\dots, \wt\gamma_{j_i}^{M}$. We call $i_1$ such an $i\in I_L$.
As a consequence, for any $i\in I_L\setminus\{i_1\}$, the path $I_x^{[t_1,1]}$ crosses at least $M-\lfloor M/\ell\rfloor$ lifts among $\wt\gamma_{j_i}^1,\dots, \wt\gamma_{j_i}^{M}$. 

We can repeat the previous argument: define $t_2\in(t_1,1]$ as the smallest time $t\in [t_1,1]$ satisfying: there exists $i\in I_L\setminus\{i_1\}$ such that the path $I_x^{[t_1,t]}$ crosses $\lfloor M/\ell\rfloor$ lifts among $\wt\gamma_{j_i}^1,\dots, \wt\gamma_{j_i}^{M}$. We call $i_2$ such an $i\in I_L\setminus\{i_1\}$. Iterating this process, we obtain $t_0 = 0<t_1<\dots<t_\ell\le 1$ and $I_L = \{i_1,\dots,i_\ell\}$ such that for any $1\le k\le\ell$, the path $I_x^{[t_{k-1},t_{k}]}$ crosses $\lfloor M/\ell\rfloor$ lifts among $\wt\gamma_{j_{i_k}}^1,\dots, \wt\gamma_{j_{i_k}}^{M}$. 

Because $\ell\le g$ and the definition $M = gN_0$ of $M$, we have $\lfloor M/\ell\rfloor\ge N_0$. By the definition of $N_0$, this implies that for any $1\le k\le \ell-2$ we have that any lift of $I_x^{[\lfloor t_{k-1}\rfloor,\lfloor t_{k+1}\rfloor]}$ crosses first a lift of $\alpha_{i_k}^-$ and then a lift of $\alpha_{i_{k+1}}^+$. This implies the existence of $m\ge 0$ such that $f^m(B_{i_k}^-) \cap B_{i_{k+1}}^+ \neq \emptyset$, hence that 
$\cl_{i_k}\to \cl_{i_{k+1}}$. 
\end{proof}

For $E$ a set and $R>0$, denote $B_R(E) = \{x\mid d(x, E)<R\}$.

\begin{prop}\label{TheoEquiConnec1}
Let $f\in \Homeo_0(S)$ with big rotation set. Then there exists $L>0$ such that the following conditions are equivalent:
\begin{enumerate}[label=(\roman*), leftmargin=1cm]
\item $\cl_i\to \cl_j$ or $\cl_j\to \cl_i$;
\item $\conv(\rho_i, \rho_j)\subset \rot(f)$;
\item there exist $x\in S$ and $n\in\N$ such that 
\[\sum_{i=0}^{n-1} [a_{f^i(x)}] \in \conv(n\rho_i, n\rho_j)\setminus B_L(n\rho_i\cup n\rho_j).\]
\end{enumerate}
\end{prop}

Note that assertions of this proposition are equivalent to say that in the graph $G$, the strong connected components corresponding to $\cl_i$ and $\cl_j$ are in the same connected component.

\begin{proof}[Proof of Proposition~\ref{TheoEquiConnec1}]
\smallskip\noindent\textbf{\textit{(ii) $\implies$ (iii)}:} This implication is trivial.

\smallskip\noindent\textbf{\textit{(iii) $\implies$ (i)}:} This is Proposition~\ref{PropDevImpliesConnec}.

\smallskip\noindent\textbf{\textit{(i) $\implies$ (ii)}:} This follows from \cite[Corollary~4.22]{G25Cvx1}.
\end{proof}

\subsection{Boundedness of deviations with respect to $\conv(\rot(f))$}


As the sets $\rho_i$ span the subspaces $V_i$, there exists a basis $w_1,\dots,w_{2g}$ of $H_1(S,\R)$ that is adapted to the decomposition $H_1(S,\R) = \bigoplus_{i\in I_{\mathrm{h}}} V_i$ such that for any $j$ we have $w_j\in \bigcup_{i\in I_{\mathrm{h}}} \rho_i$. 
As the rotation vectors of periodic orbits are dense in each $\rho_i$, we can suppose that any $w_j$ is the rotation vector of a periodic orbit $z_j$. One can moreover suppose that the tracking geodesic $\gamma_j$ of each $z_j$ is not simple (\cite[Proposition~2.5]{G25Cvx1}). We denote $T_j$ the primitive deck transformation associated to $\gamma_j$.


We prove the following weak version of Theorem~\ref{TheoBndedHomo}, that is already an improvement of \cite[Th\'eor\`eme G]{lellouch} (Theorem~\ref{TheoBndedHomoLellou}).

\begin{prop}\label{PropBnddDevConvHull}
Let $f\in \Homeo_0(f)$ be such that $\inte\big(\conv(\rot(f))\big)\neq\emptyset$. 
Then there exists $L>0$ such that for any $x\in S$ and any $n\in\N$, we have 
\[d\big([a_x^n],\, n\conv(\rot(f))  \big) \le L.\]
\end{prop}

\begin{proof}
Let us apply \cite[Theorem~2.7]{G25Cvx1} and \cite[Lemma~2.8]{G25Cvx1} to all the geodesics $\gamma_j$. This gives us a constant $N'_0\in\N$. Set $M = g N'_0$.

The maps $\cdot \wedge [\gamma_j]$ form a basis of the dual of $H_1(S,\R)$, hence (because the $V_i$ are symplectic) there exists $L>0$ such that for any $i\in I_{\mathrm{h}}$ and any $v_i\in V_i$ such that $\|v_i\|\ge L$, there exists $1\le j\le 2g$ such that $[\gamma_j]\in V_i$ and
\begin{equation}\label{EqBigWedge}
\big|v_i\wedge [\gamma_j]\big|\ge M.
\end{equation}
Using the boundedness of the fundamental domain $D$, we deduce the following fact.

\begin{fact}\label{FactBndedNorm}
There exists $C>0$ such that for any $x\in S$ and $n\in\N$ such that $I_x^{[0,n]}$ crosses at most two lifts of the $(\wt\gamma_j)_{1\le j\le 2g}$, we have $\|[a_x^n]\| \le C$.
\end{fact}

As a consequence, for any $x\in S$ and $n\in\N$ such that $I_{\wt x}^{[0,n]}$ crosses geometrically at most $k$ of the $(\wt\gamma_j)_{1\le j\le 2g}$, we have $\|[I_{\wt x}^{[0,n]}]\| \le kC$.
\bigskip

As a first step, we throw away some intersections to get pairwise disjoint lifts.
Let us apply \cite[Lemma~2.8]{G25Cvx1} to all the closed geodesics $\gamma_1,\dots,\gamma_{2g}$, and $M_0=5$; it gives a constant $N_0$ adapted to all the loops $\gamma_1,\dots,\gamma_{2g}$.

We denote by $0\le t_1\le t_2\le \dots\le t_k$ the intersections between $I_{\wt x}^{[0,n]}$ and all the lifts of the closed geodesics $\gamma_1,\dots,\gamma_{2g}$ (with the convention that for $i\neq i'$, the times $t_i$ and $t_{i'}$ correspond to different geodesics of $\wt S$). Let us group these intersections by sets of at least $2g N_0$ (but at most $4g N_0$) consecutive elements: $E_1 = \{t_1,\dots,t_{2gN_0}\}$, $E_2 = \{t_{2gN_0+1},\dots, t_{4gN_0}\}$, \dots{}
By the pigeonhole principle, on each of these sets $(E_i)_{1\le i\le Q}$ there are at least $N_0$ times corresponding to a single geodesic $\gamma_{j_i}$: there are $t'_{i,1}\le\dots \le t'_{i,N_0}\in E_i$ such that $I_{ x}^{t'_{i,\ell}}\in \gamma_{j_i}$, and that the $I_{\wt x}^{t'_{i,\ell}}$ belong to pairwise different lifts of $\gamma_{j_i}$. 
Applying \cite[Lemma~2.8]{G25Cvx1}, there exist $t''_{i,1}\le \dots \le t''_{i,5}\in E_i$ such that $I_{\wt x}^{t''_{i,\ell}}\in \gamma_{j_i}$, and that the $I_{\wt x}^{t''_{i,\ell}}\in \gamma_{j_i}$ belong to lifts of $\gamma_{j_i}$ that are pairwise disjoint and have the same orientation. 
\bigskip

Let us define some subpaths $\beta_k$ of $I_{\wt x}^{[0,n]}$ by induction (corresponding to pairwise disjoint sub-intervals of the interval of definition $[0,n]$ of $I_{\wt x}^{[0,n]}$). Set $i_0=0$.
Let $j'_1 = j_1$ be the index of the geodesic corresponding to $E_1$.
Consider the largest index $i_1$ such that $j_{i_1} = j'_1$ (this is the last time the geodesic corresponding to $E_{i_1}$ is the same as the geodesic corresponding to $E_1$). 
We then set $\beta_1 = I_{\wt x}^{[t''_{1,2}, t''_{i_1,3}]}$.

Now, if $i_1<Q$ (\emph{i.e.}~if $j'_2$ can be defined), define $j'_2 = j_{i_1+1}$. 
Consider the largest index $i_2$ such that $j_{i_2} = j'_2$ and set $\beta_2 = I_{\wt x}^{[t''_{i_1+1,2}, t''_{i_2,3}]}$.

One can continue by induction to build subpaths $(\beta_k)_{1\le k\le k_0}$ of $I_{\wt x}^{[0,n]}$. Note that by construction $k_0\le 2g$.

Now, for each $k$ one can apply \cite[Theorem~2.7]{G25Cvx1} to the path $I_{\wt x}^{[t''_{i_{k-1}+1,1}, t''_{i_k,5}]}$, the geodesic $\gamma_{j'_k}$ and to the intersection times $t''_{i_{k-1}+1,1}, t''_{i_{k-1}+1,2}, t''_{i_k,3}, t''_{i_k,4}, t''_{i_k,5}$. This theorem implies that there exists $m_1 > 0$, $d_0>0$ and for each $k$, there exists a periodic point $z_k$ of period $q_k\le  t''_{i_k,5} - t''_{i_{k-1}+1,1} + m_1$ and such that $\wt f^{q_k}(\wt z_k) = R_k \wt z_k$, where $R_k\in \G$ satisfies:
\begin{equation}\label{EqDistBetaRk}
\big\|[\beta_k] - [R_k]\big\| \le d_0.
\end{equation}

The complement of the intervals $[t''_{i_{k-1}+1,2}, t''_{i_k,3}]$ in $[0,n]$ is made of at most $2g+2$ intervals $J_1,\dots, J_p$, for each of them the isotopy path $I_{\wt x}^{J_\ell}$ intersects at most $8gN_0$ lifts of the $\gamma_i$. By Fact~\ref{FactBndedNorm}, these paths satisfy $\|[I_{\wt x}^{J_\ell}]\| \le 4gN_0C$. This implies that 
\begin{equation}\label{EqDistAlphaBeta}
\Big\|[a_x^n] - \sum_{k=1}^{k_0} [\beta_k]\Big\|\le (2g+1)4gN_0C \le 12g^2N_0C.
\end{equation}
Moreover, 
\[0\le \sum_{k=1}^{k_0} q_k \le \sum_{k=1}^{k_0}\big(t''_{i_k,5} - t''_{i_{k-1}+1,1} + m_1\big) \le n+2gm_1.\]
Combined with
\[\sum_{k=1}^{k_0} [R_k] = n\underbrace{\left(\sum_{k=1}^{k_0}\frac{q_k}{\sum_j q_j} \frac{[R_k]}{q_k}\right)}_{\in \conv(\rot(f))} + \Big(\sum_j q_j-n\Big)\underbrace{\left(\sum_{k=1}^{k_0}\frac{q_k}{\sum_j q_j} \frac{[R_k]}{q_k}\right)}_{\in \conv(\rot(f))},\]
this implies that (recall that $0\in n\conv(\rot(f))$)
\[d\left(\sum_{k=1}^{k_0} [R_k], n\conv(\rot(f))\right) \le 2gm_1 \diam\big(\conv(\rot(f))\big).\]
Using \eqref{EqDistBetaRk} and \eqref{EqDistAlphaBeta}, this gives
\[d\big([a_x^n] , n\conv(\rot(f))\big) \le 2gm_1 \diam\big(\conv(\rot(f))\big) + 12g^2N_0C + 2g d_0.\]
\end{proof}

\section{Bounds on the number of convex sets and proof of Theorems~\ref{MainTheoUnionConvex} and \ref{MainTheoConvex}, and Propositions~\ref{CoroTheoConvexTrans} and \ref{CoroTheoConvexNW}}\label{SecRefRefRef}

\subsection{The graph $G$ contains all the rotational information on $f$}

Let us recall that the rotation set of a graph of rectangles is defined in \cite[Definition~3.9]{G25Cvx1}. The oriented graph $G$ was built in \cite[Subsection~4.1]{G25Cvx1} to characterize part of the rotational behaviour of $f$; let us summarize the main properties of this graph. 
Its vertices are rectangles of $\wt S$ as defined in \cite[Definition~3.1]{G25Cvx1}, and we put an oriented edge $R_1\to R_2$ if there exists $n\ge 0$ and a deck transformation $T\in\G$ such that the intersection $\wt f^{n}(R_1)\cap TR_{2}$ is Markovian (see \cite[Definition~3.8]{G25Cvx1}). 
The collection of rectangles $(R_\omega)_{\omega\in\Omega}$ is chosen so that each $R_\omega$ is a rotational horseshoe associated to a deck transformation $T_\omega\in\G$ (see \cite[Definition~3.7]{G25Cvx1}), so that the family of associated rotation vectors is dense in $\bigcup_{i\in I_{\mathrm h}}\rho_i$.

The following proposition asserts that when $f$ has big rotation set, the graph $G$ contains all the rotational behaviour of $f$.

\begin{prop}\label{PropRotFRotG}
If $f\in\Homeo_0(S)$ has big rotation set, then $\rot(f) = \overline{\rot(G)}$.
\end{prop}

\begin{proof}
The inclusion $\rot(f) \supset \overline{\rot(G)}$ is \cite[Proposition~3.10]{G25Cvx1}.

Let us prove the remaining inclusion $\rot(f) \subset \overline{\rot(G)}$. Let $v\in \rot(f)$.
Because of the decomposition $H_1(S, \R) = \bigoplus_{i\in I_{\mathrm{h}}} V_i$, one can write $v = \sum_{i\in I_v} v_i$ in a unique way, where for any $i\in I_v\subset I_{\mathrm h}$ one has $v_i \in V_i\setminus\{0\}$. The vector $v$ is realised by orbit segments: there exists $(x_k)_k\in S^\N$ and $n_k\to +\infty$ such that 
\[\lim\limits_{k \to +\infty}\frac{1}{n_k}\sum_{i=0}^{n_k-1}[a_{f^{i}(x_k)}] = v.\]
Consider the constant $L$ given by Proposition~\ref{PropDevImpliesConnec}. If $k$ is large enough, then for any $i\in I_v$ we have $\|n_k v_i\|\ge L$. Applying Proposition~\ref{PropDevImpliesConnec} we get that there exists an ordering $i_1\prec \dots \prec i_\ell$ of $I_v$ such that $\cl_{i_1}\overset\F\to\cdots \overset\F\to \cl_{i_\ell}$. By \cite[Lemma~4.10]{G25Cvx1}, this implies that 
\[\cl_{i_1}\to\cdots \to \cl_{i_\ell}.\]
By \cite[Proposition~4.21]{G25Cvx1}, we deduce that
\begin{equation}\label{EqConnecClass}
\conv\bigg(\bigcup_{k=1}^\ell\overline{\rho_{i_k}}\bigg) \subset \overline{\rot(G)}.
\end{equation}

Moreover, as $v\in \rot(f)\cap \bigoplus _{k=1}^\ell V_{i_k}$, we have that 
\[v\in \conv\bigg(\bigcup_{k=1}^\ell\overline{\rho_{i_k}}\bigg) \subset \overline{\rot(G)};\] 
combined with \eqref{EqConnecClass} this implies that $v\in \overline{\rot(G)}$.
%
\end{proof}

\subsection{The graph $\Tr$ and associated open invariant sets}\label{SubSecGrafT}

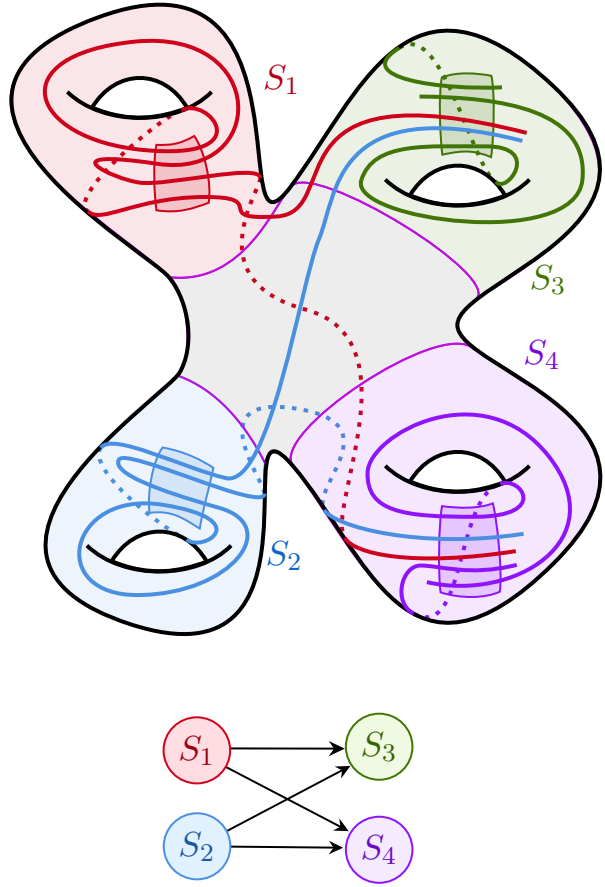
\begin{figure}[!t]
\begin{center}

\tikzset{every picture/.style={line width=0.75pt}} 
\vspace{-25pt}
\begin{tikzpicture}[x=0.75pt,y=0.75pt,yscale=-1.2,xscale=1.2]

\draw [color={rgb, 255:red, 208; green, 2; blue, 27 }  ,draw opacity=1 ][line width=1.5]  [dash pattern={on 1.69pt off 2.76pt}]  (337.99,98.39) .. controls (339.05,101.98) and (310.65,140.89) .. (356.82,153.56) .. controls (402.99,166.22) and (364.99,237.56) .. (373.15,249.22) ;
\draw [color={rgb, 255:red, 74; green, 144; blue, 226 }  ,draw opacity=1 ][line width=1.5]  [dash pattern={on 1.69pt off 2.76pt}]  (339.82,230.72) .. controls (339.82,225.72) and (311.99,188.89) .. (352.65,195.06) .. controls (393.32,201.22) and (358.65,227.89) .. (364.65,237.56) ;
\draw [color={rgb, 255:red, 208; green, 2; blue, 27 }  ,draw opacity=1 ][line width=1.5]  [dash pattern={on 1.69pt off 2.76pt}]  (305.13,70.02) .. controls (295,74.02) and (257.28,106.43) .. (266.4,112.43) ;
\draw  [color={rgb, 255:red, 189; green, 16; blue, 224 }  ,draw opacity=1 ][fill={rgb, 255:red, 208; green, 2; blue, 27 }  ,fill opacity=0.1 ] (270.68,31.05) .. controls (361.88,-3.73) and (323.13,141.52) .. (350.7,102.01) .. controls (343.88,109.77) and (331.13,143.52) .. (300,140) .. controls (289.13,125.34) and (178.25,65.77) .. (270.68,31.05) -- cycle ;
\draw [color={rgb, 255:red, 74; green, 144; blue, 226 }  ,draw opacity=1 ][line width=1.5]  [dash pattern={on 1.69pt off 2.76pt}]  (271.67,211.52) .. controls (265.89,219.74) and (293.65,243.17) .. (307.39,249.96) ;
\draw  [color={rgb, 255:red, 74; green, 144; blue, 226 }  ,draw opacity=1 ][fill={rgb, 255:red, 74; green, 144; blue, 226 }  ,fill opacity=0.1 ] (283.96,285.47) .. controls (353.22,307.82) and (336.93,229.54) .. (340.03,218.06) .. controls (334.79,207.43) and (322.47,187.61) .. (303.85,181.34) .. controls (294.03,193.7) and (198.36,257.82) .. (283.96,285.47) -- cycle ;
\draw  [color={rgb, 255:red, 189; green, 16; blue, 224 }  ,draw opacity=1 ][fill={rgb, 255:red, 74; green, 74; blue, 74 }  ,fill opacity=0.1 ] (350.7,102.01) .. controls (360.02,90.63) and (440.84,140.81) .. (427.33,149.62) .. controls (421.93,155.02) and (420.7,154.01) .. (420,160) .. controls (420.7,166.3) and (424.73,168.02) .. (426.73,170.02) .. controls (413.02,157.15) and (340.22,206.35) .. (351.84,219.44) .. controls (349.95,215.61) and (342.03,208.92) .. (340.03,218.06) .. controls (331.75,200.92) and (319.49,187.16) .. (303.85,181.34) .. controls (311.67,171.52) and (306.22,143.16) .. (300,140) .. controls (331.63,143.7) and (344.13,108.87) .. (350.7,102.01) -- cycle ;
\draw [color={rgb, 255:red, 144; green, 19; blue, 254 }  ,draw opacity=1 ][line width=1.5]  [dash pattern={on 1.69pt off 2.76pt}]  (401.78,281.01) .. controls (423.62,293.01) and (419.05,231.88) .. (437.87,226.65) ;
\draw  [color={rgb, 255:red, 189; green, 16; blue, 224 }  ,draw opacity=1 ][fill={rgb, 255:red, 144; green, 19; blue, 254 }  ,fill opacity=0.1 ] (426.74,170.02) .. controls (413.02,157.02) and (340.62,205.95) .. (351.85,219.44) .. controls (380.01,256.6) and (401.75,319.77) .. (460,260) .. controls (505.44,214.32) and (460.87,190.89) .. (426.74,170.02) -- cycle ;
\draw [color={rgb, 255:red, 65; green, 117; blue, 5 }  ,draw opacity=1 ][line width=1.5]  [dash pattern={on 1.69pt off 2.76pt}]  (395.87,43.9) .. controls (413.8,31.78) and (425.77,93.55) .. (435.99,99.99) ;
\draw  [color={rgb, 255:red, 189; green, 16; blue, 224 }  ,draw opacity=1 ][fill={rgb, 255:red, 65; green, 117; blue, 5 }  ,fill opacity=0.1 ] (350.7,102.01) .. controls (360.29,90.63) and (440.66,140.63) .. (427.33,149.62) .. controls (470.43,122.1) and (499.99,101.88) .. (460,60) .. controls (400.8,2.4) and (377.86,64.4) .. (350.7,102.01) -- cycle ;
\draw  [line width=1.5]  (300,140) .. controls (290.95,128.79) and (178.28,65.05) .. (270.68,31.05) .. controls (363.08,-2.95) and (322.37,140.13) .. (350.7,102.01) .. controls (379.03,63.9) and (399.76,2.88) .. (460,60) .. controls (520.24,117.12) and (420.28,140.94) .. (420,160) .. controls (419.72,179.06) and (520.82,199.39) .. (460,260) .. controls (399.18,320.61) and (382.54,255.97) .. (351.85,219.44) .. controls (321.15,182.92) and (371.3,314.13) .. (283.97,285.47) .. controls (196.63,256.8) and (298.22,190.43) .. (304.4,180.07) .. controls (310.58,169.7) and (309.05,151.21) .. (300,140) -- cycle ;
\draw  [fill={rgb, 255:red, 255; green, 255; blue, 255 }  ,fill opacity=1 ][line width=1.5]  (439.89,105.85) .. controls (428.41,109.97) and (416.62,112.39) .. (400.3,105.66) .. controls (410.28,89.13) and (430.17,89.03) .. (439.89,105.85) -- cycle ;
\draw [line width=1.5]    (390,99.08) .. controls (405.08,113.81) and (436.08,113.47) .. (450,99.08) ;

\draw  [fill={rgb, 255:red, 255; green, 255; blue, 255 }  ,fill opacity=1 ][line width=1.5]  (439.89,225.85) .. controls (428.41,229.97) and (416.62,232.39) .. (400.3,225.66) .. controls (410.28,209.13) and (430.17,209.03) .. (439.89,225.85) -- cycle ;
\draw [line width=1.5]    (390,219.08) .. controls (405.08,233.81) and (436.08,233.47) .. (450,219.08) ;

\draw  [color={rgb, 255:red, 65; green, 117; blue, 5 }  ,draw opacity=1 ][fill={rgb, 255:red, 65; green, 117; blue, 5 }  ,fill opacity=0.15 ] (413.55,56.29) .. controls (419.89,53.73) and (428.12,54.49) .. (434.12,55.57) .. controls (435.66,61.26) and (435.29,81.15) .. (434.02,89.67) .. controls (427.77,87.17) and (423.52,86.42) .. (413.77,88.67) .. controls (412.69,81.25) and (412.16,62.7) .. (413.55,56.29) -- cycle ;
\draw [color={rgb, 255:red, 65; green, 117; blue, 5 }  ,draw opacity=1 ][line width=1.5]    (404.26,64.29) .. controls (438.26,64.57) and (475.1,69.77) .. (471.99,99.33) .. controls (468.88,128.88) and (375.04,115.91) .. (380.02,99.42) .. controls (385,82.92) and (408.02,85.28) .. (431.77,85.17) .. controls (455.53,85.05) and (440.43,105.55) .. (435.99,99.99) ;
\draw [color={rgb, 255:red, 65; green, 117; blue, 5 }  ,draw opacity=1 ][line width=1.5]    (438.55,60.43) .. controls (413.83,61.57) and (374.99,56.65) .. (395.87,43.9) ;
\draw  [color={rgb, 255:red, 144; green, 19; blue, 254 }  ,draw opacity=1 ][fill={rgb, 255:red, 144; green, 19; blue, 254 }  ,fill opacity=0.15 ] (413.64,235.97) .. controls (420.44,236.95) and (429.06,237.11) .. (434.75,234.8) .. controls (436.29,240.49) and (438.91,263.31) .. (437.64,271.83) .. controls (431.94,274.59) and (421.64,274.9) .. (413.02,272.13) .. controls (411.94,264.71) and (412.25,242.39) .. (413.64,235.97) -- cycle ;
\draw [color={rgb, 255:red, 144; green, 19; blue, 254 }  ,draw opacity=1 ][line width=1.5]    (407.33,267.83) .. controls (443.48,276.13) and (485.41,250.34) .. (458.02,215.11) .. controls (430.64,179.88) and (382.94,205.88) .. (382.64,222.03) .. controls (382.33,238.18) and (396.87,240.59) .. (427.64,240.44) .. controls (458.41,240.29) and (446.33,223.26) .. (437.87,226.65) ;
\draw [color={rgb, 255:red, 144; green, 19; blue, 254 }  ,draw opacity=1 ][line width=1.5]    (401.78,281.01) .. controls (393.07,276.54) and (396.68,261.03) .. (404.54,260.88) .. controls (412.4,260.72) and (424.64,265.66) .. (444.71,260.59) ;
\draw  [fill={rgb, 255:red, 255; green, 255; blue, 255 }  ,fill opacity=1 ][line width=1.5]  (307.64,69.1) .. controls (296.16,73.22) and (284.37,75.64) .. (268.05,68.91) .. controls (278.03,52.38) and (297.92,52.28) .. (307.64,69.1) -- cycle ;
\draw [line width=1.5]    (257.75,62.33) .. controls (272.83,77.06) and (303.83,76.72) .. (317.75,62.33) ;

\draw  [color={rgb, 255:red, 208; green, 2; blue, 27 }  ,draw opacity=1 ][fill={rgb, 255:red, 208; green, 2; blue, 27 }  ,fill opacity=0.15 ] (294.82,84.53) .. controls (301.03,84.76) and (304.53,84.25) .. (310.53,80.82) .. controls (314.82,86.82) and (317.92,100.63) .. (316.53,108.82) .. controls (309.81,111.24) and (304.25,113.39) .. (294.82,111.96) .. controls (293.74,104.54) and (293.43,90.95) .. (294.82,84.53) -- cycle ;
\draw [color={rgb, 255:red, 208; green, 2; blue, 27 }  ,draw opacity=1 ][line width=1.5]    (337.99,98.39) .. controls (337.82,92.06) and (299.25,101.88) .. (287.99,100.39) .. controls (276.72,98.9) and (248.97,85.54) .. (295.06,92.47) .. controls (341.15,99.39) and (341.97,35.68) .. (285.88,41.27) .. controls (229.78,46.86) and (241.32,82.22) .. (289.65,86.72) .. controls (337.99,91.22) and (316.88,64.02) .. (305.13,70.02) ;
\draw [color={rgb, 255:red, 208; green, 2; blue, 27 }  ,draw opacity=1 ][line width=1.5]    (448.81,79.52) .. controls (339.31,51.52) and (373.4,112.77) .. (344.83,114.57) .. controls (316.27,116.37) and (342.98,106.19) .. (314.49,106.39) .. controls (285.99,106.59) and (271.78,117.93) .. (266.4,112.43) ;
\draw  [color={rgb, 255:red, 0; green, 0; blue, 0 }  ,draw opacity=1 ][fill={rgb, 255:red, 255; green, 255; blue, 255 }  ,fill opacity=1 ][line width=1.5]  (315.78,259.25) .. controls (304.3,263.37) and (292.51,265.79) .. (276.19,259.05) .. controls (286.17,242.52) and (306.06,242.42) .. (315.78,259.25) -- cycle ;
\draw [color={rgb, 255:red, 0; green, 0; blue, 0 }  ,draw opacity=1 ][line width=1.5]    (265.89,252.48) .. controls (280.98,267.2) and (311.98,266.86) .. (325.89,252.48) ;

\draw  [color={rgb, 255:red, 74; green, 144; blue, 226 }  ,draw opacity=1 ][fill={rgb, 255:red, 74; green, 144; blue, 226 }  ,fill opacity=0.15 ] (297.39,211.39) .. controls (306.82,213.1) and (313.96,216.82) .. (318.53,219.96) .. controls (316.53,229.1) and (314.2,237.2) .. (312.82,245.39) .. controls (306.53,241.39) and (301.1,237.39) .. (291.96,238.82) .. controls (290.88,231.4) and (293.1,217.1) .. (297.39,211.39) -- cycle ;
\draw [color={rgb, 255:red, 74; green, 144; blue, 226 }  ,draw opacity=1 ][line width=1.5]    (339.82,230.72) .. controls (339.82,239.06) and (281.82,210.89) .. (279.15,216.56) .. controls (276.49,222.22) and (294.82,229.22) .. (320.99,236.72) .. controls (347.15,244.22) and (326.56,272.73) .. (287.56,273.25) .. controls (248.56,273.76) and (258.21,229.17) .. (297.32,235.39) .. controls (336.43,241.61) and (313.96,255.63) .. (307.39,249.96) ;
\draw [color={rgb, 255:red, 74; green, 144; blue, 226 }  ,draw opacity=1 ][line width=1.5]    (271.67,211.52) .. controls (276.49,204.72) and (305.19,219.76) .. (323.65,224.22) .. controls (342.11,228.68) and (354.8,142.31) .. (362.43,124.04) .. controls (370.06,105.77) and (367.81,62.52) .. (447.04,82.88) ;
\draw [color={rgb, 255:red, 208; green, 2; blue, 27 }  ,draw opacity=1 ][line width=1.5]    (373.15,249.22) .. controls (376.82,256.06) and (401.32,261.56) .. (444.31,254.83) ;
\draw [color={rgb, 255:red, 74; green, 144; blue, 226 }  ,draw opacity=1 ][line width=1.5]    (364.65,237.56) .. controls (368.32,244.39) and (404.56,254.28) .. (447.54,247.56) ;

\draw (448.58,133.4) node [anchor=north west][inner sep=0.75pt]  [color={rgb, 255:red, 65; green, 117; blue, 5 }  ,opacity=1 ,xscale=1.2,yscale=1.2]  {$S_{3}$};
\draw (446.08,179.98) node [anchor=south west] [inner sep=0.75pt]  [color={rgb, 255:red, 118; green, 3; blue, 220 }  ,opacity=1 ,xscale=1.2,yscale=1.2]  {$S_{4}$};
\draw (338.87,256.87) node [anchor=west] [inner sep=0.75pt]  [color={rgb, 255:red, 11; green, 99; blue, 203 }  ,opacity=1 ,xscale=1.2,yscale=1.2]  {$S_{2}$};
\draw (337.73,57.3) node [anchor=west] [inner sep=0.75pt]  [color={rgb, 255:red, 177; green, 1; blue, 21 }  ,opacity=1 ,xscale=1.2,yscale=1.2]  {$S_{1}$};

\end{tikzpicture}

\begin{tikzpicture}[x=0.75pt,y=0.75pt,yscale=-1,xscale=1]

\draw    (298.75,81.32) -- (372.68,120.24) ;
\draw [shift={(375.33,121.64)}, rotate = 207.77] [fill={rgb, 255:red, 0; green, 0; blue, 0 }  ][line width=0.08]  [draw opacity=0] (7.14,-3.43) -- (0,0) -- (7.14,3.43) -- (4.74,0) -- cycle    ;
\draw    (299.59,129.26) -- (369.5,129.94) ;
\draw [shift={(372.5,129.97)}, rotate = 180.56] [fill={rgb, 255:red, 0; green, 0; blue, 0 }  ][line width=0.08]  [draw opacity=0] (7.14,-3.43) -- (0,0) -- (7.14,3.43) -- (4.74,0) -- cycle    ;
\draw    (300.73,128.89) -- (373.52,90.21) ;
\draw [shift={(376.17,88.81)}, rotate = 152.02] [fill={rgb, 255:red, 0; green, 0; blue, 0 }  ][line width=0.08]  [draw opacity=0] (7.14,-3.43) -- (0,0) -- (7.14,3.43) -- (4.74,0) -- cycle    ;
\draw    (370.33,80.13) -- (300,80) ;
\draw [shift={(373.33,80.14)}, rotate = 180.11] [fill={rgb, 255:red, 0; green, 0; blue, 0 }  ][line width=0.08]  [draw opacity=0] (8.04,-3.86) -- (0,0) -- (8.04,3.86) -- (5.34,0) -- cycle    ;

\draw  [color={rgb, 255:red, 208; green, 2; blue, 27 }  ,draw opacity=1 ][fill={rgb, 255:red, 255; green, 221; blue, 227 }  ,fill opacity=1 ]  (299.5, 81) circle [x radius= 16.62, y radius= 16.62]   ;
\draw (299.5,81) node  [color={rgb, 255:red, 152; green, 0; blue, 17 }  ,opacity=1 ,xscale=1.2,yscale=1.2]  {$S_{1}$};
\draw  [color={rgb, 255:red, 74; green, 144; blue, 226 }  ,draw opacity=1 ][fill={rgb, 255:red, 224; green, 241; blue, 255 }  ,fill opacity=1 ]  (299.59, 129.26) circle [x radius= 16.62, y radius= 16.62]   ;
\draw (299.59,129.26) node  [color={rgb, 255:red, 26; green, 87; blue, 156 }  ,opacity=1 ,xscale=1.2,yscale=1.2]  {$S_{2}$};
\draw  [color={rgb, 255:red, 144; green, 19; blue, 254 }  ,draw opacity=1 ][fill={rgb, 255:red, 245; green, 235; blue, 255 }  ,fill opacity=1 ]  (390.5, 131) circle [x radius= 16.62, y radius= 16.62]   ;
\draw (390.5,131) node  [color={rgb, 255:red, 107; green, 7; blue, 193 }  ,opacity=1 ,xscale=1.2,yscale=1.2]  {$S_{4}$};
\draw  [color={rgb, 255:red, 65; green, 117; blue, 5 }  ,draw opacity=1 ][fill={rgb, 255:red, 240; green, 250; blue, 226 }  ,fill opacity=1 ]  (390.5, 79.17) circle [x radius= 16.62, y radius= 16.62]   ;
\draw (390.5,79.17) node  [color={rgb, 255:red, 54; green, 100; blue, 1 }  ,opacity=1 ,xscale=1.2,yscale=1.2]  {$S_{3}$};

\end{tikzpicture}

\caption{An example of a homeomorphism of a surface of genus 4 having 5 horseshoe classes $S_i$, and some horseshoes in these classes, with the graph $\Tr$ associated to it. The images of the rectangles are represented as thick lines, we suppose that all the pictured intersections are Markovian. 
\\
In this case, the ergodic rotation set of $f$ is the union of 4 pieces that are of dimension 2 (their closures are convex sets) and the rotation set of $f$ is the union of 4 convex sets, each of dimension 4. The oriented graph $\Tr$ has no oriented loop but is not a tree.}\label{FigTree}
\end{center}
\end{figure}

Let us adapt the definition of the graph $\Tr$ given in \cite[Subsection~4.4]{G25Cvx1}. An example of such a graph $\Tr$ is given in Figure~\ref{FigTree}.
For a homeomorphism with big rotation set, the vertices of the graph $\Tr$ are the surfaces $S_i$ for $i\in I_{\mathrm h}$; we put an edge $S_i - S_j$ if $i\neq j$ and $\overline S_i\cap \overline S_j\neq \emptyset$ (in other words, if they share a boundary component). 
As the correspondence $\cl_i\leftrightarrow S_i$ is 1 to 1 in $I_{\mathrm h}$, we will sometimes label the vertices of $\Tr$ with the classes $\cl_i$. 

We now orient some of the edges of $\Tr$: we change $S_i - S_j$ by $S_i\to S_j$ if $\cl_i\to \cl_j$.
The orientation of edges of $\Tr$ is well defined because $\to $ is an order relation (\cite[Proposition~4.18]{G25Cvx1}). This also shows that having only the data of relations $\cl_i\to \cl_j$ for $S_i$ and $S_j$ adjacent in the graph $\Tr$ allows to recover the whole relation $\to$.
\bigskip

Proposition~\ref{PropRotFRotG} and \cite[Proposition 4.21]{G25Cvx1} imply directly the following:

\begin{coro}\label{CoroPropRotFRotG}
If $f\in\Homeo_0(S)$ has big rotation set, then
\[\rot(f) = \bigcup_{\mathrm{p} \text{ path in }\Tr}\conv\big(\{\rho_i\}_{i\in \mathrm{p}}\big).\]
\end{coro}

Let us recall and improve the results of \cite[Subsection~4.5]{G25Cvx1}.

The connected components $(G_\alpha)_{\alpha\in\mathcal A}$ of the graph $G$ correspond to the connected components of $\Tr$ (\cite[Remark~4.19]{G25Cvx1}). For any $\alpha\in\mathcal A$, identified with the set of $i\in I_{\mathrm h}$ such that $G_i\subset G_\alpha$, set (the sets $B_i^o$ are defined in \cite[Proposition~4.17]{G25Cvx1}).
\[B_\alpha = \bigcup_{i\in \alpha} B_i^o.\]
By \cite[Proposition~4.17]{G25Cvx1}, this gives a collection $(B_\alpha)_{\alpha\in\mathcal A}$ of pairwise disjoint, connected, essential and filled open sets, satisfying $f(B_\alpha) = B_\alpha$, and such that for any $\alpha\in\mathcal A$ we have
\begin{equation}\label{EqPi1Si}
\left\langle i_* \pi_1(S_i) \mid {i\in \alpha}\right\rangle \subset i_*\pi_1(B_\alpha).
\end{equation}
The set $B_\alpha$ is also of full measure for any $\mu\in\cl_i$ with $i\in\alpha$.
Recall that $H_1(S,\R) = \bigoplus_{i\in I_{\mathrm h}} i_* H_1(S_i, \R)$. 
Recall also that for $i\neq j$, we have $i_*H_1(B_i,\R) \perp i_* H_1(B_j, \R)$ (for the intersection form $\wedge$).
Combined with \eqref{EqPi1Si} this implies that for any $i\in I_{\mathrm h}$ we have 
\[i_* H_1(S_i,\R) = i_* H_1(B_i^o, \R).\]

Denote $K = S\setminus\bigcup_{\alpha\in\mathcal A}B_\alpha$.

\begin{lemma}\label{RotK0}
If $f$ has big rotation set, then $\rot(f|_{K} ) = \{0\}$.
\end{lemma}

\begin{proof}
 Suppose there exist $(x_k)_k\in K$, $n_k\to +\infty$ and $\rho\neq 0$ such that $\lim_{k \to +\infty}\frac{1}{n_k}\sum_{i=0}^{n_k-1}[a_{f^{i}(x_k)}] = \rho$. By Krylov-Bogolyubov procedure, this implies the existence of $\mu\in\M(f)$ such that $\rho(\mu) = \rho$ and (because $K$ is compact and invariant) $\supp(\mu)\subset K$. Using the ergodic decomposition of $\mu$, we get the existence of $\nu\in \Me(f)$ such that $\supp(\nu)\subset K$ and $\rho(\nu)\neq 0$. But all typical points for measures of $\Me(f)$ with nonzero rotation vector belong to the complement of $K$, a contradiction. 
\end{proof}

\begin{rem}
We can see that the set $K$ is included in a finite union (at most $g$ pieces) of pairwise disjoint essential separating sub-surfaces of genus 0 (aka punctured spheres). This can be seen by considering, for any $i$, a skeleton of the open set $B_i$; the union of all these skeletons is made of at most $g$ open surfaces of genus $0$. 
\end{rem}

\subsection{Proofs of the results}

\begin{proof}[Proof of Theorem~\ref{MainTheoUnionConvex}]
Corollary~\ref{CoroPropRotFRotG} asserts that $\rot(f)$ is a union of convex sets containing 0 (as the $\rho_i$ are convex sets containing 0, by Proposition~\ref{Prop:decomposition}), and the number of such convex sets is bounded by the total number of paths in $\Tr$. More precisely, it is bounded by the total number of \emph{maximal} paths in $\Tr$, as if $\mathrm{p}, \mathrm{p}'$ are two paths in $\Tr$ with $\mathrm p'$ a subpath of $\mathrm p$, then $
\conv(\{\rho_i\}_{i\in \mathrm{p}'}) \subset
\conv(\{\rho_i\}_{i\in \mathrm{p}})$. Hence to prove Theorem~\ref{MainTheoUnionConvex} it suffices to bound the number of paths in $\Tr$. 

Note that if a surface $S_i$ of $\Tr$ (with $i\in I_{\mathrm h}$) has 0 genus and is either starting or ending (meaning that either none of the oriented edges of $\Tr$ having $S_i$ as a tip end at $S_i$, or none of the oriented edges of $\Tr$ having $S_i$ as a tip start at $S_i$), then we do not change the rotation set of the graph by removing these oriented edges starting or ending at $S_i$: if $\Tr'$ is the oriented graph obtained from $\Tr$ by removing these oriented edges, then 
\[\rot(f) = \bigcup_{\mathrm{p} \text{ path in }\Tr}\conv(\{\rho_i\}_{i\in \mathrm{p}})
= \bigcup_{\mathrm{p} \text{ path in }\Tr'}\conv(\{\rho_i\}_{i\in \mathrm{p}}).\]

Denote $\mathcal S$ and $\mathcal E$ the starting and ending vertices of $\Tr'$. 
Note that as $\Tr'$ is an oriented graph with no closed loop, any path $\mathrm p$ of $\Tr'$ is determined by its starting and ending vertices; if moreover $\mathrm p$ is maximal (\emph{i.e.}~cannot be extended), then its starting and ending edges belong to respectively $\mathcal S$ and $\mathcal E$. Hence there are at most 
\[\card(\mathcal S\cap \mathcal E) + \card(\mathcal S\setminus \mathcal E)\cdot \card(\mathcal E\setminus \mathcal S)\]
maximal paths in $\Tr'$. Noting that the elements of $\mathcal S$ and $\mathcal E$ are surfaces with positive genus, we deduce that the number of convex sets forming $\rot(f)$ is bounded by
\[\max\big\{n+pq \mid n,p,q\in\N, n+p+q = g \big\}.\] 
Using the method of Lagrange multipliers, one can see that 
\[\max\big\{x+yz \mid x,y,z\in\R_+, x+y+z = g \big\} = \max(g, g^2/4);\] 
we deduce that 
\[\max\big\{n+pq \mid n,p,q\in\N, n+p+q = g \big\} \le \max(g, \lfloor g^2/4\rfloor).\] 
\end{proof}

\begin{proof}[Proof of Proposition~\ref{CoroTheoConvexTrans}]
As in the proof of Theorem~\ref{MainTheoUnionConvex}, we start by noting that Corollary~\ref{CoroPropRotFRotG} asserts that $\rot(f)$ is a union of convex sets containing 0, and the number of such convex sets is bounded by the total number of paths in $\Tr$.

Using the definition~\ref{DefToOpen2} of the relation $\to$ and the fact that $f$ is transitive, we deduce that the graph $\Tr$ is made of a single strongly connected component. Hence, there is a single class $S_i$ (for $i\in I_{\mathrm h}$) and hence $\rot(f) = \overline{\rote(f)}$ is convex. 
\end{proof}

\begin{proof}[Proof of Proposition~\ref{CoroTheoConvexNW}]
As before, Corollary~\ref{CoroPropRotFRotG} implies that $\rot(f)$ is a union of convex sets containing 0, and the number of such convex sets is bounded by the total number of paths in $\Tr$.

Using \cite[Theorem~B]{G25Cvx1}, the relation $\overset *\to$ and the fact that $f$ is nonwandering (more precisely, the equivalent property that the recurrent points are dense), we deduce that the graph $\Tr$ is recurrent. This is equivalent to ask that there is no relation $\to$ between vertices of $\Tr$. Hence
\[\rot(f) = \bigcup_{i\in I_{\mathrm h}} \rho_i.\]

As there are at most $g$ vertices $S_i$ of $\Tr$ with $\rho_i\neq\{0\}$, we deduce that $\rot(f)$ is the union of at most $g$ convex sets. Note that in this case, the linear subspaces $V_i$ these convex sets span are pairwise orthogonal and in direct sum. 
\end{proof}

\begin{proof}[Proof of Proposition~\ref{PropExistBj}]
The existence of the family $B_1,\dots, B_k$ follows from \cite[Proposition~4.17]{G25Cvx1} (using the sets $B_i^o$). Denoting $K$ the complement of the $B_j$, the fact that $\rot(f|_{K} ) = \{0\}$ is a consequence of Lemma~\ref{RotK0}.
\end{proof}

\begin{proof}[Proof of Theorem~\ref{MainTheoConvex}]
First, if $\inte(\conv(\rot(f))) = \emptyset$, then $\rot(f)$ is included in a hyperplane of $H_1(S,\R)$.

From now we suppose that $\inte(\conv(\rot(f))) \neq \emptyset$. This allows to use the proof of Theorem~\ref{MainTheoUnionConvex}.

Let $\Tr'$ be the graph defined in the proof of Theorem~\ref{MainTheoUnionConvex}.
If $\Tr'$ has at least one strict connected component $\Tr'_1$, then both $\Tr'_1$ and its complement in $\Tr'$ contain genus and hence (using Proposition~\ref{Prop:decomposition}) $\rot(f)\subset \rot(\Tr'_1) \cup (\rot(\Tr'_1) )^\perp$ (the $\perp$ is taken for the intersection form $\wedge$). In particular, $\rot(f)$ is included in the union of two linear subspaces of $H_1(S,\R)$ that are of codimension $\ge 2$.

Hence one can suppose that $\Tr'$ is connected. As in the  proof of Theorem~\ref{MainTheoUnionConvex}, denote $\mathcal S$ and $\mathcal E$ the starting and ending vertices of $\Tr'$. 

If $\card(\mathcal S)\ge 2$, then we denote $\mathcal S = \{S_{i_1},\dots, S_{i_\ell}\}$. Hence 
\[\rot(f)\subset \conv\left(\bigcup_{i\in I_{\mathrm h} \setminus \{i_1\}} \rho_i \right) \cup \conv\left(\bigcup_{i\in I_{\mathrm h} \setminus \{i_2\}} \rho_i \right),\]
with $\conv\big(\bigcup_{i\in I_{\mathrm h} \setminus \{i_1\}} \rho_i \big)$ orthogonal to $\rho_{i_1}$ (for $\wedge$) and $\conv\big(\bigcup_{i\in I_{\mathrm h} \setminus \{i_2\}} \rho_i \big)$ orthogonal to $\rho_{i_2}$. 
So $\rot(f)$ is included in the union of $2$ vector subspaces of $H_1(S,\R)$ of codimension $\ge 2$.

Similarly, if $\card(\mathcal E)\ge 2$, then $\rot(f)$ is included in the union of $2$ vector subspaces of $H_1(S,\R)$ of codimension $\ge 2$.

The only remaining case is when $\card(\mathcal S) = \card(\mathcal E) = 1$. In this case, there is a unique maximal path in $\Tr'$ and hence $\rot(f)$ is convex. 
\end{proof}

\subsection{Sharpness of the bound}\label{paragSharpness}

Let us first describe, for any $g\ge 2$, a homeomorphism of the closed surface $S$ of genus $g$ that has big rotation set that is not the union of less than $\lfloor g^2/4 \rfloor$ convex sets.

Consider $g$ separating simple closed geodesics on $S$, pairwise disjoint, such that the complement of their union is made of $g$ sub-surfaces of genus 1 with one boundary component and one sub-surface $S_0$ of genus 0 with $g$ boundary components (as in Figure~\ref{FigExample11}). We denote $S_1, \dots, S_{\lceil g/2\rceil}, S'_1, \dots, S'_{\lfloor g/2\rfloor}$ these $g$ sub-surfaces of genus 1. For any such surface $S_i$ (resp. $S'_i$), consider two primitive deck transformations $T_{i,1}$ and $T_{i,2}$ (resp. $T_{i,1}$ and $T_{i,2}$) that generate a free group and whose axes on $S$ are contained in the surface $S_i$ (resp. $S'_i$). For any such surface $S_i$ (resp. $S'_i$ and $S_0$), consider a rectangle $R_i$ (resp. $R'_i$ and $R_0$) included in this sub-surface.

\begin{figure}[!t]
\begin{center}

\tikzset{every picture/.style={line width=0.75pt}} 
\vspace{-25pt}
\begin{tikzpicture}[x=0.75pt,y=0.75pt,yscale=-1.3,xscale=1.3]

\draw [color={rgb, 255:red, 144; green, 19; blue, 254 }  ,draw opacity=1 ][line width=1.5]  [dash pattern={on 1.69pt off 2.76pt}]  (381.78,261.01) .. controls (403.62,273.01) and (399.05,211.88) .. (417.87,206.65) ;
\draw  [color={rgb, 255:red, 189; green, 16; blue, 224 }  ,draw opacity=1 ][fill={rgb, 255:red, 144; green, 19; blue, 254 }  ,fill opacity=0.1 ] (406.74,150.02) .. controls (393.02,137.02) and (320.62,185.95) .. (331.85,199.44) .. controls (360.01,236.6) and (381.75,299.77) .. (440,240) .. controls (485.44,194.32) and (440.87,170.89) .. (406.74,150.02) -- cycle ;
\draw [color={rgb, 255:red, 245; green, 166; blue, 35 }  ,draw opacity=1 ][line width=1.5]  [dash pattern={on 1.69pt off 2.76pt}]  (329.16,84.71) .. controls (336.93,73.16) and (417.2,122.94) .. (404.08,132.33) ;
\draw  [color={rgb, 255:red, 189; green, 16; blue, 224 }  ,draw opacity=1 ][fill={rgb, 255:red, 245; green, 166; blue, 35 }  ,fill opacity=0.1 ] (330.7,82.01) .. controls (340.02,70.63) and (420.84,120.81) .. (407.34,129.62) .. controls (401.94,135.02) and (400.7,134.01) .. (400,140) .. controls (400.7,146.3) and (404.74,148.02) .. (406.74,150.02) .. controls (393.02,137.15) and (320.22,186.35) .. (331.85,199.44) .. controls (329.96,195.61) and (322.03,188.91) .. (320.03,198.06) .. controls (311.75,180.91) and (291.18,162.91) .. (280,160) .. controls (290.14,159.82) and (289.81,119.92) .. (280,120) .. controls (311.64,123.7) and (324.13,88.87) .. (330.7,82.01) -- cycle ;
\draw [color={rgb, 255:red, 139; green, 87; blue, 42 }  ,draw opacity=1 ][line width=1.5]  [dash pattern={on 1.69pt off 2.76pt}]  (252.88,193.33) .. controls (247.1,201.55) and (273.65,223.17) .. (287.39,229.96) ;
\draw  [color={rgb, 255:red, 189; green, 16; blue, 224 }  ,draw opacity=1 ][fill={rgb, 255:red, 139; green, 87; blue, 42 }  ,fill opacity=0.1 ] (263.97,265.47) .. controls (333.22,287.82) and (316.94,209.53) .. (320.03,198.06) .. controls (314.79,187.43) and (298.63,166.26) .. (280,160) .. controls (294.63,162.43) and (178.37,237.82) .. (263.97,265.47) -- cycle ;
\draw [color={rgb, 255:red, 245; green, 166; blue, 35 }  ,draw opacity=1 ][line width=1.5]  [dash pattern={on 1.69pt off 2.76pt}]  (324.6,193.2) .. controls (338.27,200.73) and (414.43,155.56) .. (402.58,145.71) ;
\draw [color={rgb, 255:red, 245; green, 166; blue, 35 }  ,draw opacity=1 ][line width=1.5]    (324.6,193.2) .. controls (314.49,190.95) and (315.09,146.92) .. (341.9,145.69) .. controls (368.71,144.46) and (399.66,137.4) .. (404.08,132.33) ;
\draw [color={rgb, 255:red, 208; green, 2; blue, 27 }  ,draw opacity=1 ][line width=1.5]  [dash pattern={on 1.69pt off 2.76pt}]  (285.13,50.02) .. controls (275,54.02) and (239.75,82.77) .. (248.88,88.77) ;
\draw  [color={rgb, 255:red, 189; green, 16; blue, 224 }  ,draw opacity=1 ][fill={rgb, 255:red, 208; green, 2; blue, 27 }  ,fill opacity=0.1 ] (250.68,11.05) .. controls (341.88,-23.73) and (303.13,121.52) .. (330.7,82.01) .. controls (323.88,89.77) and (311.13,123.52) .. (280,120) .. controls (300.13,121.27) and (158.25,45.77) .. (250.68,11.05) -- cycle ;
\draw [color={rgb, 255:red, 65; green, 117; blue, 5 }  ,draw opacity=1 ][line width=1.5]  [dash pattern={on 1.69pt off 2.76pt}]  (375.87,23.9) .. controls (393.8,11.78) and (405.77,73.55) .. (415.99,79.99) ;
\draw  [color={rgb, 255:red, 189; green, 16; blue, 224 }  ,draw opacity=1 ][fill={rgb, 255:red, 65; green, 117; blue, 5 }  ,fill opacity=0.1 ] (330.7,82.01) .. controls (340.29,70.63) and (420.66,120.63) .. (407.33,129.62) .. controls (450.43,102.1) and (479.99,81.88) .. (440,40) .. controls (380.8,-17.6) and (357.86,44.4) .. (330.7,82.01) -- cycle ;
\draw [color={rgb, 255:red, 74; green, 144; blue, 226 }  ,draw opacity=1 ][line width=1.5]  [dash pattern={on 1.69pt off 2.76pt}]  (212.55,142.91) .. controls (226.4,146.29) and (243.32,188.42) .. (252.23,178.23) ;
\draw  [color={rgb, 255:red, 189; green, 16; blue, 224 }  ,draw opacity=1 ][fill={rgb, 255:red, 74; green, 144; blue, 226 }  ,fill opacity=0.1 ] (230,90) .. controls (250.06,90.17) and (260.4,119.7) .. (280,120) .. controls (290.13,119.62) and (290.13,159.82) .. (280,160) .. controls (260.67,160.81) and (250.13,189.92) .. (230,190) .. controls (209,189.81) and (170.4,189.41) .. (170,140) .. controls (170.68,90.84) and (210.11,89.99) .. (230,90) -- cycle ;
\draw  [line width=1.5]  (230,90) .. controls (249.83,90.29) and (260.33,119.86) .. (280,120) .. controls (299.67,120.14) and (158.28,45.05) .. (250.68,11.05) .. controls (343.08,-22.95) and (302.37,120.13) .. (330.7,82.01) .. controls (359.03,43.9) and (379.76,-17.12) .. (440,40) .. controls (500.24,97.12) and (400.28,120.94) .. (400,140) .. controls (399.72,159.06) and (500.82,179.39) .. (440,240) .. controls (379.18,300.61) and (362.54,235.97) .. (331.85,199.44) .. controls (301.15,162.92) and (351.3,294.13) .. (263.97,265.47) .. controls (176.63,236.8) and (298.06,159.94) .. (280,160) .. controls (261.94,160.06) and (250.38,190.17) .. (230,190) .. controls (209.63,189.83) and (170.67,188.71) .. (170,140) .. controls (169.33,91.29) and (210.17,89.71) .. (230,90) -- cycle ;
\draw  [fill={rgb, 255:red, 255; green, 255; blue, 255 }  ,fill opacity=1 ][line width=1.5]  (252.14,143.1) .. controls (240.66,147.22) and (228.87,149.64) .. (212.55,142.91) .. controls (222.53,126.38) and (242.42,126.28) .. (252.14,143.1) -- cycle ;
\draw [line width=1.5]    (202.25,136.33) .. controls (217.33,151.06) and (248.33,150.72) .. (262.25,136.33) ;

\draw  [fill={rgb, 255:red, 255; green, 255; blue, 255 }  ,fill opacity=1 ][line width=1.5]  (419.89,85.85) .. controls (408.41,89.97) and (396.62,92.39) .. (380.3,85.66) .. controls (390.28,69.13) and (410.17,69.03) .. (419.89,85.85) -- cycle ;
\draw [line width=1.5]    (370,79.08) .. controls (385.08,93.81) and (416.08,93.47) .. (430,79.08) ;

\draw  [fill={rgb, 255:red, 255; green, 255; blue, 255 }  ,fill opacity=1 ][line width=1.5]  (419.89,205.85) .. controls (408.41,209.97) and (396.62,212.39) .. (380.3,205.66) .. controls (390.28,189.13) and (410.17,189.03) .. (419.89,205.85) -- cycle ;
\draw [line width=1.5]    (370,199.08) .. controls (385.08,213.81) and (416.08,213.47) .. (430,199.08) ;

\draw  [color={rgb, 255:red, 74; green, 144; blue, 226 }  ,draw opacity=1 ][fill={rgb, 255:red, 74; green, 144; blue, 226 }  ,fill opacity=0.15 ] (214.82,152.95) .. controls (221.43,154.92) and (228.66,154.45) .. (234.82,152.95) .. controls (236.58,158.92) and (238.47,170.23) .. (237.08,178.42) .. controls (230.36,180.83) and (222.71,180.1) .. (214.25,178.25) .. controls (213.17,170.83) and (213.43,159.37) .. (214.82,152.95) -- cycle ;
\draw [color={rgb, 255:red, 74; green, 144; blue, 226 }  ,draw opacity=1 ][line width=1.5]    (209.51,156.42) .. controls (223.17,163.42) and (281.84,159.62) .. (266.96,128.23) .. controls (252.08,96.84) and (182.04,106.57) .. (184.6,137.87) .. controls (187.15,169.17) and (206.37,175.85) .. (227.56,174.67) .. controls (248.76,173.48) and (256.73,173.23) .. (252.23,178.23) ;
\draw  [color={rgb, 255:red, 65; green, 117; blue, 5 }  ,draw opacity=1 ][fill={rgb, 255:red, 65; green, 117; blue, 5 }  ,fill opacity=0.15 ] (393.33,42.59) .. controls (399.67,40.03) and (408.44,40.34) .. (414.44,41.42) .. controls (415.98,47.11) and (415.29,61.15) .. (414.02,69.67) .. controls (407.77,67.17) and (403.52,66.42) .. (393.77,68.67) .. controls (392.69,61.25) and (391.94,49) .. (393.33,42.59) -- cycle ;
\draw [color={rgb, 255:red, 65; green, 117; blue, 5 }  ,draw opacity=1 ][line width=1.5]    (383.77,50.44) .. controls (398.09,56.09) and (455.1,49.77) .. (451.99,79.33) .. controls (448.88,108.88) and (355.04,95.91) .. (360.02,79.42) .. controls (365,62.92) and (388.02,65.28) .. (411.77,65.17) .. controls (435.53,65.05) and (420.43,85.55) .. (415.99,79.99) ;
\draw [color={rgb, 255:red, 65; green, 117; blue, 5 }  ,draw opacity=1 ][line width=1.5]    (420.18,50.06) .. controls (389.34,47.94) and (354.99,36.65) .. (375.87,23.9) ;
\draw  [color={rgb, 255:red, 245; green, 166; blue, 35 }  ,draw opacity=1 ][fill={rgb, 255:red, 245; green, 166; blue, 35 }  ,fill opacity=0.15 ] (341.17,114.23) .. controls (350.6,111.66) and (363.61,112.83) .. (369.9,115.69) .. controls (371.9,124.26) and (371.49,146.87) .. (369.49,155.44) .. controls (362.76,157.86) and (350.8,157.29) .. (342.34,155.44) .. controls (339.49,148.3) and (339.78,120.65) .. (341.17,114.23) -- cycle ;
\draw [color={rgb, 255:red, 245; green, 166; blue, 35 }  ,draw opacity=1 ][line width=1.5]    (329.16,84.71) .. controls (322.71,91.38) and (323.14,120.4) .. (336.99,119.32) .. controls (350.84,118.25) and (367.17,120.42) .. (378.45,121.14) .. controls (389.72,121.87) and (395.52,117.42) .. (379.27,111.67) .. controls (363.02,105.92) and (348.42,90.62) .. (348.27,74.42) .. controls (348.12,58.21) and (388.52,59.67) .. (420.02,60.67) ;
\draw [color={rgb, 255:red, 74; green, 144; blue, 226 }  ,draw opacity=1 ][line width=1.5]    (212.55,142.91) .. controls (203.81,138.92) and (193.81,170.24) .. (230.23,170.6) .. controls (266.65,170.96) and (262.12,129.61) .. (378.45,131.69) ;
\draw  [color={rgb, 255:red, 144; green, 19; blue, 254 }  ,draw opacity=1 ][fill={rgb, 255:red, 144; green, 19; blue, 254 }  ,fill opacity=0.15 ] (393.48,219.51) .. controls (400.29,220.49) and (408.91,220.65) .. (414.6,218.34) .. controls (416.14,224.03) and (417.6,235.36) .. (416.33,243.88) .. controls (410.64,246.65) and (402.17,246.49) .. (393.56,243.72) .. controls (392.48,236.3) and (392.09,225.93) .. (393.48,219.51) -- cycle ;
\draw [color={rgb, 255:red, 144; green, 19; blue, 254 }  ,draw opacity=1 ][line width=1.5]    (386.79,240.18) .. controls (422.94,248.49) and (465.41,230.34) .. (438.02,195.11) .. controls (410.64,159.88) and (362.94,185.88) .. (362.64,202.03) .. controls (362.33,218.18) and (377.25,223.88) .. (408.02,223.72) .. controls (438.79,223.57) and (426.33,203.26) .. (417.87,206.65) ;
\draw [color={rgb, 255:red, 144; green, 19; blue, 254 }  ,draw opacity=1 ][line width=1.5]    (381.78,261.01) .. controls (373.07,256.54) and (377.56,236.08) .. (385.42,235.92) .. controls (393.27,235.76) and (405.32,239.79) .. (425.4,234.72) ;
\draw [color={rgb, 255:red, 245; green, 166; blue, 35 }  ,draw opacity=1 ][line width=1.5]    (427.63,229.78) .. controls (359.63,233.19) and (353.07,211.32) .. (344.58,199.14) .. controls (336.1,186.96) and (314.05,151.09) .. (348.44,150.95) .. controls (382.84,150.82) and (396.89,136.63) .. (402.58,145.71) ;
\draw  [fill={rgb, 255:red, 255; green, 255; blue, 255 }  ,fill opacity=1 ][line width=1.5]  (287.64,49.1) .. controls (276.16,53.22) and (264.37,55.64) .. (248.05,48.91) .. controls (258.03,32.38) and (277.92,32.28) .. (287.64,49.1) -- cycle ;
\draw [line width=1.5]    (237.75,42.33) .. controls (252.83,57.06) and (283.83,56.72) .. (297.75,42.33) ;

\draw  [color={rgb, 255:red, 208; green, 2; blue, 27 }  ,draw opacity=1 ][fill={rgb, 255:red, 208; green, 2; blue, 27 }  ,fill opacity=0.15 ] (274.82,64.53) .. controls (281.03,64.76) and (284.53,64.25) .. (290.53,60.82) .. controls (294.82,66.82) and (297.92,80.63) .. (296.53,88.82) .. controls (289.81,91.24) and (284.25,93.39) .. (274.82,91.96) .. controls (273.74,84.54) and (273.43,70.95) .. (274.82,64.53) -- cycle ;
\draw [color={rgb, 255:red, 208; green, 2; blue, 27 }  ,draw opacity=1 ][line width=1.5]    (263.5,77.02) .. controls (326.13,92.77) and (326.13,15.27) .. (265.88,21.27) .. controls (205.63,27.27) and (216.19,58.88) .. (269.13,69.02) .. controls (322.06,79.16) and (296.88,44.02) .. (285.13,50.02) ;
\draw [color={rgb, 255:red, 208; green, 2; blue, 27 }  ,draw opacity=1 ][line width=1.5]    (378.13,126.02) .. controls (278,133.52) and (322.75,83.02) .. (287.25,83.27) .. controls (251.75,83.52) and (254.25,94.27) .. (248.88,88.77) ;
\draw  [fill={rgb, 255:red, 255; green, 255; blue, 255 }  ,fill opacity=1 ][line width=1.5]  (295.78,239.25) .. controls (284.3,243.37) and (272.51,245.79) .. (256.19,239.05) .. controls (266.17,222.52) and (286.06,222.42) .. (295.78,239.25) -- cycle ;
\draw [line width=1.5]    (245.89,232.48) .. controls (260.98,247.2) and (291.98,246.86) .. (305.89,232.48) ;

\draw  [color={rgb, 255:red, 139; green, 87; blue, 42 }  ,draw opacity=1 ][fill={rgb, 255:red, 139; green, 87; blue, 42 }  ,fill opacity=0.15 ] (277.39,191.39) .. controls (286.82,193.1) and (293.96,196.82) .. (298.53,199.96) .. controls (296.53,209.1) and (294.2,217.2) .. (292.82,225.39) .. controls (286.53,221.39) and (281.1,217.39) .. (271.96,218.82) .. controls (270.88,211.4) and (273.1,197.1) .. (277.39,191.39) -- cycle ;
\draw [color={rgb, 255:red, 139; green, 87; blue, 42 }  ,draw opacity=1 ][line width=1.5]    (268.71,203.8) .. controls (345.85,220.09) and (310.7,252.68) .. (267.56,253.25) .. controls (224.41,253.82) and (235.82,206.03) .. (274.93,212.25) .. controls (314.04,218.47) and (293.96,235.63) .. (287.39,229.96) ;
\draw [color={rgb, 255:red, 139; green, 87; blue, 42 }  ,draw opacity=1 ][line width=1.5]    (252.88,193.33) .. controls (256.74,189.08) and (283.03,201.79) .. (301.49,206.25) .. controls (319.96,210.71) and (293.42,165.22) .. (301.96,153.79) .. controls (310.49,142.36) and (322.11,137.17) .. (377.65,137.33) ;

\draw (190.14,89.63) node [anchor=south west] [inner sep=0.75pt]  [color={rgb, 255:red, 30; green, 99; blue, 179 }  ,opacity=1 ,xscale=1.2,yscale=1.2]  {$S_{1}$};
\draw (306.03,125.87) node  [color={rgb, 255:red, 215; green, 134; blue, 1 }  ,opacity=1 ,xscale=1.2,yscale=1.2]  {$S_{0}$};
\draw (428.58,113.4) node [anchor=north west][inner sep=0.75pt]  [color={rgb, 255:red, 65; green, 117; blue, 5 }  ,opacity=1 ,xscale=1.2,yscale=1.2]  {$S'_{1}$};
\draw (426.08,159.98) node [anchor=south west] [inner sep=0.75pt]  [color={rgb, 255:red, 118; green, 3; blue, 220 }  ,opacity=1 ,xscale=1.2,yscale=1.2]  {$S'_{2}$};
\draw (318.87,237.87) node [anchor=west] [inner sep=0.75pt]  [color={rgb, 255:red, 133; green, 71; blue, 18 }  ,opacity=1 ,xscale=1.2,yscale=1.2]  {$S_{3}$};
\draw (317.44,37.3) node [anchor=west] [inner sep=0.75pt]  [color={rgb, 255:red, 177; green, 1; blue, 21 }  ,opacity=1 ,xscale=1.2,yscale=1.2]  {$S_{2}$};

\end{tikzpicture}

\begin{tikzpicture}[x=0.75pt,y=0.75pt,yscale=-1,xscale=1]

\draw    (241.17,105.83) -- (299.88,142.91) ;
\draw [shift={(302.42,144.51)}, rotate = 212.27] [fill={rgb, 255:red, 0; green, 0; blue, 0 }  ][line width=0.08]  [draw opacity=0] (7.14,-3.43) -- (0,0) -- (7.14,3.43) -- (4.74,0) -- cycle    ;
\draw    (230,153) -- (297.41,153.71) ;
\draw [shift={(300.41,153.74)}, rotate = 180.61] [fill={rgb, 255:red, 0; green, 0; blue, 0 }  ][line width=0.08]  [draw opacity=0] (7.14,-3.43) -- (0,0) -- (7.14,3.43) -- (4.74,0) -- cycle    ;
\draw    (317.71,154.71) -- (364.68,193.9) ;
\draw [shift={(366.98,195.83)}, rotate = 219.84] [fill={rgb, 255:red, 0; green, 0; blue, 0 }  ][line width=0.08]  [draw opacity=0] (7.14,-3.43) -- (0,0) -- (7.14,3.43) -- (4.74,0) -- cycle    ;
\draw    (364.82,118.4) -- (317.71,154.71) ;
\draw [shift={(367.2,116.57)}, rotate = 142.37] [fill={rgb, 255:red, 0; green, 0; blue, 0 }  ][line width=0.08]  [draw opacity=0] (8.04,-3.86) -- (0,0) -- (8.04,3.86) -- (5.34,0) -- cycle    ;
\draw    (243.27,206.8) -- (302.44,166.53) ;
\draw [shift={(304.92,164.85)}, rotate = 145.77] [fill={rgb, 255:red, 0; green, 0; blue, 0 }  ][line width=0.08]  [draw opacity=0] (7.14,-3.43) -- (0,0) -- (7.14,3.43) -- (4.74,0) -- cycle    ;

\draw  [color={rgb, 255:red, 208; green, 2; blue, 27 }  ,draw opacity=1 ][fill={rgb, 255:red, 255; green, 221; blue, 227 }  ,fill opacity=1 ]  (242.3, 105.6) circle [x radius= 16.62, y radius= 16.62]   ;
\draw (242.3,105.6) node  [color={rgb, 255:red, 152; green, 0; blue, 17 }  ,opacity=1 ,xscale=1.2,yscale=1.2]  {$S_{2}$};
\draw  [color={rgb, 255:red, 74; green, 144; blue, 226 }  ,draw opacity=1 ][fill={rgb, 255:red, 224; green, 241; blue, 255 }  ,fill opacity=1 ]  (225.21, 152.86) circle [x radius= 16.62, y radius= 16.62]   ;
\draw (225.21,152.86) node  [color={rgb, 255:red, 26; green, 87; blue, 156 }  ,opacity=1 ,xscale=1.2,yscale=1.2]  {$S_{1}$};
\draw  [color={rgb, 255:red, 245; green, 166; blue, 35 }  ,draw opacity=1 ][fill={rgb, 255:red, 255; green, 247; blue, 229 }  ,fill opacity=1 ]  (317.21, 152.86) circle [x radius= 16.62, y radius= 16.62]   ;
\draw (317.21,152.86) node  [color={rgb, 255:red, 210; green, 135; blue, 13 }  ,opacity=1 ,xscale=1.2,yscale=1.2]  {$S_{0}$};
\draw  [color={rgb, 255:red, 144; green, 19; blue, 254 }  ,draw opacity=1 ][fill={rgb, 255:red, 245; green, 235; blue, 255 }  ,fill opacity=1 ]  (380.5, 205.86) circle [x radius= 16.62, y radius= 16.62]   ;
\draw (380.5,205.86) node  [color={rgb, 255:red, 107; green, 7; blue, 193 }  ,opacity=1 ,xscale=1.2,yscale=1.2]  {$S'_{2}$};
\draw  [color={rgb, 255:red, 65; green, 117; blue, 5 }  ,draw opacity=1 ][fill={rgb, 255:red, 240; green, 250; blue, 226 }  ,fill opacity=1 ]  (380.21, 105.86) circle [x radius= 16.62, y radius= 16.62]   ;
\draw (380.21,105.86) node  [color={rgb, 255:red, 54; green, 100; blue, 1 }  ,opacity=1 ,xscale=1.2,yscale=1.2]  {$S'_{1}$};
\draw  [color={rgb, 255:red, 139; green, 87; blue, 42 }  ,draw opacity=1 ][fill={rgb, 255:red, 235; green, 221; blue, 215 }  ,fill opacity=1 ]  (243.04, 207.45) circle [x radius= 16.62, y radius= 16.62]   ;
\draw (243.04,207.45) node  [color={rgb, 255:red, 126; green, 79; blue, 39 }  ,opacity=1 ,xscale=1.2,yscale=1.2]  {$S_{3}$};

\end{tikzpicture}

\caption{Example of a homeomorphism with big rotation set, the associates surfaces $S_i$ and the graph $\Tr$. Note that the piece $\rho_0$ of $\rote(f)$ associated to $S_0$ is equal to $\{0\}$ (as $S_0$ has no genus, we have $i_*(H_1(S_0, \R)) = \{0\}$).}\label{FigExample11}
\end{center}
\end{figure}
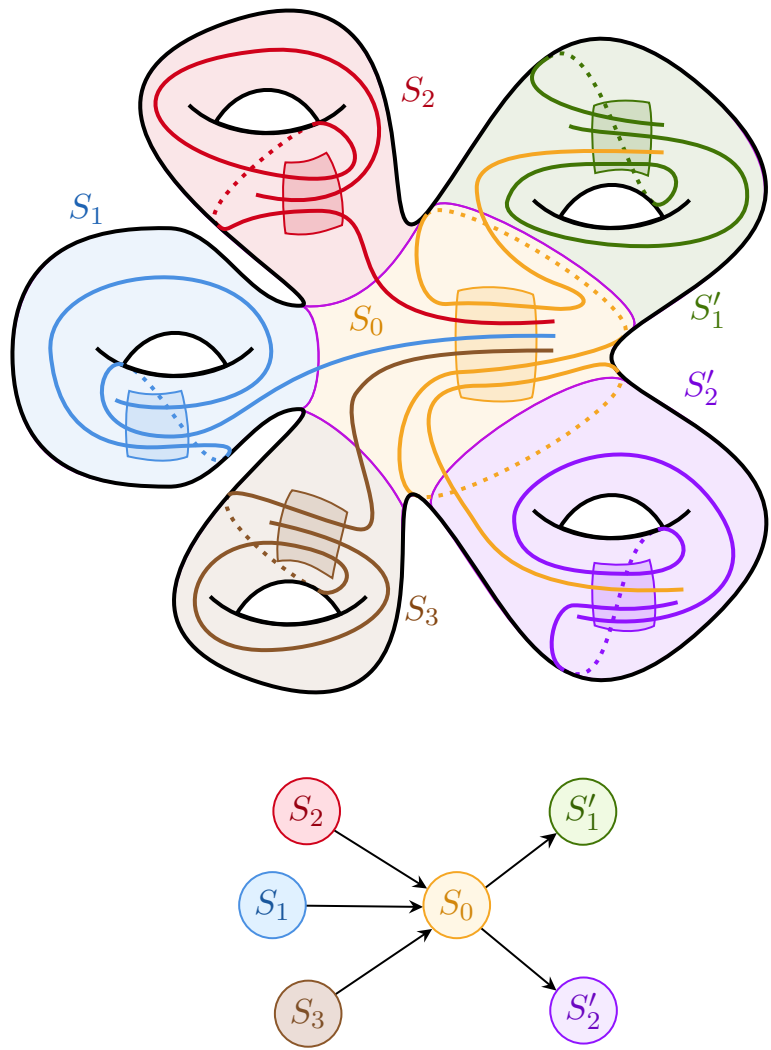

One can then build $f\in \Homeo_0(S)$ such that:
\begin{itemize}
\item For any $1\le i\le \lceil g/2\rceil$, the rectangle $R_i$ is a rotational horseshoe with deck transformations $\Id, T_{i,1},T_{i,2}$ (see \cite[Definition~3.7]{G25Cvx1}, and the intersection $f(R_i)\cap R_0$ is Markovian.
\item For any $1\le j\le \lfloor g/2\rfloor$, the rectangle $R'_j$ is a rotational horseshoe with deck transformations $\Id, T'_{j,1},T'_{j,2}$, and the intersection $f(R_0)\cap R'_j$ is Markovian.
\item For any $1\le i\le \lceil g/2\rceil$, $S_i\subset f(S_i)\subset S_i\cup S_0$.
\item For any $1\le j\le \lfloor g/2\rfloor$, $f(S'_j)\subset S'_j$.
\end{itemize}

These conditions imply that in the oriented graph $\Tr$, the starting vertices are the $(S_i)_{1\le i\le \lceil g/2\rceil}$ (that are surfaces of genus 1), and the ending vertices are the $(S'_j)_{1\le j\le\lfloor g/2\rfloor}$ (that are surfaces with genus 1). Moreover, for any couple $(S_i, S'_j)$ of starting/ending vertices, there is an oriented path in $\Tr$ from $S_i$ to $S'_j$. Call $\rho_i$ the rotation set of the class of $\Merg(f)$ whose associated surface is $S_i$, and $\rho'_j$ the rotation set of the class of $\Merg(f)$ whose associated surface is $S'_j$. Write $V_i = \spn(\rho_i)$ and $V'_j = \spn(\rho'_j)$. 
By Corollary~\ref{CoroPropRotFRotG}, we have 
\[\rot(f) = \bigcup_{\substack{1\le i\le \lceil g/2\rceil \\1\le j\le\lfloor g/2\rfloor}}\conv\big(\rho_i \cup \rho'_j\big).\]
Each of these pieces $\conv(\rho_i, \rho'_j)$ is a nonempty interior convex set that spans $V_i \oplus V'_j$. Hence, $\rot(f)$ is the union of $\lceil g/2\rceil \lfloor g/2\rfloor = \lfloor g^2/4\rfloor$ convex pieces, and cannot be written as the union of $\lfloor g^2/4\rfloor-1$ convex pieces. 
\bigskip

Let us now describe, for any $g\ge 2$, a homeomorphism of the closed surface $S$ of genus $g$ that has big rotation set that is not the union of less than $g$ convex sets. To get this it suffices to modify the preceding construction by deleting all the connections between classes. More precisely, replace the 4 above hypotheses by:
\begin{itemize}
\item For any $1\le i\le \lceil g/2\rceil$, the rectangle $R_i$ is a rotational horseshoe with deck transformations $\Id, T_{i,1},T_{i,2}$ (see \cite[Definition~3.7]{G25Cvx1}).
\item For any $1\le j\le \lfloor g/2\rfloor$, the rectangle $R'_j$ is a rotational horseshoe with deck transformations $\Id, T'_{j,1},T'_{j,2}$.
\item For any $1\le i\le \lceil g/2\rceil$, $f(S_i)= S_i$.
\item For any $1\le j\le \lfloor g/2\rfloor$, $f(S'_j)= S'_j$.
\end{itemize}

These conditions ensure that in the oriented graph $\Tr$, all the vertices with genus are both starting and ending (there is no connection between them).
By Corollary~\ref{CoroPropRotFRotG}, this implies
\[\rot(f) = \bigcup_{1\le i\le \lceil g/2\rceil }\conv(\rho_i) \cup \bigcup_{1\le j\le\lfloor g/2\rfloor}\conv(\rho'_j).\]
Each of these pieces is a nonempty interior convex set that spans $V_i$ or $V'_j$. Hence, $\rot(f)$ is the union of $g$ convex pieces, and cannot be written as the union of $g-1$ convex pieces.

\section{Bounded deviations II: proof of Theorem~\ref{TheoBndedHomo} and its consequences}\label{SecBnddDev2}

In this section we finish our study of bounded deviations initiated in Section~\ref{SecBnddDev1} and prove Theorem~\ref{TheoBndedHomo} and its consequences. 

This will be made by combining the two bounded deviation results we already proved (Propositions~\ref{PropDevImpliesConnec} and \ref{PropBnddDevConvHull}) with the structure theorem of rotation sets (Theorem~\ref{MainTheoUnionConvex}).

The idea of the proof is that there are two ways of having big deviation with respect to $\rot(f)$: either being far away from $\conv(n\rot(f))$, or being in $\conv(n\rot(f))$ but far away from $n\rot(f)$ (Lemma~\ref{LemInterVois}). In the first case we prove that this creates periodic orbits having rotation vector outside of $\conv(\rot(f))$. In the second case, we prove that this implies new connections between classes. 

\begin{proof}[Proof of Theorem~\ref{TheoBndedHomo}]
To simplify the exposition it will be simpler to replace $a_x^n$ (defined page~\pageref{Defay}) with a curve that is at finite homological distance to it: we fix a basepoint $o\in D$ and denote $\alpha_x^n$ the geodesic segment linking $o$ to $a_{f^{n-1}(y)}\dots a_y o$. 
Note that $\wt f^n(\wt x) \in a_{f^{n-1}(y)}\dots a_y D$.

Recall that by Theorem~\ref{MainTheoUnionConvex} the rotation set can be written
\[\rot(f) = \bigcup_{j\in J} \rho'_j,\]
where for any $j\in J$ we can write
\[\rho'_j = \conv \big\{\rho_i \mid i\in I^j\big\}
\qquad\text{and}\qquad
V'_j =  \bigoplus_{i\in I^j} V_i .\]

Recall that for $E$ a set and $R>0$, denote $B_R(E) = \{x\mid d(x, E)<R\}$.

\begin{lemma}\label{LemInterVois}
For any $R>0$, we have
\[B_R\Big(\bigcup_{j\in J} V'_j\Big)\cap B_R\big(\conv(\rot(f)\big) \subset  B_{3R}\big(\rot(f)).\]
\end{lemma}

\begin{proof}
Let $x\in B_R\big(\bigcup_{j\in J} V'_j\big)\cap B_R\big(\conv(\rot(f)\big)$.
This ensures the existence of $j\in J$ such that $d(x,V'_j)\le R$. 
There also exists $z\in \conv(\rot(f))$ such that $\|x-z\|\le R$. One can write $z = \sum_{i\in I_{\mathrm{h}}} \lambda_i v_i$, with $v_i\in\rho_i$ and $\sum_i\lambda_i=1$, $\lambda_i\ge 0$. 
Let us denote $\pr_{V'_j}$ the projection on $V'_j$ parallel to $\bigoplus_{i\notin I^j} V_i$. One has
\[\pr_{V'_j}(z) = \sum_{i\in I^j} \lambda_i v_i = \sum_{i\in I^j} \lambda_i v_i + \Big(\sum_{i\notin I^j} \lambda_i\Big). 0.\]
This last equality ensures that (recall that $\overline{\rho_j'}$ is star-shaped with respect to 0) $\pr_{V'_j}(z)\in \overline{\rho_j'} \subset \rot(f)$.

Note that $\pr_{V'_j}(z)$ is also the projection of $z$ on the closest point of $V_j$ for the norm $\|\cdot\|$. Hence, 
\[\|z-\pr_{V'_j}(z)\| = d(z, V'_j) \le \|x-z\| + d(x, V'_j) \le 2R,\]
so
\[\|x-\pr_{V'_j}(z)\| \le \|x-z\| + \|z-\pr_{V'_j}(z)\| \le 3R.\]
This proves that $d(x,\rot(f))\le 3R$.
\end{proof}

We argue by contradiction and suppose that Theorem~\ref{TheoBndedHomo} is false: there exist $x\in S$ and $n\in\N$ such that 
\[d\big([\alpha_x^n],\, n\rot(f)  \big) > 3L.\]
Lemma~\ref{LemInterVois} implies that either $d\big([\alpha_x^n],\, \bigcup_{j\in J} V'_j  \big) > L$, or $d\big([\alpha_x^n],\, n\conv(\rot(f))  \big) > L$.
But these are impossible, the first one by Proposition~\ref{PropDevImpliesConnec}, and the second one by Proposition~\ref{PropBnddDevConvHull}.
\end{proof}

We now come to the consequences of bounded deviations. 

Corollary~\ref{CoroBoyland1} is a classical consequence of Theorem~\ref{TheoBndedHomo}: as noted in the first paragraph of \cite[Section~4.3.1]{lellouch}, the only hypothesis the author uses is the boundedness of deviations, \emph{i.e.}\ the conclusion of Theorem~\ref{TheoBndedHomo}). 

Let us prove Corollary~\ref{CoroBoyland2}, starting with some notations (following \cite{lellouch}). Consider $\omega$ a closed 1-form on $S$ and $[\omega]\in H^1(S,\R)$ its cohomology class. For $\mu\in\M(f)$, denote
\[\rho(\mu)\cdot[\omega] = \int_S \left(\int_{I(x)} \omega\right)\dd\mu(x)\]
(this amounts seeing $\rho(\mu)\in H_1(S,\R)$ as an element of the dual of $H^1(S,\R)$). Write
\[\phi_{[\omega]} = \max_{\mu\in\M(f)} \rho(\mu)\cdot[\omega] 
\qquad\text{and}\qquad
\M_{[\omega]} = \big\{\mu\in\M(f) \mid \rho(\mu)\cdot[\omega] = \phi_{[\omega]}\big\}.\]
By compactness of $\M(f)$, $\phi_{[\omega]}$ is well defined and $\M_{[\omega]}$ is nonempty and compact. Also write
\[X_{[\omega]} = \overline{\bigcup_{\mu\in \M_{[\omega]}} \supp(\mu)}.\]
Finally, we fix a norm $\|\cdot\|$ on $H^1(S,\R)$. 

\begin{prop}\label{PropBoyland2}
Let $f\in\Homeo_0(S)$ such that $\inte\big(\conv(\rot(f))\big)\neq\emptyset$. Then there exists $L_0>0$ such that for any closed 1-form $\omega$, any $x\in X_{[\omega]}$ and any $n\ge 1$, 
\[\left|\left(\int_{I^n(x)}\omega\right) - n \phi_{[\omega]}\right|\le L_0 \big\|[\omega]\big\|.\]
\end{prop}

This proposition, as Corollary~\ref{CoroBoyland1}, is a classical consequence of Theorem~\ref{TheoBndedHomo} (as noted in the first paragraph of \cite[Section~4.3.1]{lellouch}, the only hypothesis the author uses is the boundedness of deviations, \emph{i.e.}\ the conclusion of Theorem~\ref{TheoBndedHomo}). 
\bigskip

Corollary~\ref{CoroBoyland2} then follows directly from Proposition~\ref{PropBoyland2}, by taking $[\omega]$ associated to the linear form defining the hyperplane $H$ of the definition of an exposed point.

\section{Continuity properties}

The following completes Definition~\ref{DefErgHomRot}:

\begin{definition}\label{DefErgHomRot2}
Let $f \in \Homeo_0(S)$.
The \emph{measures' (homological) rotation set} $\rotm(f)$ of $f$ is
\[\rotm(f) = \big\{\rho(\mu) \mid \mu \in \M(f)\big\}.\] 
\end{definition}

By \eqref{eq:homologyequation}, the map $\mu\mapsto \rho(\mu)$ is continuous (for the weak-$*$ topology) and affine. Hence $\rotm(f)$ is a convex compact subset of $H_1(S,\R)$. 

\begin{lemma}
The map $f\mapsto \rotm(f)$ is upper semi-continuous.
\end{lemma}

\begin{proof}
This follows directly from the upper-semi continuity of the map $f\mapsto \M(f)$ and the continuity of $\mu\mapsto \rho(\mu)$.
\end{proof}

The following proposition is a higher genus version of the continuity of the map $f\mapsto \rot(f)$ at any torus homeomorphism whose rotation set has nonempty interior \cite{llibremackay}.  

\begin{prop}\label{PropContinuityRot}
The map $f\mapsto \rotm(f)$ is continuous at every homeomorphism $f$ such that $\inte\rotm(f)\neq \emptyset$ (\emph{i.e.}~with big rotation set).

The map $f\mapsto \rote(f)$ is continuous at every homeomorphism $f$ such that $\inte\rote(f)\neq \emptyset$.
\end{prop}

As a consequence, the sets of homeomorphisms with big rotation set and whose ergodic rotation set has nonempty interior are both open in $\Homeo_0(S)$.

\begin{proof}
It suffices to show the lower semicontinuity of these maps.
This is a consequence of the persistence of horseshoes under $C^0$ perturbations (\cite[Lemma~3.6]{G25Cvx1}) and the density of elements $(r_\omega)_{\omega\in\Omega}$ of $\rote(f)\cap H_1(S,\Q)$ realised by rotational horseshoes: such a family of horseshoes was built in \cite[Subsection~4.1]{G25Cvx1}.
\end{proof}

\begin{lemma}\label{ExFinal}
The maps $f\mapsto \rote(f)$ and  $f\mapsto \rot(f)$ are not upper semi-continuous, even in restriction to homeomorphisms with big rotation set.
\end{lemma}

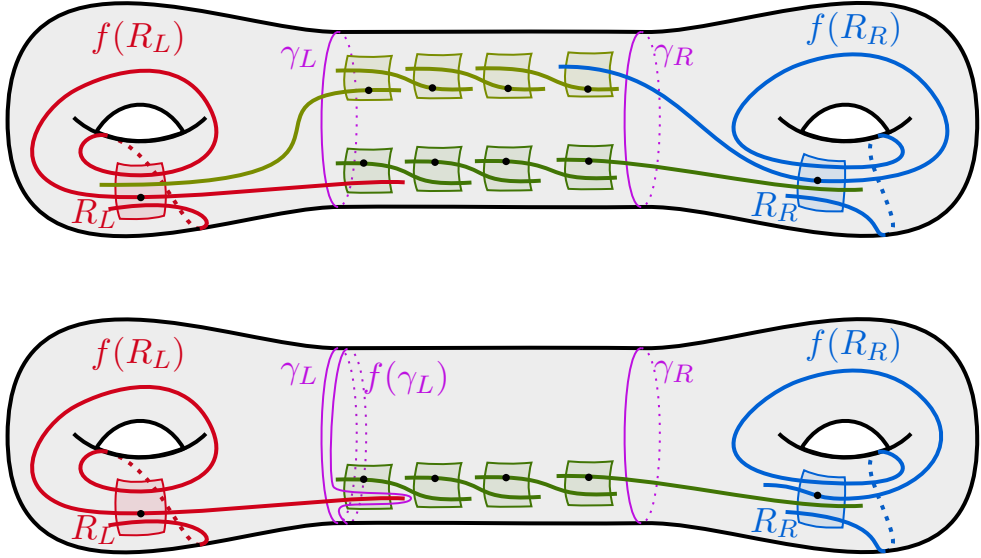
\begin{figure}[!t]
\begin{center}
\vspace{-25pt}
\tikzset{every picture/.style={line width=0.75pt}} 

\begin{tikzpicture}[x=0.75pt,y=0.75pt,yscale=-1.1,xscale=1.1]
\draw [color={rgb, 255:red, 0; green, 97; blue, 212 }  ,draw opacity=1 ][line width=1.5]  [dash pattern={on 1.69pt off 2.76pt}]  (508.38,172.75) .. controls (521.63,171.55) and (491.27,131.95) .. (505,127.55) ;
\draw [color={rgb, 255:red, 208; green, 2; blue, 27 }  ,draw opacity=1 ][line width=1.5]  [dash pattern={on 1.69pt off 2.76pt}]  (155.4,127.42) .. controls (179.4,135.02) and (189.2,171.22) .. (197.2,169.82) ;
\draw  [fill={rgb, 255:red, 74; green, 74; blue, 74 }  ,fill opacity=0.1 ][line width=1.5]  (110,120) .. controls (110.37,31.71) and (212.75,80.04) .. (260,80) .. controls (307.25,79.96) and (352.25,79.54) .. (400,80) .. controls (447.75,80.46) and (548.75,31.04) .. (550,120) .. controls (551.25,208.96) and (449.75,159.46) .. (400,160) .. controls (350.25,160.54) and (308.25,159.46) .. (260,160) .. controls (211.75,160.54) and (109.63,208.29) .. (110,120) -- cycle ;
\draw  [fill={rgb, 255:red, 255; green, 255; blue, 255 }  ,fill opacity=1 ][line width=1.5]  (189.89,125.85) .. controls (178.41,129.97) and (166.62,132.39) .. (150.3,125.66) .. controls (160.28,109.13) and (180.17,109.03) .. (189.89,125.85) -- cycle ;
\draw [line width=1.5]    (140,119.08) .. controls (155.08,133.81) and (186.08,133.47) .. (200,119.08) ;

\draw  [fill={rgb, 255:red, 255; green, 255; blue, 255 }  ,fill opacity=1 ][line width=1.5]  (509.89,125.85) .. controls (498.41,129.97) and (486.62,132.39) .. (470.3,125.66) .. controls (480.28,109.13) and (500.17,109.03) .. (509.89,125.85) -- cycle ;
\draw [line width=1.5]    (460,119.08) .. controls (475.08,133.81) and (506.08,133.47) .. (520,119.08) ;

\draw [color={rgb, 255:red, 189; green, 16; blue, 224 }  ,draw opacity=1 ]   (260,80) .. controls (250.28,80.08) and (249.61,160.08) .. (260,160) ;
\draw [color={rgb, 255:red, 189; green, 16; blue, 224 }  ,draw opacity=1 ] [dash pattern={on 0.84pt off 2.51pt}]  (260,80) .. controls (271.61,80.08) and (270.94,159.75) .. (260,160) ;
\draw [color={rgb, 255:red, 189; green, 16; blue, 224 }  ,draw opacity=1 ]   (397.54,80) .. controls (387.82,80.08) and (387.15,160.08) .. (397.54,160) ;
\draw [color={rgb, 255:red, 189; green, 16; blue, 224 }  ,draw opacity=1 ] [dash pattern={on 0.84pt off 2.51pt}]  (397.54,80) .. controls (409.15,80.08) and (408.49,159.75) .. (397.54,160) ;
\draw  [color={rgb, 255:red, 208; green, 2; blue, 27 }  ,draw opacity=1 ][fill={rgb, 255:red, 208; green, 2; blue, 27 }  ,fill opacity=0.1 ] (160,140) .. controls (164.8,141.62) and (170.8,141.42) .. (180,140) .. controls (181.61,146.08) and (182.61,156.43) .. (181,164.02) .. controls (174.6,165.82) and (167.8,165.62) .. (161,165.42) .. controls (158.28,158.17) and (158.61,146.42) .. (160,140) -- cycle ;
\draw  [color={rgb, 255:red, 65; green, 117; blue, 5 }  ,draw opacity=1 ][fill={rgb, 255:red, 65; green, 117; blue, 5 }  ,fill opacity=0.1 ] (264.33,133.33) .. controls (272,133.88) and (275.4,134.08) .. (284.33,133.33) .. controls (282.8,138.48) and (282.4,145.08) .. (284.33,153.33) .. controls (277,153.48) and (272.8,153.68) .. (264.33,153.33) .. controls (261.61,146.08) and (262.94,139.75) .. (264.33,133.33) -- cycle ;
\draw [color={rgb, 255:red, 208; green, 2; blue, 27 }  ,draw opacity=1 ][line width=1.5]    (155.4,127.42) .. controls (144.4,124.82) and (134.41,139.12) .. (155.05,143.98) .. controls (175.68,148.85) and (215.15,142.97) .. (202.2,115.64) .. controls (189.25,88.31) and (156,98.13) .. (140.47,106.73) .. controls (124.93,115.33) and (108.93,140.95) .. (134.87,151.73) .. controls (160.8,162.52) and (257.71,146.72) .. (290.21,148.75) ;
\draw [color={rgb, 255:red, 65; green, 117; blue, 5 }  ,draw opacity=1 ][line width=1.5]    (258.94,140.08) .. controls (302.94,138.75) and (282.84,151.18) .. (320.51,148.84) ;
\draw  [color={rgb, 255:red, 65; green, 117; blue, 5 }  ,draw opacity=1 ][fill={rgb, 255:red, 65; green, 117; blue, 5 }  ,fill opacity=0.1 ] (295.93,132.53) .. controls (303.6,133.08) and (307,133.28) .. (315.93,132.53) .. controls (314.4,137.68) and (314,144.28) .. (315.93,152.53) .. controls (308.6,152.68) and (304.4,152.88) .. (295.93,152.53) .. controls (293.21,145.28) and (294.54,138.95) .. (295.93,132.53) -- cycle ;
\draw [color={rgb, 255:red, 65; green, 117; blue, 5 }  ,draw opacity=1 ][line width=1.5]    (290.26,139.28) .. controls (334.26,137.95) and (314.27,150.32) .. (351.94,147.99) ;
\draw  [color={rgb, 255:red, 65; green, 117; blue, 5 }  ,draw opacity=1 ][fill={rgb, 255:red, 65; green, 117; blue, 5 }  ,fill opacity=0.1 ] (327.73,132.33) .. controls (335.4,132.88) and (338.8,133.08) .. (347.73,132.33) .. controls (346.2,137.48) and (345.8,144.08) .. (347.73,152.33) .. controls (340.4,152.48) and (336.2,152.68) .. (327.73,152.33) .. controls (325.01,145.08) and (326.34,138.75) .. (327.73,132.33) -- cycle ;
\draw [color={rgb, 255:red, 65; green, 117; blue, 5 }  ,draw opacity=1 ][line width=1.5]    (322.06,139.08) .. controls (366.06,137.75) and (348.84,148.89) .. (386.51,146.56) ;
\draw  [color={rgb, 255:red, 65; green, 117; blue, 5 }  ,draw opacity=1 ][fill={rgb, 255:red, 65; green, 117; blue, 5 }  ,fill opacity=0.1 ] (364.53,131.73) .. controls (372.2,132.28) and (375.6,132.48) .. (384.53,131.73) .. controls (383,136.88) and (382.6,143.48) .. (384.53,151.73) .. controls (377.2,151.88) and (373,152.08) .. (364.53,151.73) .. controls (361.81,144.48) and (363.14,138.15) .. (364.53,131.73) -- cycle ;
\draw  [color={rgb, 255:red, 0; green, 97; blue, 212 }  ,draw opacity=1 ][fill={rgb, 255:red, 0; green, 97; blue, 212 }  ,fill opacity=0.1 ] (471.47,135.82) .. controls (476.74,138.23) and (482.87,139.22) .. (489.87,139.02) .. controls (490.47,144.62) and (489.77,154.23) .. (488.87,162.02) .. controls (481.47,162.02) and (477.11,162.35) .. (469.17,159.6) .. controls (467.47,152.22) and (469.27,141.42) .. (471.47,135.82) -- cycle ;
\draw [color={rgb, 255:red, 65; green, 117; blue, 5 }  ,draw opacity=1 ][line width=1.5]    (359.14,138.48) .. controls (403.14,137.15) and (460.4,154.35) .. (498.07,152.02) ;
\draw  [color={rgb, 255:red, 121; green, 141; blue, 0 }  ,draw opacity=1 ][fill={rgb, 255:red, 121; green, 141; blue, 0 }  ,fill opacity=0.1 ] (264.33,90.93) .. controls (272,91.48) and (275.4,91.68) .. (284.33,90.93) .. controls (282.8,96.08) and (282.4,102.68) .. (284.33,110.93) .. controls (277,111.08) and (272.8,111.28) .. (264.33,110.93) .. controls (261.61,103.68) and (262.94,97.35) .. (264.33,90.93) -- cycle ;
\draw [color={rgb, 255:red, 121; green, 141; blue, 0 }  ,draw opacity=1 ][line width=1.5]    (258.94,97.68) .. controls (302.94,96.35) and (283.41,108.32) .. (321.08,105.99) ;
\draw  [color={rgb, 255:red, 121; green, 141; blue, 0 }  ,draw opacity=1 ][fill={rgb, 255:red, 121; green, 141; blue, 0 }  ,fill opacity=0.1 ] (295.93,90.13) .. controls (303.6,90.68) and (307,90.88) .. (315.93,90.13) .. controls (314.4,95.28) and (314,101.88) .. (315.93,110.13) .. controls (308.6,110.28) and (304.4,110.48) .. (295.93,110.13) .. controls (293.21,102.88) and (294.54,96.55) .. (295.93,90.13) -- cycle ;
\draw [color={rgb, 255:red, 121; green, 141; blue, 0 }  ,draw opacity=1 ][line width=1.5]    (290.54,96.88) .. controls (334.54,95.55) and (316.27,108.32) .. (353.94,105.99) ;
\draw  [color={rgb, 255:red, 121; green, 141; blue, 0 }  ,draw opacity=1 ][fill={rgb, 255:red, 121; green, 141; blue, 0 }  ,fill opacity=0.1 ] (327.73,89.93) .. controls (335.4,90.48) and (338.8,90.68) .. (347.73,89.93) .. controls (346.2,95.08) and (345.8,101.68) .. (347.73,109.93) .. controls (340.4,110.08) and (336.2,110.28) .. (327.73,109.93) .. controls (325.01,102.68) and (326.34,96.35) .. (327.73,89.93) -- cycle ;
\draw [color={rgb, 255:red, 121; green, 141; blue, 0 }  ,draw opacity=1 ][line width=1.5]    (322.34,96.68) .. controls (366.34,95.35) and (349.7,108.32) .. (387.37,105.99) ;
\draw  [color={rgb, 255:red, 121; green, 141; blue, 0 }  ,draw opacity=1 ][fill={rgb, 255:red, 121; green, 141; blue, 0 }  ,fill opacity=0.1 ] (364.53,89.33) .. controls (372.2,89.88) and (375.6,90.08) .. (384.53,89.33) .. controls (383,94.48) and (382.6,101.08) .. (384.53,109.33) .. controls (377.2,109.48) and (373,109.68) .. (364.53,109.33) .. controls (361.81,102.08) and (363.14,95.75) .. (364.53,89.33) -- cycle ;
\draw [color={rgb, 255:red, 0; green, 97; blue, 212 }  ,draw opacity=1 ][line width=1.5]    (360.07,95.88) .. controls (423,94.53) and (423.6,147.75) .. (482.87,148.22) .. controls (542.13,148.68) and (543.85,125.2) .. (527.83,105.75) .. controls (511.81,86.3) and (487.73,88.05) .. (468.33,95.25) .. controls (448.93,102.45) and (411.47,135.62) .. (474.27,141.02) .. controls (537.07,146.42) and (513.01,124.95) .. (505,127.55) ;
\draw [color={rgb, 255:red, 121; green, 141; blue, 0 }  ,draw opacity=1 ][line width=1.5]    (151.77,149.88) .. controls (303.37,151.28) and (191.37,103.28) .. (288.57,106.68) ;
\draw  [draw opacity=0][fill={rgb, 255:red, 0; green, 0; blue, 0 }  ,fill opacity=1 ] (269.91,139.93) .. controls (269.91,138.99) and (270.67,138.24) .. (271.6,138.24) .. controls (272.54,138.24) and (273.3,138.99) .. (273.3,139.93) .. controls (273.3,140.87) and (272.54,141.63) .. (271.6,141.63) .. controls (270.67,141.63) and (269.91,140.87) .. (269.91,139.93) -- cycle ;
\draw  [draw opacity=0][fill={rgb, 255:red, 0; green, 0; blue, 0 }  ,fill opacity=1 ] (302.28,139.56) .. controls (302.28,138.62) and (303.04,137.86) .. (303.98,137.86) .. controls (304.91,137.86) and (305.67,138.62) .. (305.67,139.56) .. controls (305.67,140.49) and (304.91,141.25) .. (303.98,141.25) .. controls (303.04,141.25) and (302.28,140.49) .. (302.28,139.56) -- cycle ;
\draw  [draw opacity=0][fill={rgb, 255:red, 0; green, 0; blue, 0 }  ,fill opacity=1 ] (334.53,139.18) .. controls (334.53,138.24) and (335.29,137.49) .. (336.23,137.49) .. controls (337.16,137.49) and (337.92,138.24) .. (337.92,139.18) .. controls (337.92,140.12) and (337.16,140.88) .. (336.23,140.88) .. controls (335.29,140.88) and (334.53,140.12) .. (334.53,139.18) -- cycle ;
\draw  [draw opacity=0][fill={rgb, 255:red, 0; green, 0; blue, 0 }  ,fill opacity=1 ] (372.05,139.06) .. controls (372.05,138.12) and (372.81,137.36) .. (373.75,137.36) .. controls (374.68,137.36) and (375.44,138.12) .. (375.44,139.06) .. controls (375.44,139.99) and (374.68,140.75) .. (373.75,140.75) .. controls (372.81,140.75) and (372.05,139.99) .. (372.05,139.06) -- cycle ;
\draw  [draw opacity=0][fill={rgb, 255:red, 0; green, 0; blue, 0 }  ,fill opacity=1 ] (371.65,106.11) .. controls (371.65,105.17) and (372.41,104.41) .. (373.35,104.41) .. controls (374.28,104.41) and (375.04,105.17) .. (375.04,106.11) .. controls (375.04,107.04) and (374.28,107.8) .. (373.35,107.8) .. controls (372.41,107.8) and (371.65,107.04) .. (371.65,106.11) -- cycle ;
\draw  [draw opacity=0][fill={rgb, 255:red, 0; green, 0; blue, 0 }  ,fill opacity=1 ] (301.1,105.57) .. controls (301.1,104.64) and (301.86,103.88) .. (302.79,103.88) .. controls (303.73,103.88) and (304.49,104.64) .. (304.49,105.57) .. controls (304.49,106.51) and (303.73,107.27) .. (302.79,107.27) .. controls (301.86,107.27) and (301.1,106.51) .. (301.1,105.57) -- cycle ;
\draw  [draw opacity=0][fill={rgb, 255:red, 0; green, 0; blue, 0 }  ,fill opacity=1 ] (335.44,105.81) .. controls (335.44,104.87) and (336.2,104.11) .. (337.13,104.11) .. controls (338.07,104.11) and (338.83,104.87) .. (338.83,105.81) .. controls (338.83,106.74) and (338.07,107.5) .. (337.13,107.5) .. controls (336.2,107.5) and (335.44,106.74) .. (335.44,105.81) -- cycle ;
\draw  [draw opacity=0][fill={rgb, 255:red, 0; green, 0; blue, 0 }  ,fill opacity=1 ] (272.18,106.68) .. controls (272.18,105.74) and (272.94,104.99) .. (273.87,104.99) .. controls (274.81,104.99) and (275.57,105.74) .. (275.57,106.68) .. controls (275.57,107.62) and (274.81,108.38) .. (273.87,108.38) .. controls (272.94,108.38) and (272.18,107.62) .. (272.18,106.68) -- cycle ;
\draw [color={rgb, 255:red, 208; green, 2; blue, 27 }  ,draw opacity=1 ][line width=1.5]    (155.8,161.02) .. controls (196.2,154.22) and (206.6,168.42) .. (197.2,169.82) ;
\draw [color={rgb, 255:red, 0; green, 97; blue, 212 }  ,draw opacity=1 ][line width=1.5]    (508.38,172.75) .. controls (500.65,173.95) and (511.32,158.08) .. (463.05,154.88) ;
\draw  [draw opacity=0][fill={rgb, 255:red, 0; green, 0; blue, 0 }  ,fill opacity=1 ] (168.78,155.68) .. controls (168.78,154.74) and (169.54,153.99) .. (170.47,153.99) .. controls (171.41,153.99) and (172.17,154.74) .. (172.17,155.68) .. controls (172.17,156.62) and (171.41,157.38) .. (170.47,157.38) .. controls (169.54,157.38) and (168.78,156.62) .. (168.78,155.68) -- cycle ;
\draw  [draw opacity=0][fill={rgb, 255:red, 0; green, 0; blue, 0 }  ,fill opacity=1 ] (475.87,147.88) .. controls (475.87,146.94) and (476.63,146.19) .. (477.56,146.19) .. controls (478.5,146.19) and (479.26,146.94) .. (479.26,147.88) .. controls (479.26,148.82) and (478.5,149.58) .. (477.56,149.58) .. controls (476.63,149.58) and (475.87,148.82) .. (475.87,147.88) -- cycle ;

\draw (161.37,154.72) node [anchor=north east] [inner sep=0.75pt]  [color={rgb, 255:red, 208; green, 2; blue, 27 }  ,opacity=1 ,xscale=1.2,yscale=1.2]  {$R_{L}$};
\draw (470.88,151.4) node [anchor=north east] [inner sep=0.75pt]  [color={rgb, 255:red, 0; green, 97; blue, 212 }  ,opacity=1 ,xscale=1.2,yscale=1.2]  {$R_{R}$};
\draw (252.8,91.82) node [anchor=east] [inner sep=0.75pt]  [color={rgb, 255:red, 189; green, 16; blue, 224 }  ,opacity=1 ,xscale=1.2,yscale=1.2]  {$\gamma _{L}$};
\draw (401.86,90.9) node [anchor=west] [inner sep=0.75pt]  [color={rgb, 255:red, 189; green, 16; blue, 224 }  ,opacity=1 ,xscale=1.2,yscale=1.2]  {$\gamma _{R}$};
\draw (169.32,92.26) node [anchor=south] [inner sep=0.75pt]  [color={rgb, 255:red, 208; green, 2; blue, 27 }  ,opacity=1 ,xscale=1.2,yscale=1.2]  {$f( R_{L})$};
\draw (494.06,88.47) node [anchor=south] [inner sep=0.75pt]  [color={rgb, 255:red, 0; green, 97; blue, 212 }  ,opacity=1 ,xscale=1.2,yscale=1.2]  {$f( R_{R})$};

\end{tikzpicture}\vspace{-30pt}
\begin{tikzpicture}[x=0.75pt,y=0.75pt,yscale=-1.1,xscale=1.1]

\draw [color={rgb, 255:red, 189; green, 16; blue, 224 }  ,draw opacity=1 ] [dash pattern={on 0.84pt off 2.51pt}]  (280,100) .. controls (291.61,100.08) and (290.94,179.75) .. (280,180) ;
\draw [color={rgb, 255:red, 189; green, 16; blue, 224 }  ,draw opacity=1 ] [dash pattern={on 0.84pt off 2.51pt}]  (284.29,100.33) .. controls (295.9,100.42) and (295.23,180.08) .. (284.29,180.33) ;
\draw [color={rgb, 255:red, 0; green, 97; blue, 212 }  ,draw opacity=1 ][line width=1.5]  [dash pattern={on 1.69pt off 2.76pt}]  (528.38,192.75) .. controls (541.63,191.55) and (511.27,151.95) .. (525,147.55) ;
\draw [color={rgb, 255:red, 208; green, 2; blue, 27 }  ,draw opacity=1 ][line width=1.5]  [dash pattern={on 1.69pt off 2.76pt}]  (175.4,147.42) .. controls (199.4,155.02) and (209.2,191.22) .. (217.2,189.82) ;
\draw  [fill={rgb, 255:red, 74; green, 74; blue, 74 }  ,fill opacity=0.1 ][line width=1.5]  (130,140) .. controls (130.37,51.71) and (232.75,100.04) .. (280,100) .. controls (327.25,99.96) and (372.25,99.54) .. (420,100) .. controls (467.75,100.46) and (568.75,51.04) .. (570,140) .. controls (571.25,228.96) and (469.75,179.46) .. (420,180) .. controls (370.25,180.54) and (328.25,179.46) .. (280,180) .. controls (231.75,180.54) and (129.63,228.29) .. (130,140) -- cycle ;
\draw  [fill={rgb, 255:red, 255; green, 255; blue, 255 }  ,fill opacity=1 ][line width=1.5]  (209.89,145.85) .. controls (198.41,149.97) and (186.62,152.39) .. (170.3,145.66) .. controls (180.28,129.13) and (200.17,129.03) .. (209.89,145.85) -- cycle ;
\draw [line width=1.5]    (160,139.08) .. controls (175.08,153.81) and (206.08,153.47) .. (220,139.08) ;

\draw  [fill={rgb, 255:red, 255; green, 255; blue, 255 }  ,fill opacity=1 ][line width=1.5]  (529.89,145.85) .. controls (518.41,149.97) and (506.62,152.39) .. (490.3,145.66) .. controls (500.28,129.13) and (520.17,129.03) .. (529.89,145.85) -- cycle ;
\draw [line width=1.5]    (480,139.08) .. controls (495.08,153.81) and (526.08,153.47) .. (540,139.08) ;

\draw [color={rgb, 255:red, 189; green, 16; blue, 224 }  ,draw opacity=1 ]   (280,100) .. controls (270.28,100.08) and (269.61,180.08) .. (280,180) ;
\draw [color={rgb, 255:red, 189; green, 16; blue, 224 }  ,draw opacity=1 ]   (417.54,100) .. controls (407.82,100.08) and (407.15,180.08) .. (417.54,180) ;
\draw [color={rgb, 255:red, 189; green, 16; blue, 224 }  ,draw opacity=1 ] [dash pattern={on 0.84pt off 2.51pt}]  (417.54,100) .. controls (429.15,100.08) and (428.49,179.75) .. (417.54,180) ;
\draw  [color={rgb, 255:red, 208; green, 2; blue, 27 }  ,draw opacity=1 ][fill={rgb, 255:red, 208; green, 2; blue, 27 }  ,fill opacity=0.1 ] (180,160) .. controls (184.8,161.62) and (190.8,161.42) .. (200,160) .. controls (201.61,166.08) and (202.61,176.43) .. (201,184.02) .. controls (194.6,185.82) and (187.8,185.62) .. (181,185.42) .. controls (178.28,178.17) and (178.61,166.42) .. (180,160) -- cycle ;
\draw  [color={rgb, 255:red, 65; green, 117; blue, 5 }  ,draw opacity=1 ][fill={rgb, 255:red, 65; green, 117; blue, 5 }  ,fill opacity=0.1 ] (284.33,153.33) .. controls (292,153.88) and (295.4,154.08) .. (304.33,153.33) .. controls (302.8,158.48) and (302.4,165.08) .. (304.33,173.33) .. controls (297,173.48) and (292.8,173.68) .. (284.33,173.33) .. controls (281.61,166.08) and (282.94,159.75) .. (284.33,153.33) -- cycle ;
\draw [color={rgb, 255:red, 208; green, 2; blue, 27 }  ,draw opacity=1 ][line width=1.5]    (175.4,147.42) .. controls (164.4,144.82) and (154.41,159.12) .. (175.05,163.98) .. controls (195.68,168.85) and (235.15,162.97) .. (222.2,135.64) .. controls (209.25,108.31) and (176,118.13) .. (160.47,126.73) .. controls (144.93,135.33) and (128.93,160.95) .. (154.87,171.73) .. controls (180.8,182.52) and (277.71,166.72) .. (310.21,168.75) ;
\draw [color={rgb, 255:red, 65; green, 117; blue, 5 }  ,draw opacity=1 ][line width=1.5]    (278.94,160.08) .. controls (322.94,158.75) and (302.84,171.18) .. (340.51,168.84) ;
\draw  [color={rgb, 255:red, 65; green, 117; blue, 5 }  ,draw opacity=1 ][fill={rgb, 255:red, 65; green, 117; blue, 5 }  ,fill opacity=0.1 ] (315.93,152.53) .. controls (323.6,153.08) and (327,153.28) .. (335.93,152.53) .. controls (334.4,157.68) and (334,164.28) .. (335.93,172.53) .. controls (328.6,172.68) and (324.4,172.88) .. (315.93,172.53) .. controls (313.21,165.28) and (314.54,158.95) .. (315.93,152.53) -- cycle ;
\draw [color={rgb, 255:red, 65; green, 117; blue, 5 }  ,draw opacity=1 ][line width=1.5]    (310.26,159.28) .. controls (354.26,157.95) and (334.27,170.32) .. (371.94,167.99) ;
\draw  [color={rgb, 255:red, 65; green, 117; blue, 5 }  ,draw opacity=1 ][fill={rgb, 255:red, 65; green, 117; blue, 5 }  ,fill opacity=0.1 ] (347.73,152.33) .. controls (355.4,152.88) and (358.8,153.08) .. (367.73,152.33) .. controls (366.2,157.48) and (365.8,164.08) .. (367.73,172.33) .. controls (360.4,172.48) and (356.2,172.68) .. (347.73,172.33) .. controls (345.01,165.08) and (346.34,158.75) .. (347.73,152.33) -- cycle ;
\draw [color={rgb, 255:red, 65; green, 117; blue, 5 }  ,draw opacity=1 ][line width=1.5]    (342.06,159.08) .. controls (386.06,157.75) and (368.84,168.89) .. (406.51,166.56) ;
\draw  [color={rgb, 255:red, 65; green, 117; blue, 5 }  ,draw opacity=1 ][fill={rgb, 255:red, 65; green, 117; blue, 5 }  ,fill opacity=0.1 ] (384.53,151.73) .. controls (392.2,152.28) and (395.6,152.48) .. (404.53,151.73) .. controls (403,156.88) and (402.6,163.48) .. (404.53,171.73) .. controls (397.2,171.88) and (393,172.08) .. (384.53,171.73) .. controls (381.81,164.48) and (383.14,158.15) .. (384.53,151.73) -- cycle ;
\draw  [color={rgb, 255:red, 0; green, 97; blue, 212 }  ,draw opacity=1 ][fill={rgb, 255:red, 0; green, 97; blue, 212 }  ,fill opacity=0.1 ] (491.47,155.82) .. controls (496.74,158.23) and (502.87,159.22) .. (509.87,159.02) .. controls (510.47,164.62) and (509.77,174.23) .. (508.87,182.02) .. controls (501.47,182.02) and (497.11,182.35) .. (489.17,179.6) .. controls (487.47,172.22) and (489.27,161.42) .. (491.47,155.82) -- cycle ;
\draw [color={rgb, 255:red, 65; green, 117; blue, 5 }  ,draw opacity=1 ][line width=1.5]    (379.14,158.48) .. controls (423.14,157.15) and (480.4,174.35) .. (518.07,172.02) ;
\draw [color={rgb, 255:red, 0; green, 97; blue, 212 }  ,draw opacity=1 ][line width=1.5]    (473.01,162.54) .. controls (495.3,162.54) and (492.44,171.11) .. (518.44,168.25) .. controls (544.44,165.4) and (563.85,145.2) .. (547.83,125.75) .. controls (531.81,106.3) and (507.73,108.05) .. (488.33,115.25) .. controls (468.93,122.45) and (431.47,155.62) .. (494.27,161.02) .. controls (557.07,166.42) and (533.01,144.95) .. (525,147.55) ;
\draw  [draw opacity=0][fill={rgb, 255:red, 0; green, 0; blue, 0 }  ,fill opacity=1 ] (289.91,159.93) .. controls (289.91,158.99) and (290.67,158.24) .. (291.6,158.24) .. controls (292.54,158.24) and (293.3,158.99) .. (293.3,159.93) .. controls (293.3,160.87) and (292.54,161.63) .. (291.6,161.63) .. controls (290.67,161.63) and (289.91,160.87) .. (289.91,159.93) -- cycle ;
\draw  [draw opacity=0][fill={rgb, 255:red, 0; green, 0; blue, 0 }  ,fill opacity=1 ] (322.28,159.56) .. controls (322.28,158.62) and (323.04,157.86) .. (323.98,157.86) .. controls (324.91,157.86) and (325.67,158.62) .. (325.67,159.56) .. controls (325.67,160.49) and (324.91,161.25) .. (323.98,161.25) .. controls (323.04,161.25) and (322.28,160.49) .. (322.28,159.56) -- cycle ;
\draw  [draw opacity=0][fill={rgb, 255:red, 0; green, 0; blue, 0 }  ,fill opacity=1 ] (354.53,159.18) .. controls (354.53,158.24) and (355.29,157.49) .. (356.23,157.49) .. controls (357.16,157.49) and (357.92,158.24) .. (357.92,159.18) .. controls (357.92,160.12) and (357.16,160.88) .. (356.23,160.88) .. controls (355.29,160.88) and (354.53,160.12) .. (354.53,159.18) -- cycle ;
\draw  [draw opacity=0][fill={rgb, 255:red, 0; green, 0; blue, 0 }  ,fill opacity=1 ] (392.05,159.06) .. controls (392.05,158.12) and (392.81,157.36) .. (393.75,157.36) .. controls (394.68,157.36) and (395.44,158.12) .. (395.44,159.06) .. controls (395.44,159.99) and (394.68,160.75) .. (393.75,160.75) .. controls (392.81,160.75) and (392.05,159.99) .. (392.05,159.06) -- cycle ;
\draw [color={rgb, 255:red, 208; green, 2; blue, 27 }  ,draw opacity=1 ][line width=1.5]    (175.8,181.02) .. controls (216.2,174.22) and (226.6,188.42) .. (217.2,189.82) ;
\draw [color={rgb, 255:red, 0; green, 97; blue, 212 }  ,draw opacity=1 ][line width=1.5]    (528.38,192.75) .. controls (520.65,193.95) and (531.32,178.08) .. (483.05,174.88) ;
\draw  [draw opacity=0][fill={rgb, 255:red, 0; green, 0; blue, 0 }  ,fill opacity=1 ] (188.78,175.68) .. controls (188.78,174.74) and (189.54,173.99) .. (190.47,173.99) .. controls (191.41,173.99) and (192.17,174.74) .. (192.17,175.68) .. controls (192.17,176.62) and (191.41,177.38) .. (190.47,177.38) .. controls (189.54,177.38) and (188.78,176.62) .. (188.78,175.68) -- cycle ;
\draw  [draw opacity=0][fill={rgb, 255:red, 0; green, 0; blue, 0 }  ,fill opacity=1 ] (495.87,167.31) .. controls (495.87,166.37) and (496.63,165.61) .. (497.56,165.61) .. controls (498.5,165.61) and (499.26,166.37) .. (499.26,167.31) .. controls (499.26,168.24) and (498.5,169) .. (497.56,169) .. controls (496.63,169) and (495.87,168.24) .. (495.87,167.31) -- cycle ;
\draw [color={rgb, 255:red, 189; green, 16; blue, 224 }  ,draw opacity=1 ]   (284.29,100.33) .. controls (277.39,100.39) and (274.94,158.65) .. (277.87,163.32) .. controls (280.8,167.98) and (313.2,164.78) .. (313.07,168.65) .. controls (312.94,172.52) and (288.67,170.65) .. (282.8,171.98) .. controls (276.94,173.32) and (283.28,180.34) .. (284.29,180.33) ;

\draw (181.37,174.72) node [anchor=north east] [inner sep=0.75pt]  [color={rgb, 255:red, 208; green, 2; blue, 27 }  ,opacity=1 ,xscale=1.2,yscale=1.2]  {$R_{L}$};
\draw (490.88,171.4) node [anchor=north east] [inner sep=0.75pt]  [color={rgb, 255:red, 0; green, 97; blue, 212 }  ,opacity=1 ,xscale=1.2,yscale=1.2]  {$R_{R}$};
\draw (272.8,111.82) node [anchor=east] [inner sep=0.75pt]  [color={rgb, 255:red, 189; green, 16; blue, 224 }  ,opacity=1 ,xscale=1.2,yscale=1.2]  {$\gamma _{L}$};
\draw (421.86,110.9) node [anchor=west] [inner sep=0.75pt]  [color={rgb, 255:red, 189; green, 16; blue, 224 }  ,opacity=1 ,xscale=1.2,yscale=1.2]  {$\gamma _{R}$};
\draw (189.32,112.26) node [anchor=south] [inner sep=0.75pt]  [color={rgb, 255:red, 208; green, 2; blue, 27 }  ,opacity=1 ,xscale=1.2,yscale=1.2]  {$f( R_{L})$};
\draw (514.06,108.47) node [anchor=south] [inner sep=0.75pt]  [color={rgb, 255:red, 0; green, 97; blue, 212 }  ,opacity=1 ,xscale=1.2,yscale=1.2]  {$f( R_{R})$};
\draw (289.6,113.84) node [anchor=west] [inner sep=0.75pt]  [color={rgb, 255:red, 189; green, 16; blue, 224 }  ,opacity=1 ,xscale=1.2,yscale=1.2]  {$f( \gamma _{L})$};

\end{tikzpicture}
\vspace{-30pt}

\caption{The counterexamples of Example~\ref{ExFinal}: construction of $f_n$ for $n=4$. The black dots represent contractible fixed points. Top: first part of the proof; bottom: second part of the proof. Note that the second example gives a counterexample of Problem 2 of \cite{pollicott} (note that the set $\rho(f)$ of \cite{pollicott} is the \emph{punctual} rotation set): the rotation set of points has nonempty interior while the rotation vectors of periodic orbits are contained in the union of two planes.}\label{Fig14Alepablo}
\end{center}
\end{figure}

\begin{proof}
For the map $f\mapsto \rote(f)$, the counterexample is given by adapting \cite[Figure 14]{alepablo} (see Figure~\ref{Fig14Alepablo}, top): for $S$ a genus 2 surface, there exist a sequence $(f_n)_{n\ge 1}$ of $\Homeo_0(S)$ converging to $f\in\Homeo_0(f)$ and a nonempty interior set $\rho_0\subset H_1(S,\R)$ such that $\rot(f_n) \supset \rho_0$ for any $n\ge 1$ and $\rot(f)$ is included in the union of two 2-dimensional linear subspaces of $H_1(S,\R)$. 

More precisely, we fix two separating curves $\gamma_L$ and $\gamma_R$ of $S$ that are disjoint, and two rectangles $R_L$ and $R_R$ that are included in the two connected components of the complement of $\gamma_L\cup\gamma_R$ that are not annuli. For any $n$, pick $2n$ pairwise disjoint rectangles $R_1,\dots R_n,R'_1,\dots R'_n$ in the connected component of the complement of $\gamma_L\cup\gamma_R$ that is an annulus. We pick lifts $\wt R_L,\wt R_R, \wt R_1,\dots \wt R_n,\wt R'_1,\dots \wt R'_n$ of the rectangles to $\wt S$, and four deck transformations $T_1, T_2, T_3, T_4 \in G$ generating $\pi_1(S)$ such that the axes in $T_1$ and $T_2$ on $S$ are included in the connected component of the complement of $\gamma_L\cup\gamma_R$ containing $R_L$, and the axes in $T_3$ and $T_4$ on $S$ are included in the connected component of the complement of $\gamma_L\cup\gamma_R$ containing $R_R$. 

We then pick, for any $n\ge 1$, an homeomorphism $f_n\in \Homeo_0(S)$ such that the following intersections are Markovian:
\begin{itemize}
\item $\wt f(\wt R_L)$ and $\wt R_L$, $\wt f(\wt R_L)$ and $T_1 \wt R_L$, $\wt f(\wt R_L)$ and $T_2T_1 \wt R_L$ and $\wt f(\wt R_L)$ and $\wt R_1$;
\item for any $1\le i <n$, $\wt f(\wt R_i)$ and $\wt R_{i+1}$, and $\wt f(\wt R_n)$ and $\wt R_R$;
\item $\wt f(\wt R_R)$ and $\wt R_R$, $\wt f(\wt R_R)$ and $T_3 \wt R_R$, $\wt f(\wt R_R)$ and $T_4T_3 \wt R_R$ and $\wt f(\wt R_R)$ and $\wt R'_1$;
\item for any $1\le i <n$, $\wt f(\wt R'_i)$ and $\wt R'_{i+1}$, and $\wt f(\wt R'_n)$ and $\wt R_L$;
\end{itemize}
We moreover suppose that $\wt f_n^{n}(\gamma_L)\cap\gamma_R = \wt f_n^{n}(\gamma_R)\cap\gamma_L = \emptyset$. Finally, we suppose that the sequence $(f_n)_n$ converges towards $f\in\Homeo_0(S)$ such that the restriction of $f$ to the connected component of the complement of $\gamma_L\cup\gamma_R$ that is an annulus is the identity. Such a construction can be made by hand by using time-dependant vector fields. 

We see that for any $n\ge 1$, we have 
\[\conv\big([T_1], [T_1]+[T_2], [T_3], [T_3]+[T_4]\big)  \subset \rote(f_n),\]
while
\[\rot(f) \subset \operatorname{span}\big([T_1], [T_2]\big) \cup \operatorname{span}\big([T_4], [T_4]\big)\]
(in fact, $\rot(f)$ has to contain the union of two convex sets of dimension 2). 
\bigskip

For the example concerning $\rot(f)$, it suffices to adapt the previous example in the flavour of \cite[Figure 1]{guiheneuf2023hyperbolic} (see Figure~\ref{Fig14Alepablo}, bottom): we do not ask for the properties about the $R'_i$, but require instead that $f(\gamma_L)\cap \gamma_L = \emptyset$. 
\end{proof}

\small

\bibliographystyle{alpha}
\bibliography{Biblio}

\end{document}